\documentclass[letterpaper,11pt,reqno]{amsart}
\RequirePackage[utf8]{inputenc}
\usepackage[portrait,margin=2.5cm]{geometry}
\usepackage{amsmath,mathtools}
\usepackage{mathrsfs}
\usepackage[foot]{amsaddr}
\usepackage{amssymb,amsthm,amsfonts,amsbsy,latexsym,dsfont,bbm}
\usepackage[textsize=small]{todonotes}       
\usepackage{graphicx}
\usepackage[numeric,initials,nobysame]{amsrefs}
\usepackage{upref,setspace}
\usepackage{yhmath}
\usepackage{mathrsfs}
\usepackage{verbatim}
\usepackage[colorlinks, linkcolor=blue, citecolor=violet, pdfborder={0 1 0}]{hyperref}
\usepackage[foot]{amsaddr}
\usepackage{amssymb,amsthm,amsfonts,amsbsy,latexsym,dsfont}
\usepackage[textsize=small]{todonotes}       
\usepackage{graphicx}
\usepackage[numeric,initials,nobysame]{amsrefs}
\usepackage{upref,setspace}
\usepackage{xcolor,colortbl}
\usepackage{enumerate}
\usepackage{graphicx}
\usepackage{subcaption}   
\usepackage{bm}
\usepackage{tikz}
\usepackage{bigints}
\mathtoolsset{showonlyrefs}

\usetikzlibrary{calc,intersections,through,backgrounds,shapes.geometric}
\usetikzlibrary{graphs}
\tikzset{every path/.style={line width=.07 cm}}
\usepackage[normalem]{ulem}

\allowdisplaybreaks

\definecolor{darkblue}{rgb}{0,0,0.6}

\newenvironment{enumeratea}{\begin{enumerate}[\upshape (a)]}{\end{enumerate}}

\usepackage{bbm}

\usepackage{paralist}

\numberwithin{equation}{section}
\numberwithin{figure}{section}
\numberwithin{table}{section}

\newtheorem{thm}{Theorem}[section]
\newtheorem{lem}[thm]{Lemma}

\newtheorem{cor}[thm]{Corollary}

\newtheorem{prop}[thm]{Proposition}
\newtheorem{defn}[thm]{Definition}

\newtheorem{assump}[thm]{Assumption}

\theoremstyle{definition}
\newtheorem{rem}[thm]{Remark}
\newtheorem{example}[thm]{Example}

\newcommand{\mycomment}[1]{}

\def\qed{ \hfill $\blacksquare$}

\newcommand{\cF}{\mathcal{F}}
\newcommand{\cH}{\mathcal{H}}\newcommand{\cI}{\mathcal{I}}
\newcommand{\cL}{\mathcal{L}}

\newcommand{\cQ}{\mathcal{Q}}
\newcommand{\cS}{\mathcal{S}}\newcommand{\cT}{\mathcal{T}}
\newcommand{\cW}{\mathcal{W}}

\newcommand{\vn}{\mathbf{n}}

\newcommand{\ft}{\mathfrak{t}}

\newcommand{\bP}{\mathbb{P}}

\newcommand{\sS}{\mathscr{S}}
\newcommand{\sT}{\mathscr{T}}

\newcommand{\bcdot}{\boldsymbol{\cdot}}

\DeclareMathOperator{\E}{\mathbb{E}}

\DeclareMathOperator{\argmin}{argmin}

\newcommand{\erdos}{Erd\H{o}s-R\'enyi}

\def\RR{\mathbb{R}}

\usepackage{pdfsync}

\definecolor{aqua}{rgb}{0.0, 1.0, 1.0}
\definecolor{boo}{rgb}{1.0, 0.0, 1.0}
\definecolor{stred}{rgb}{1.0, 0.24, 0.67}

\usepackage{accents}

\DeclareMathAlphabet\mathbfcal{OMS}{cmsy}{b}{n}

\usepackage{accents}

\def\beq{ \begin{equation} }
\def\eeq{ \end{equation} }

\usepackage{tgpagella,eulervm}

 \newcommand{\NN}{\mathbb{N}}

 \newcommand{\clu}{\mathcal{U}}

 \newcommand{\clv}{\mathcal{V}}
 \newcommand{\clx}{\mathcal{X}}

 \newcommand{\clp}{\mathcal{P}}

 \newcommand{\cln}{\mathcal{N}}

 \newcommand{\cla}{\mathcal{A}}
 \newcommand{\clb}{\mathcal{B}}
  \newcommand{\cld}{\mathcal{D}}
 \newcommand{\cls}{\mathcal{S}}

\begin{document}

\title[LDP for graphon processes and  associated dynamical systems]{Large Deviations for Markovian Graphon Processes and Associated Dynamical Systems on Networks }

\date{}

\author{Shankar Bhamidi${}^1$}
\thanks{${}^1$Department of Statistics and Operations Research, University of North Carolina, Chapel Hill, NC 27599.}
\email{bhamidi@email.unc.edu}

\author{Amarjit Budhiraja${}^1$}
\email{budhiraj@email.unc.edu}

\author{Souvik Ray${}^2$}
\thanks{${}^2$Department of Mathematics, Iowa State University, Ames, IA 50011}
\email{souvikr@iastate.edu}


\begin{abstract}
We consider temporal models of rapidly evolving Markovian networks whose edge-formation and dissolution rates are determined by time-dependent spatial kernels. Equivalently, these may be viewed as Markovian networks with $O(1)$ jump rates observed over long time horizons. In this regime, paths of graphon-valued processes obtained by averaging over suitable moving time windows provide natural state descriptors. Under appropriate conditions on the jump-rate kernels, we establish laws of large numbers and large deviation principles for these window-averaged paths, both in the weak topology and in the cut metric. We also show that, without such local averaging, the rapidly oscillating graphon process does not satisfy a nontrivial path-space LDP. The resulting rate functions admit explicit and tractable representations, distinct from those arising in static random graph models and finite-horizon dynamic graph models. We further analyze the associated variational problems in several examples and apply the graphon LDP to node-valent dynamical systems driven by the evolving network.

\end{abstract}

\maketitle

\section{Introduction}

The main object of interest in this paper is a model of dynamically evolving (directed) simple random graphs and the large deviation principles of certain objects associated with this random graph process. Driven by the rapid growth of available data on real-world systems and their influence on science and society, network models have become essential tools to represent these systems and to understand the dynamics they exhibit, including processes such as epidemics and search algorithms operating within them \cite{newman2003structure,newman2010networks,MR2271734,MR3617364,van2024random,albert2002statistical}. Within this vast research discipline, \emph{dynamic} network models as well as node-level processes modulated through these time-varying network structures have become increasingly important \cite{holme2012temporal,boccaletti2023structure,bassett2017network,berner2023adaptive,li2017fundamental}. These include classical work in the context of mean-field (complete graph) interactions \cite{Kolokoltsov2010,Sznitman1991} to more general network structure, see e.g. \cite{bhamidi2019weakly,bayraktar2023graphon,bet2024weakly} and the references therein for starting points when the underlying graph is \emph{dense} and \cite{lacker2023local,lacker2021locally} as starting points for the sparse regime. Insights from these works have been applied in a range of domain sciences, from network neuroscience and the dynamics of spiking neurons \cite{BaladronFaugeras2012,FaugerasTouboulCessac2009constructive} to the study of swarming and flocking 
behaviors \cite{bertozzi1,bertozzi2}.

 The main goal of this paper is to study large deviations for a class of dense dynamic random graph models, and node-valent processes modulated by these evolving networks. Here, a graph is called dense if its average vertex degree is of the same order as the number of vertices in the system.  We shall view the directed random graphs as \emph{graphons}, see \cite{Lovasz2006, Lovasz1, Lovasz2, Lovasz3, Lovasz}  by Lov\'asz and his co-authors for detailed account of this theory.  
 
The study of large deviations for dense random graphs was initiated in an important work of \cite{Chatterjee2011} which combined ideas from classical large deviations theory and tools from combinatorics and graph theory, such as the famous \emph{Szemer\'edi's Regularity Lemma} and the graphon limit theory \cite{Lovasz}, to prove a large deviation principle for homogeneous Erd\H{o}s-R\'enyi random graphs. Following this work, there have been quite a few advances in this area,  including the extensions  to  uniform Erd\H{o}s-R\'enyi random graphs~\cite{Dembo2018}, inhomogeneous random graphs~\cite{Dhara2022, Markering2023, Dupuis2022}, uniform random graphs with given degree sequence~\cite{Dhara2022}, and node-valent processes with interactions governed by inhomogeneous random graphs\cite{Dupuis2022},  see \cite{Chatterjee2017} for a non-exhaustive list of these developments. Conceptually, in addition to the insight such large deviation principles provide into extremal behavior of the processes in question and the mathematical techniques they necessitate the development of, there are two additional reasons for studying such behavior:

\begin{enumeratea}
 \item {\bf To understand more complex systems via appropriate tilts:} Large deviation results for the classical Erd\H{o}s-R\'enyi  random graphs \cite{Chatterjee2011} have been instrumental in providing a mathematical understanding and the statistical estimation of other core random graph models such as the exponential random graphs (ERGM)~\cite{chatterjee2013estimating}, which can be expressed as appropriate tilts of the uniform (\erdos) measure.
   \item {\bf Understanding extremal events in the evolution of temporal motifs and node-valent network processes:} This includes using appropriately defined temporal network motifs to detect anomalies in networks \cite{paranjape2017motifs,kovanen2013temporal,jiang2016catching,eswaran2018spotlight}, or understanding the effect of network noise on inducing rare events on excitable systems (e.g., neuron level signals) with interactions governed by such networks \cite{bressloff2014path,avitabile2024neural}.
\end{enumeratea}

Most of the existing work on large deviations focuses on \emph{static} random graph models. In many applications, the connectivity structure changes on the same time
scale as the dynamics it supports.  Contacts in social, biological, and
technological systems may appear and disappear over time, and the temporal
ordering and persistence of these interactions can substantially alter
spreading and other network-driven dynamics relative to what is predicted by
a static aggregate graph \cite{holme2012temporal,li2017fundamental}.  This distinction
is particularly important in epidemic models in which susceptible individuals
may break and rewire potentially infectious contacts, since such adaptation
can change the nature of the epidemic phase transition
\cite{DurrettYao2022}.  In longitudinal social networks, retaining temporal
information makes it possible to study network stability, the emergence of
groups, and the relationship between individual stability and popularity
\cite{SewellChen2015}.  Analyses of evolving interaction networks constructed
from Twitter and Reddit data have likewise used their temporal structure to
separate attraction toward one's own political group from disengagement with
the opposing group \cite{ZhuEtAl2023}.  In adaptive systems, network evolution
may also be coupled to the node states; for example, in social and opinion
networks, agents' beliefs can influence their subsequent interactions, which
in turn shape the evolution of those beliefs
\cite{berner2023adaptive,GangulySpiliopoulosSussman2025}.

The first work, to our knowledge, on large deviations of a dense dynamic random graph model appeared in~\cite{Braunsteins2023}. The basic model of \cite{Braunsteins2023} for a dynamic random graph is described as follows. Starting from some initial graph, each edge switches on/off at certain pre-specified rates, independently of other edges. This dynamical model was introduced in~\cite{Zhang2017} in the context of statistical inference for longitudinal network data and a sparse version of this model was analyzed by \cite{Roberts2018} for the critical Erd\H{o}s-R\'enyi random graph focusing on the largest component size of the graph over time. The paper \cite{Braunsteins2023} establishes  large deviation results, in the dense regime, for the sample path of the  graphon process  observed until a fixed time horizon in the setting where the edge-flipping rates are space and time-homogeneous.

In the current work, we study a setting where the edge-flipping rate is both time and space dependent (i.e., may depend on the edge endpoints) and which becomes large with system size. One can regard this as a system with rapidly changing edge connections or, by time scaling, a system evolving over a long time horizon with $O(1)$ rate of edge formation and dissolution.  
In the rapidly evolving regime considered here, the graphon process exhibits persistent microscopic oscillations.  We first establish a large deviation principle for the full space--time graphon in the weak topology, with an explicit rate function obtained by integrating the Donsker--Varadhan occupation cost of the underlying two-state edge processes; see Theorem~\ref{thm:weakLDP}.  We also prove that, the graphon trajectory admits no nontrivial sample-path LDP, either in the weak topology or in the cut-metric topology, because the diverging edge-flipping rates destroy the required path-space compactness; see Propositions~\ref{prop:no-ldp-weak} and~\ref{prop:no-ldp-cut}.  Local averaging over a small time window removes these rapid oscillations and yields sample-path LDPs in both topologies, culminating in the cut-metric result of Theorem~\ref{thm:LDP-avg-path}; see also Corollary~\ref{cor:weakLDP-path}.  By appealing to standard results, we also establish large deviation principles for the graphons observed at a finite collection of fixed time instants. Interestingly, these finite-dimensional LDPs occur at a different speed from the averaged path-space LDPs and involve a fundamentally different local rate function, reflecting instantaneous fluctuations rather than long-time occupation behavior; see Proposition~\ref{prop:ldp-point-weak} and Theorem~\ref{thm:ldp-point-cut}.

These results describe a regime that is different both from static dense random graphs \cite{Chatterjee2011, Dupuis2022, Dhara2022} and from the dynamic setting of \cite{Braunsteins2023}.  In the static setting, the local cost is a Bernoulli relative entropy, while in the dynamic setting  the rate function is an action integral over space and time of the local Lagrangian for a two-state jump Markov process.  Here, the fast dynamics leads instead to a stationary occupation-measure cost of Donsker--Varadhan type.  In Section~\ref{sec:apps}, we use the resulting variational formulas to identify the most likely mechanisms for several rare events, including atypical local edge densities, prescribed terminal graphons, and symmetry-breaking transitions for subgraph counts.

As mentioned earlier, node-level processes with spatially structured interactions governed by underlying graphs have been of immense interest and arise from real-world applications. Continuum limits describing the law of large number (LLN) behavior for interacting dynamical systems on (static) convergent dense graph sequences were formulated and established in a series of works~\cite{medvedev2014nonlinear, medvedev2014W,medvedev2019continuum} using tools from graphon limit theory. Large deviation results for some such interacting dynamical systems on static inhomogeneous graphs have been studied in \cite{Dupuis2022}. In the current work, we study dynamical systems with interactions governed by a fast-changing network structure of the form described in the previous paragraph.  The law of large numbers for such interacting dynamical systems describes an averaging principle, giving a continuum space-time field as a suitable limit of the spatially interpolated state processes. Our work establishes a large deviation principle characterizing probabilities of deviation from this LLN behavior.

\subsection{Organization of the paper}
We give a precise formulation of the evolving network model and its connection to dense graph limits in  Section \ref{sec:mod}. All our main results are described next in Section \ref{obj}. These two sections also discuss related work, comment on proof approaches, and place the contributions of this paper within the current body of research. See in particular Remarks \ref{rem:2.5}-\ref{rem:2.6},  \ref{weakandcut}-\ref{rem:homldp}, \ref{rem:3.21}-\ref{rem:3.22}. The remaining sections contain proofs of the main results. A detailed description of the proof organization is given in Section \ref{sec:orgpf}.

	\section{Setting} 
    \label{sec:mod}
 We start by describing the dynamically evolving random graph model of our interest. For $n \in \mathbb{N}$ and $i,j \in [n] := \{1, \ldots , n\}$, 
let 
$$Q_1^n := [0,1/n], Q_{i}^n := ((i-1)/n, i/n], \text{ for } i \neq 1, \text{ and } \;\; Q^n_{i,j}:=Q_{i}^n \times Q_{j}^n.$$  
	We shall work with the graphon representations of the random graphs. We note that since we work with directed graphs, the graphons in our study will not necessarily be symmetric.
For a graph on $n$ vertices labeled $1, \ldots,n$, vertex $i$ of the graph will correspond to the interval $Q_i^n$ in its graphon representation. Graphons are formally defined as (Lebesgue) measurable functions on $[0,1]^2$ taking values in $[0,1]$. We denote the set of all graphons by $\cls_0$ : 	
	$$ \cls_0 := \left\{f:[0,1]^2 \to \RR \text{ measurable } \mid \, 0 \leq f \leq 1 \text{ a.e.} \right\}. $$
	For a simple directed graph $G$ on $n$ vertices, the graphon representation of $G$ is the block function in $\cls_0$ which assigns, for each $i,j \in [n]$, the value $1$ to all the points in $Q_{i,j}^n$ if there is an edge connecting vertex $i$ to vertex $j$ in $G$  and assigns the value $0$ otherwise. In this paper, we shall deal with graphon processes, indexed by time. Fix some $T \in (0,\infty)$. Hereafter, $\cW_0$ will denote the set of all graphon processes over the time interval $[0,T]$ :
    $$ \cW _0 : = \left\{\phi : [0,1]^2 \times [0,T] \to \RR \text{ measurable } \mid 0 \leq \phi \leq 1 \text{ a.e.} \right\}. $$
	In the last two displays, a.e. refers to the Lebesgue measure on $[0,1]^2$ and $[0,1]^2 \times [0,T]$, respectively. We will often use $d\lambda$ when integrating with respect to the Lebesgue measure on spaces such as $[0,1]^d$ or $[0,1]^d \times [0,T]$.  In general we will omit reference to the measure when clear from the context. All $L^p$-spaces on  $[0,1]^d$ or $[0,1]^d \times [0,T]$  in this paper are with respect to the corresponding Lebesgue measures.

	Let $\cW$ be the set of all time-dependent kernels on time interval $[0,T]$, i.e., 
	$$ \cW := \left\{\phi \, \mid \, \phi:[0,1]^2 \times [0,T] \to [0,\infty) \text{ measurable}\right\}. $$
    For any $\phi \in \cW$ and time $s \in [0,T]$, we shall use the notation $\phi_s := \phi(\cdot,\cdot,s) \in \cls$ to denote the kernel at time $s$. Here, $\cls$ is the set of all time-homogeneous kernels, i.e., 
	$$ \cls := \left\{\phi \, \mid \, \phi:[0,1]^2 \to [0,\infty) \text{ measurable}\right\}. $$
	Fix $\beta^+, \beta^{-} \in \cW$. These two kernels $\beta^{\pm}$ will control the rate of change for the dynamical model in the following sense : $\beta^+(x,y,t)=\beta^+_t(x,y)$ and $\beta^-(x,y,t)=\beta^-_t(x,y)$ dictate the rate at which the directed edge connecting vertex $x \in [0,1]$ to vertex $y \in [0,1]$ switches from being absent to present and present to absent, respectively, at time $t \in [0,T]$. To make this idea precise, we define a sequence of block kernels of the form
	\begin{equation}{\label{def:beta}}
		\beta^{n,\pm}_t(x,y) \equiv \beta^{n, \pm} (x,y,t):= \sum_{i,j=1}^n \beta^{n,\pm}_{i,j}(t) \mathbbm{1}_{Q^n_{i,j}}(x,y), \forall \; x,y \in [0,1],
	\end{equation} 
	where $\beta^{n,\pm}_{i,j} : [0,T] \to [0,\infty)$ are measurable functions. 
	We will make the following assumption which connects these block kernels to the original kernels $\beta^{\pm}$. This will be a standing assumption in this work and will not be explicitly noted in the statements of various results. We shall also write $[0,1]^2_T$ as an abbreviation for $[0,1]^2 \times [0,T]$. In general, we shall use the notation $B_T$ to denote $B \times [0,T]$ for any set $B$.
	\begin{assump}{\label{assmp:1}}
		\begin{enumerate}[(a)] 
			\item {\label{assump:1Aweak}} We have $\beta^{n,\pm}, \beta^{\pm} \in L^1 \left([0,1]^2_T\right)$ for all $n \geq 1$ and $\beta^{n,\pm}$ converges in $L^1\left([0,1]^2_T\right)$ to $\beta^{\pm}$ as $n \to \infty$.
			\item {\label{assump:1B}} There exists $c_{\beta} >0$ such that $\beta^{n,\pm}_{i,j}(t) \geq c_{\beta}$, for all $i,j \in [n], n \geq 1$ and $t \in [0,T]$. 
			\item {\label{assump:1D}}The functions $t \mapsto \beta^{\pm}_t(x,y)$ are bounded and left-continuous in $t$ for almost all $(x,y) \in [0,1]^2$.  Further, $t \mapsto \beta^{n,\pm}_{i,j}(t)$ is bounded and left-continuous for all $i,j \in [n]$ and large enough $n$.
		\end{enumerate} 
	\end{assump}
	The assumption in \ref{assump:1Aweak} ensures that for any subsequence, there is a further subsequence along which $\beta^{n,\pm}$ converges almost everywhere to $\beta^{\pm}$ on $[0,1]^2_T$ which, in particular,  as a consequence of \ref{assump:1B}, also guarantees that $\beta^{\pm} \geq c_{\beta}$, almost everywhere on $[0,1]^2_T$. We will need one of the following two stronger versions of \ref{assump:1Aweak} in some of the results.
	\begin{assump}{\label{assmp:2}}
		\begin{enumerate}[(i)] 
			\item {\label{assump:1A}} For some $\eta >0$,  we have $\beta^{n,\pm}, \beta^{\pm} \in L^{1+\eta}\left([0,1]^2_T\right)$ for all $n \geq 1$ and $\beta^{n,\pm}$ converges in $L^{1+\eta}$ to $\beta^{\pm}$ as $n \to \infty$. 
			\item {\label{assump:1AB}} We have $\sup_{t \in [0,T]} \beta_t^{\pm} \in L^1 \left( [0,1]^2 \right)$ and as $n\to \infty$,
			\begin{equation}{\label{assump:1ABeq}}
				\int_{[0,1]^2} \sup_{ t\in [0,T]} \big \rvert \beta^{n,\pm}_t(x,y)- \beta^{\pm}_t(x,y)  \big \rvert \,dx \,dy \longrightarrow 0.
			\end{equation}
		\end{enumerate}
	\end{assump}

We will shortly give a canonical example of parameter specifications that satisfies the above assumptions. First, let us describe the main model studied in this paper. 
	\begin{defn}[Markovian dynamic directed random graph]{\label{defn:rate}}
		Let $a(n) \in (0,\infty)$ such that $a(n)\to \infty$ as $n\to \infty$.
		Let $\{X^n_{i,j}(t), t \in [0,T]\}_{i,j \in [n]}$ be a collection of $n^2$ mutually independent Markov chains with state space $\mathcal{I} := \{0,1\}$, defined on the probability space $(\Omega, \mathcal{F}, \mathbb{P})$,  $X^n_{i,j}(0) = x^n_{i,j} \in \left\{0,1\right\}$ for 
		all $i,j \in [n]$, and (time-inhomogeneous) rate matrices $R^n_{i,j}$ defined as
		\begin{align*}
			R^n_{i,j}(0,0;t) & := -a(n)\beta_{i,j}^{n,+}(t), \;
			R^n_{i,j}(0,1;t) := a(n)\beta_{i,j}^{n,+}(t), \\ 
			R_{i,j}^n(1,0;t)& := a(n)\beta_{i,j}^{n,-}(t), \; 	R_{i,j}^n(1,1;t) := -a(n)\beta_{i,j}^{n,-}(t).
		\end{align*}
		The (directed) random graph at time $t \in [0,T]$, denoted by $G^n(t)$, is defined on the vertex set $[n]$ and has the edge which connects $i$ to $j$ (for $i,j \in [n]$) if and only if $X_{i,j}^n(t)=1$.
	\end{defn}
	Thus, at time $t$, if there is an edge connecting vertex $i$ to vertex $j$ then this edge gets disconnected at rate $a(n)\beta_{i,j}^{n,-}(t)$ while if there is no edge $i\leadsto j$ then a new edge is created at rate $a(n) \beta_{i,j}^{n,+}(t)$.   The focus of this paper is the study of the large deviations for certain objects associated with the process $\left\{G^n(t) : 0 \leq t \leq T\right\}$.

	\begin{example}{\label{ex:block}}
		One canonical example of  a block kernel sequence satisfying Assumptions~\ref{assmp:1} and \ref{assmp:2} corresponds to a setting where the sequence is the block approximations of the functions $\beta^{\pm}$ itself, namely, 
		$$	\beta^{n, \pm}_{i,j} := n^2 \int_{Q^n_{i,j}} \beta^{\pm} d\lambda, \;\; \forall \; \; i,j\in [n].$$
		Suppose that for some $c_{\beta} >0$, $\beta^{\pm} \geq c_{\beta}$, a.e. on $[0,1]^2_T$ and the functions $t \mapsto \beta^{\pm}_t(x,y)$ are left-continuous and bounded in $t$ for almost all $(x,y) \in [0,1]^2$. Further suppose that $\beta^{\pm} \in L^{1+\varepsilon}$ for some $\varepsilon>0$ and
		\begin{equation}{\label{assump:betatimebound}}
			\int_{[0,1]^2} \sup_{t \in [0,T]} \beta^{\pm}_t(x,y) \, dx \,dy < \infty.
		\end{equation}
		Then Assumption \ref{assmp:1} is satisfied. Indeed, 
		by the \textit{Lebesgue differentiation theorem}, $\beta^{n,\pm}$ converges almost everywhere on $[0,1]^2_T$ to $\beta^{\pm}$. Moreover,  by \textit{Jensen's inequality},
		$$ \int_0^T \int_{[0,1]^2} \left(\beta^{n,\pm}\right)^{1+\varepsilon} = \int_0^T \sum_{i,j \in [n]} \dfrac{1}{n^2} \left( n^2 \int_{Q^n_{i,j}} \beta^{\pm}\right)^{1+\varepsilon}  \leq   \int_0^T \int_{[0,1]^2} \left(\beta^{\pm}\right)^{1+\varepsilon}.$$  
		Therefore, since $\beta^{\pm} \in L^{1+\varepsilon}$, we have  that $\beta^{n,\pm}$ converges in $L^{1+\eta}$ to $\beta^{\pm}$ for any $\eta \in (0,\varepsilon).$
		Thus Assumption \ref{assmp:2}\eqref{assump:1A} and consequently Assumption \ref{assmp:1}\eqref{assump:1Aweak} holds.
		Also,
		\begin{equation}{\label{eq:unif-t-bound-ex}}
			\sup_{t \in [0,T]} \beta^{n,\pm}_{i,j}(t) = \sup_{t \in [0,T]}  n^2 \int_{Q^n_{i,j}} \beta_t^{\pm}(x,y)\,dx\,dy \leq  n^2 \int_{Q^n_{i,j}} \sup_{t \in [0,T]}  \beta_t^{\pm}(x,y)\,dx\,dy < \infty.
		\end{equation}
		Moreover, the condition in \eqref{assump:betatimebound}, along with the fact that $t \mapsto \beta^{\pm}_t(x,y)$ is left-continuous for almost all $(x,y)$, guarantees that $t \mapsto \beta^{n,\pm}_{i,j}(t)$ is left-continuous by a simple application of dominated convergence.  We, therefore, have that Assumption \ref{assmp:1}\eqref{assump:1D} holds as well. Part \ref{assump:1B} of the assumption is immediate from the assumed lower bound on $\beta^{\pm}$. 
		
		Finally, if  we further assume that $s \mapsto \beta^{\pm}(\cdot,\cdot,s)$ is uniformly Lipschitz, i.e., there exists $C \in (0,\infty)$ such that 
		\begin{equation}{\label{assump:betalip}}
			\Big \rvert \beta^{\pm}(x,y,t) - \beta^{\pm}(x,y,s) \Big \rvert \leq C|t-s|, \; \forall \; s,t \in [0,T], \mbox{ and } (x,y) \in [0,1]^2,
		\end{equation}
		then Assumption \ref{assmp:2}\eqref{assump:1AB} holds as well. Indeed, we can find a countable dense set $S =\left\{s_1,s_2,\ldots\right\} \subset [0,T]$ such that $\beta^{\pm}_s \in L^{1+\varepsilon}\left([0,1]^2 \right)$ for all $s \in S$ and therefore $\beta^{n,\pm}_s$ converges in $L^1$ to $\beta^{\pm}_s$.  Note that $\delta_m := \sup_{s \in [0,T]} 
        \inf_{1\le k\le m} |s-s_k| \to 0$ as $m \to \infty$. Hence,
		\begin{align*}
			&\int_{[0,1]^2} \sup_{ s\in [0,T]} \big \rvert \beta^{n,\pm}_s(x,y)- \beta^{\pm}_s(x,y)  \big \rvert \,dx \,dy \\
            & \leq  \int_{[0,1]^2} \max_{1\le k\le m} \big \rvert \beta^{n,\pm}_{s_k}(x,y)- \beta^{\pm}_{s_k}(x,y)  \big \rvert \,dx \,dy + 2C\delta_m \\
			& \leq \sum_{k=1}^m \int_{[0,1]^2}  \big \rvert \beta^{n,\pm}_{s_k}(x,y)- \beta^{\pm}_{s_k}(x,y)  \big \rvert \,dx \,dy + 2C\delta_m.
		\end{align*}
		Taking $n$ and then $m$ to infinity, we have that Assumption \ref{assmp:2}\eqref{assump:1AB} holds. 
		
		In the time-homogeneous case (i.e., $\beta^{\pm}(t,x,y)$ does not depend on time $t$), \eqref{assump:betalip} holds trivially and \eqref{assump:betatimebound} is equivalent to the integrability of $\beta^{\pm}$. 
	\end{example}

    \begin{rem}
    \label{rem:2.5}
As noted in the introduction, for the time and space homogeneous setting, i.e., when $\beta^{\pm}(t,x,y) \equiv \alpha^{\pm}$ for some $\alpha^{\pm}>0$, and $a(n)=1$, the dynamic random graph model of Definition \ref{defn:rate} has been studied in \cite{Braunsteins2023}. A similar dynamic random graph in the context of sparse critical Erd\H{o}s–Rényi random graphs was previously considered in \cite{Roberts2018}. Apart from the minor detail that our model is defined for simple directed graphs instead of undirected graphs as in \cite{Braunsteins2023}, our model differs from that of \cite{Braunsteins2023} in two crucial aspects. First, we allow for very general forms of space and time inhomogeneities in the jump rates, as specified in Assumptions \ref{assmp:1} and \ref{assmp:2}. Indeed, in the setting of \cite{Braunsteins2023} the stationary distribution of the Markovian dynamics  is the Erd\H{o}s-R\'enyi random graph, whereas for the model considered in the current work, the long-time behavior  is given by a very general inhomogeneous graphon (see Theorems \ref{cor2:thmlin} and \ref{thm:llnnew}). 
In particular, our setting covers dynamical
stochastic block models.
Second, in the setting we consider, the graphs change very rapidly, a feature that marks a change in the regime of our analysis compared to \cite{Braunsteins2023}. This feature of rapid change of network structure says that from an observer's perspective the natural statistics to analyze are those obtained by a suitable time window averaging. This leads us to the analysis of the asymptotic behavior of the window-averaged graphon process as described in the next section.

\end{rem}

\begin{rem}
\label{rem:2.6}
The  model we consider describes a particularly simple dynamical behavior in that the updates for different edges are independent of one another. There have been more complex models for temporal evolution studied in the literature. E.g., Markov process models for network dynamics have been proposed in~\cite{Holland1977}. A continuous-time model of network dynamics was considered in~\cite{Snijders2006StatisticalMF}, where each node updates its neighbors by optimizing an objective function based on some local neighborhood statistics. A discrete-time dynamical model on exponential random graphs was introduced in~\cite{Robins2001, Robins2005} and its statistical properties were analyzed later in~\cite{Hanneke2010}, see~\cite{Lee2020} for a recent work on a model-based clustering framework for this model. A concise discussion on generalizations of configuration model and stochastic block model to dynamic networks can be found in~\cite{Zhang2017}. All of these works focus on modeling real-life longitudinal network data with dynamical random graph models and statistical inference based on them. The focus of the current work is on the study of the large deviation behavior. Akin to the LDP results for the \erdos~random graph in \cite{Chatterjee2011} leading to the understanding of ERGMs in \cite{chatterjee2013estimating}, a natural next step is to leverage the techniques developed in this paper to analyze more complex network processes. 
\end{rem}

\section{{Main Results}}	{\label{obj}}

In this section, we present our main results along with accompanying discussions on related work. Law of large numbers and large deviation results  are the main focus of this section. Before stating the main results, we introduce a few notations. For any metric space $(S,d_S)$, we shall denote by $\mathbb{D}\left([0,T]: (S,d_S)\right)$ (resp.~ $\mathbb{C}\left([0,T]: (S,d_S)\right)$) the space of all c\`{a}dl\`{a}g (resp.~continuous) paths on the time interval $[0,T]$ taking values in the space $S$. We shall equip this space with the Skorokhod topology (resp.~topology of uniform convergence). We write $ \text{Osc}(\cdot, \cdot)$ and $ \text{Osc}^{\prime}(\cdot, \cdot)$ for the modulus of continuity and $D$-modulus of continuity, respectively. In other words, for any $f : [0,T]\to S$ and $\delta >0$, 
 $$ \text{Osc}(f, \delta; (S,d_S)) := \sup_{\substack{s,t \in [0,T] \\ |t-s|\leq \delta}} d_S(f(s),f(t)),$$
 and  
$$ \text{Osc}^{\prime}(f, \delta; (S,d_S)) := \inf_{\left\{t_i\right\}} \max_{1 \leq i \leq v} \sup_{s,t \in [t_{i-1},t_i)} d_S(f(s),f(t)),$$ 
where the infimum is taken over all $0 \leq t_0 < t_1 < \cdots < t_v =1$ satisfying the sparsity condition $\min_{1 \leq i \leq v} (t_i-t_{i-1}) > \delta$. For any $f \in \mathbb{D}\left([0,T]: (S,d_S)\right)$, the largest jump of $f$ is denoted by 
$$ \text{jump}\left(f;(S,d_S)\right) := \sup_{t \in [0,T]} d_S(f(t),f(t-)).$$
We have the following useful inequality, see~\cite[Eq. (12.9)]{billingsley1999convergence}. For any $f \in \mathbb{D}\left([0,T]: (S,d_S)\right)$ and $\delta >0$, 
 \begin{equation}{\label{eq:ineq-osc}}
 	\text{Osc}(f, \delta; (S,d_S)) \leq 2 \times \text{Osc}^{\prime}(f, \delta; (S,d_S))  + \text{jump}\left(f;(S,d_S)\right).
 \end{equation}

We begin by recalling what we mean by a large deviation principle \cite{dembo2009large,Budhirajaweakconv,den-hollander,varadhan}. We remark that some authors use the terminology {\em good rate function} for what is defined to be a rate function below.

\begin{defn}{\label{def:ldp}}
   Let $(S,d_S)$ be a metric space and $\left\{Y_n : n \geq 1\right\}$  be a sequence of $S$-valued random variables. A function $I: S \to [0, \infty]$ is called a rate function if $\{x\in S: I(x) \le r\}$ is compact for every $r<\infty$.
   Let $\left\{b(n) : n \geq 1\right\}$ be a sequence of positive real numbers and $I : S \to [0,\infty]$ be a rate function.  We say that the sequence $\left\{Y_n : n \geq 1\right\}$ follows a large deviation principle (LDP) in $S$ with speed $b(n)$ and rate function $I$ if the following holds true :
\begin{equation}\label{ldpdefined}
   \limsup_{n \to \infty} \dfrac{1}{b(n)} \log \mathbb{P} \left( Y_n \in F \right) \leq - \inf_{y \in F} I(y), \;\;  \liminf_{n \to \infty} \dfrac{1}{b(n)} \log \mathbb{P} \left( Y_n \in G \right) \geq -\inf_{y \in G} I(y), 
\end{equation} 
for any closed subset $F$ and open subset $G$ of $S$. The first and second inequality in  \eqref{ldpdefined}
 will be referred to as the LDP upper bound and the LDP lower bound, respectively. 
\end{defn}

We say the rate function is \textit{trivial} if it only takes values in $\left\{0,\infty\right\}$. The \textit{(effective) domain} of the rate function $I$ is the set where it takes finite values.

\subsection{Large deviation principle with respect to weak topology}	{\label{sec:results-weak}}
  Define $H^n \in \cW_0$ as
	\begin{align}
		H^n(x,y,t)\equiv H^n_t(x,y) := \sum_{i,j=1}^n X^n_{i,j}(t)\mathbbm{1}_{Q^n_{i,j}}(x,y), \;  \; x,y \in [0,1].\label{eq:hndefn}
	\end{align}
	In other words, $H^n_t$ is the graphon representation of the directed random graph $G^n(t)$. Let 
	$$\clb_{\sqrt{T}}\left(L^2\left([0,1]^2_T\right)\right) := \left\{ \phi \in L^2\left([0,1]^2_T\right) : \|\phi\|_2 \leq \sqrt{T}\right\},$$
	and we equip this  ball with the weak topology inherited from 
    $L^2\left([0,1]^2_T\right)$. Here and throughout we denote the norm on the $L^p$ space of functions on a Euclidean space by $\|\cdot\|_p$ and omit reference to the particular Euclidean space being considered. With the above topology, the space $\clb_{\sqrt{T}}\left(L^2\left([0,1]^2_T\right)\right)$ is a compact metrizable topological space which contains $\cW_0$ as a closed (and hence compact) subset. We shall refer to this topology as the \textit{weak topology} on $\cW_0$ and let $d_{\cW_0,\text{weak}}$ be a metric on $\cW_0$ which generates this topology. Therefore we can consider $\{H^n : n \geq 1\}$ as a sequence of $\left(\cW_0, d_{\cW_0, \text{weak}}\right)$-valued random variables. We remark that, later in the section we will consider the trajectories $t \mapsto H^n(\cdot, \cdot, t)$ in the space of graphons and regard these paths as  random variables with values in the Skorohod path space on the space of graphons. We caution the reader on the different topologies on the space $\cW_0$ and the Skorohod path space.

    Our first pair of results provides the Law of Large Numbers (LLN) and the LDP for the sequence $\left\{H^n : n \geq 1\right\}$ under the above weak topology. Towards that goal, for any $(i,j) \in [n]^2$, set 
	$$ w_{i,j}^n(t) := \beta^{n,+}_{i,j}(t)/(\beta^{n,+}_{i,j}(t)+\beta^{n,-}_{i,j}(t)),  \; \; w^n_t  := \sum_{i,j} w_{i,j}^n(t) \mathbbm{1}_{Q_{i,j}^n} = \dfrac{\beta_t^{n,+}}{\beta_t^{n,+}+\beta_t^{n,-}}, \text{ on } [0,1]^2.$$ Note that, $w^n_t$ is well-defined since $\beta^{n,\pm}_t >0$. We further define 
	\begin{equation}
    \label{eqn:w-w-st-def}
	    w := \dfrac{\beta^+}{\beta^+ + \beta^-} \in \cW_0. 
	\end{equation} 
    The proof of the following law of large numbers result is provided in Section \ref{lln}.
		\begin{thm}{\label{cor2:thmlin}}
		$H^n$ converges almost surely to $w$ as a sequence of $\left(\cW_0, d_{\cW_0, \text{weak}}\right)$-valued random variables.
	\end{thm}
	Our next result gives a large deviation principle for $\{H^n : n \geq 1\}$.
		\begin{thm}{\label{thm:weakLDP}}
        Suppose Assumption \ref{assmp:2}\eqref{assump:1A} is satisfied. Then the sequence $\{H^n : n \geq 1\}$ satisfies a large deviation principle in $\left(\cW_0, d_{\cW_0, \text{weak}}\right)$ with speed $a(n)n^2$ and rate function $J^{(\beta^+,\,\beta^-)}$ defined as
		\begin{align*}
			J^{(\beta^+,\,\beta^-)}(\phi) :=  \int_{[0,1]^2 \times [0,T]} \cQ_1 \left(\phi_s(x,y), \beta_s^+(x,y), \beta_s^-(x,y)\right) dx\,dy\,ds,
		\end{align*}
		for all $\phi \in \cW_0$, where $\cQ_1 : [0,1] \times [0,\infty)^2 \to \mathbb{R}$ is defined as 
		$$\cQ_1(u,v^+,v^-):= \left(\sqrt{v^+(1-u)}-\sqrt{v^-u}\right)^2, \; \forall \; u \in [0,1], v^+,v^- \geq 0.$$
	\end{thm}

The rate function in Theorem \ref{thm:weakLDP} admits a natural interpretation in terms
of the long-time behavior of the underlying edge-flip Markov chains.
Indeed, for each space--time point $(x,y,s)$, the corresponding edge is a
two-state Markov chain with birth and death rates
$\beta_s^+(x,y)$ and $\beta_s^-(x,y)$, whose stationary occupation
probability is
$w_s(x,y)$.
The local cost given by the integrand in the above integral
measures the energetic cost of forcing this Markov chain to maintain the
atypical occupation level $\varphi_s(x,y)$ rather than its (local) equilibrium
value $w_s(x,y)$. The global rate function is simply the aggregate of
these local costs over all edge locations and times.

This should be contrasted with the static dense graph setting of~\cite {Chatterjee2011, Dupuis2022, Dhara2022} for Erd\H{o}s--R\'enyi graphs, where the local costs are Bernoulli relative entropies, 
and from the  dynamic setting of \cite{Braunsteins2023}, where the rate function is an action integral over space and time of the local Lagrangian for a two-state jump Markov process.

In our setting, due to the high edge-flip rates, the rate function is a dynamical
Dirichlet-form cost associated with changing the empirical occupation
measure of an underlying edge-flip Markov process. The square-root
structure is characteristic of Donsker--Varadhan theory for empirical
measures of reversible Markov chains.

Proof of the above theorem can be found in Section \ref{ldp}, starting in Section  \ref{sec:isrf} which shows that $J^{(\beta^+,\,\beta^-)}$ is a rate function on $\cW_0$ when equipped with the weak topology, namely it has compact sub-level sets. The proof of the large deviation upper bound is completed in Section \ref{sec:pfofldpupp}. The complementary lower bound is proved in Section \ref{sec:pflowbd}. The upper bound is proved by using weak convergence and stochastic control methods whereas the proof of the lower bound makes a more direct use of the Donsker-Varadhan variational formula and Laplace asymptotics, see \cite[Chapter 1,2]{Budhirajaweakconv} for details on these methods.

\begin{rem}{\label{rem:domain-J}}
	Since $0 \leq \cQ_1(u,v^+,v^-) \leq v^++v^-$ fo any $u \in [0,1]$ and $v^+,v^- \geq 0$, we have 
	$$ J^{(\beta^+,\, \beta^-)}(\phi) \leq \|\beta^+\|_1 + \|\beta^-\|_1 < \infty, \; \text{ for any } \phi \in \cW_0.$$
	In particular, the entire $\cW_0$ space is the domain of the rate function $J^{(\beta^+,\, \beta^-)}.$
\end{rem}

It is natural to ask if   $H^n$ satisfies a sample path large deviation principle.
More precisely, let
$$\clb_1\left(L^2\left([0,1]^2\right)\right) := \left\{ f \in L^2\left([0,1]^2 \right): \|f\|_2 \leq 1\right\},$$
and we equip this unit ball with the weak topology inherited from $L^2\left([0,1]^2\right)$, which makes this space compact.  Clearly, $\cls_0$ is contained in this ball as a closed (hence compact) subset. As before, this inherited topology on $\cS_0$ will be referred to as the \textit{weak topology} and we denote by $d_{\cS_0, \text{weak}}$  a metric on $\cS_0$ inducing it. 
We may view $H^n$ as a $\mathbb{D}\left([0,T]:\left(\mathcal{S}_0, d_{\cS_0, \text{weak}}\right)\right)$-valued random variable and we write 
$H^n_{\bcdot}$ in order to emphasize this sample path perspective of the graphon process.
Then it is natural to study a LDP for  $\left\{H^n_{\bcdot} : n \geq 1 \right\}$ in the space $\mathbb{D}\left([0,T]: \left(\mathcal{S}_0, d_{\cS_0, \text{weak}}\right)\right)$. However, as the following proposition demonstrates, such a large deviation principle cannot hold. The proof of this proposition is deferred to Section~\ref{sec:pfno-ldp-path-weak}.

\begin{prop}{\label{prop:no-ldp-weak}}
	  Suppose Assumption \ref{assmp:2}\eqref{assump:1AB} is satisfied. Then $\left\{H^n_{\bcdot} : n \geq 1\right\}$, regarded  as a sequence of $\mathbb{D}\left([0,T]:\left(\mathcal{S}_0, d_{\cS_0, \text{weak}}\right)\right)$-valued random variables does not satisfy  a LDP with a non-trivial rate function.
\end{prop}
	
The reason for such  behavior  is that the edge-flipping rates are of order $a(n)$ which approach $\infty$ as $n\to \infty$. Indeed, if one considers the elementary example of a $\{0,1\}$-valued Markov process $Z^n$ with jump-rates $a(n)$ then, a calculation using the modulus of continuity $\text{Osc}^{\prime}$ shows that such a sequence cannot be tight in $\mathbb{D}\left([0,T]: \{0,1\}\right)$ and therefore does not satisfy a LDP. Arguing this for the sample paths of graphons however requires some additional work as the high oscillatory behavior at the individual edge level may get mollified by the averaging over the edges which is inherent in the weak topology on the space of graphons.

In view of the high rate of oscillations, the natural object for a sample path LDP is the graphon process obtained from a small time-window local averaging of the original graphon process. We now introduce this process.
First note that, we may choose the metric in $d_{\cS_0, \text{weak}}$ in such a way that 
$d_{\cS_0, \text{weak}} (f_1, f_2 )\leq \| f_1 - f_2 \|_2$, for any $f_1, f_2 \in \cS_0$.
Fix $\epsilon \in (0,T)$ and define the map $\text{Av}^{\epsilon} : \mathcal{W}_0 \to \mathbb{C}\left([0,T]:\left(\mathcal{S}_0, d_{\cS_0, \text{weak}}\right)\right)$ as follows :  For $\phi \in \mathcal{W}_0$,
\begin{equation}{\label{def:avg-graphon}}
	\text{Av}^{\epsilon}(\phi)_t := \dfrac{1}{|U_{\epsilon}(t)|} \int_{U_{\epsilon}(t)} \phi_s\, ds, \; \forall \; t \in [0,T],
\end{equation} 	
where $U_{\epsilon}(t)= [t-\epsilon,t+\epsilon] \cap [0,T]$ and $|U_{\epsilon}(t)|$ is its length. In other words, $\text{Av}^{\epsilon}(\phi)_t$ is obtained by averaging over all time instants around $t$ within a gap of $\epsilon$. 
 The local-averaged path is indeed continuous:  For any $ 0 \leq r < t \leq T$, we have  
\begin{align}{\label{eq:av-lip}}
 d_{\cS_0, \text{weak}}(\text{Av}^{\epsilon}(\phi)_t, \text{Av}^{\epsilon}(\phi)_r) &\le
     \| \text{Av}^{\epsilon}(\phi)_t - \text{Av}^{\epsilon}(\phi)_r  \|_2 \nonumber \\ 
& \leq 2(t-r) \left( \dfrac{1}{\epsilon}+\dfrac{T}{\epsilon^2}\right) := (t-r)L(T,\epsilon).
\end{align}

 The $\epsilon$-window local moving average process $A^{n,\epsilon}_{\bcdot} := \text{Av}^{\epsilon}(H^n)$  is a  natural observed statistic in view of a rapidly evolving system; $A^{n,\epsilon}_t(x,y)$ is the average amount of time the edge connecting vertex $x$ to vertex $y$ is present within the time interval $U_{\epsilon}(t)$. We have the following sample path LDP for the window averaged process as an immediate corollary of Theorem~\ref{thm:weakLDP} by an application of contraction principle, details can be found in Section~\ref{sec:pfcor-weakLDP-path}.
 
	\begin{cor}{\label{cor:weakLDP-path}}
		For any $\epsilon \in (0,T)$, the sequence $A^{n,\epsilon}_{\bcdot}$ converges almost surely, as $n \to \infty$, to $w^{\epsilon}_{\bcdot} := \text{Av}^{\epsilon}(w)$ in $\mathbb{C}\left([0,T]:\left(\mathcal{S}_0, d_{\cS_0, \text{weak}}\right)\right)$. Moreover, if we assume Assumption \ref{assmp:2}\eqref{assump:1A} is satisfied, then the sequence of random variables $\{ A^{n,\epsilon}_{\bcdot}: n \geq 1\}$ satisfies a large deviation principle in $\mathbb{C}\left([0,T]:\left(\mathcal{S}_0, d_{\cS_0, \text{weak}}\right)\right)$ with speed $a(n)n^2$ and rate function $I^{(\beta^+,\,\beta^-)}_{\epsilon}$ defined as
		\begin{align*}
			I_{\epsilon}^{(\beta^+,\,\beta^-)}(\phi_{\bcdot}) :=  \inf \left\{J^{(\beta^+,\,\beta^-)}(\tilde \phi) : \text{Av}^{\epsilon}(\tilde \phi )= \phi_{\bcdot}, \; \tilde \phi \in \cW_0\right\}, \; \phi_{\bcdot} \in \mathbb{C}\left([0,T]:\left(\mathcal{S}_0, d_{\cS_0, \text{weak}}\right)\right).
		\end{align*}
	\end{cor}

		\begin{rem}
		Theorem~\ref{cor2:thmlin} shows that $H^n$ converges to $w$ almost surely with respect to the weak topology. 
		In fact the weaker property of convergence in probability can be deduced simply by observing that, since $\beta^+(1-w)=\beta^-w$ on $[0,1]^2_T$,  $J^{(\beta^+,\,\beta^-)}(w)=0$ if and only if $\phi =w$ a.e. on $[0,1]^2_T$. This also shows that $J^{(\beta^+,\,\beta^-)}$ is non-trivial.
		Similarly, using lower semi-continuity of $J^{(\beta^+,\,\beta^-)}$ (see Proposition~\ref{lem:weaksemicont2}), one can easily check that
         $I_{\epsilon}^{(\beta^+,\,\beta^-)}(\phi_{\bcdot})=0$ if and only if $\phi_{\bcdot}=w^{\epsilon}_{\bcdot}$.
	\end{rem}

	\begin{rem}{\label{rem:rate-avg-process-domain}}
		From Remark~\ref{rem:domain-J}, it is immediate that 
		$$ \text{Av}^{\epsilon}(\cW_0) = \left\{ \phi_{\bcdot} : I_{\epsilon}^{(\beta^+,\,\beta^-)}(\phi_{\bcdot})<\infty\right\} = \left\{ \phi_{\bcdot} : I_{\epsilon}^{(\beta^+,\,\beta^-)}(\phi_{\bcdot})\leq  \|\beta^+\|_1 + \|\beta^-\|_1 \right\} $$
		and hence it is compact in  the space $\mathbb{C}\left([0,T]:\left(\mathcal{S}_0, d_{\cS_0, \text{weak}}\right)\right)$, since $I_{\epsilon}^{(\beta^+,\,\beta^-)}$ has compact sub-level sets by virtue of being a rate function.
	\end{rem}

	Corollary~\ref{cor:weakLDP-path} provides a sharp contrast between the large deviation behavior of the graphons $H^n_t$  and
    that of the window-averaged graphons $\text{Av}^{\epsilon}(H^n)_t$ at a finite collection of time instants.   The following proposition  is an application of similar results in~\cite{Chatterjee2011, Dhara2022, Markering2023}, see Section~\ref{sec:pfpropldp-point-weak} for the proof. 
	
	\begin{prop}{\label{prop:ldp-point-weak}}
		Suppose Assumption~\ref{assmp:2}\eqref{assump:1AB} is satisfied. Then for any $t \in (0,T]$, the graphon $H^n_t$ converges to $w_t$ almost surely in weak topology. Moreover, for any finite $\cT \subseteq (0,T]$, the sequence $\left\{\left(H^n_t\right)_{t \in \cT} : n \geq 1\right\}$ satisfies a large deviation principle in the $|\cT|$-fold product space of $\left(\cS_0, d_{\cS_0, \text{weak}}\right)$ with speed $n^2$ and rate function $I^{(\beta^+,\,\beta^-)}_{\text{point},\cT}$ given by 
		$$ I^{(\beta^+,\,\beta^-)}_{\text{point},\cT}\left(\left(f_t\right)_{t \in \cT}\right) := \sum_{t \in \cT} \int_{[0,1]^2} \cQ_2\left(f_t(x,y), w_{t}(x,y)\right)dx\,dy, \; \forall \; f_t \in \cS_0,\; t \in \cT,$$
		where $\cQ_2 : [0,1] \times (0,1) \to \mathbb{R}$ is defined as 
		$$ \cQ_2 (u,v) := u \log \dfrac{u}{v} + (1-u)\log \dfrac{1-u}{1-v}, \; \forall \; u\in [0,1], \, v \in (0,1),$$
		with the convention that $0 \cdot \log 0 =0$.
	
	\end{prop}
	
	Proposition~\ref{prop:ldp-point-weak}  shows that the collections $\{H^n_t\}$ and $\{\text{Av}^{\epsilon}(H^n)_t\}$ satisfy a finite-dimensional LDP at different rates 
    ($n^2$ versus $a(n)n^2$) and the form of the rate functions are strikingly different.

	\subsection{Large deviation principle with respect to cut-norm topology}
	
The weak topology on $\cls_0$ is  coarse and admits few interesting continuous functions. A more suitable metric for $\cls_0$ is the \textit{cut distance} $d_{\square}$, or its equivalent metric $d_{\infty \to 1}$, which can be defined as follows : For all $f,g \in \cls_0$, 
 \begin{align*}
 	d_{\square} (f,g) &:= \sup_{S,T \subseteq [0,1]} \Big \rvert \int_{S \times T} (f(x,y)-g(x,y))dx\,dy \Big \rvert
 \end{align*}  
 and 
 $$ d_{\infty \to 1} (f,g) := \sup_{\substack{a,b \\ \|a\|_{\infty}, \|b\|_{\infty} \leq 1}} \int_{[0,1]^2} a(x)b(y)\left(f(x,y)-g(x,y)\right)dx\,dy,$$
 where the first supremum is over all measurable sets $S,T \subseteq [0,1]$ while the second is over all measurable maps $a,b:[0,1] \to [-1,1]$.
 It can be easily seen that $d_{\square} (f,g) \leq d_{\infty \to 1}(f,g) \leq 4d_{\square} (f,g)$, which implies that the two metrics $d_{\square}$ and $d_{\infty \to 1}$ generate the same topology on $\cls_0$.  
 Moreover, $d_{\infty \to 1}(f,g) \leq \|f-g\|_1 \leq \|f-g\|_2$ for all $f,g \in \cls_0$. 
 The topology on $\cls_0$ given by $d_{\square}$ (equivalently $d_{\infty \to 1}$) is stronger (finer) than the weak topology on $\cls_0$. The following result gives a stronger version of Theorem~\ref{cor2:thmlin} by giving convergence in the metric $d_{\square}$.

 \begin{thm}{\label{thm:llnnew}}
As $n \to \infty$, we have 
$$ \int_0^T d_{\square} \left(H_t^n, w_t \right)dt \stackrel{a.s.}{\longrightarrow} 0.$$
 \end{thm}

Proofs of Theorem~\ref{thm:llnnew} 
is given in Section \ref{lln}. Unfortunately, the metric space $\left(\cS_0, d_{\square}\right)$ is not compact, see~\cite[Example F.6]{janson}. But we have a natural equivalence relation on this space. If we set $\mathscr{S}$ to be the group of all measure preserving bijections on $[0,1]$, we can define, for any $f \in \cls_0$ and $\sigma \in \mathscr{S}$, a new function $f^{\sigma}\in \cls_0$ as $f^{\sigma}(x,y):=f(\sigma(x),\sigma(y))$ for all $x,y \in [0,1]$. The bijection $\sigma$ represents a relabeling of the vertices of the graphon $f$. For $f,g \in \cls_0$, we say $f \sim g$ if and only if $\inf_{\sigma \in \sS} d_{\square}\left(f,g^{\sigma}\right)=0$. In other words, we identify two graphons $f$ and $g$ if they are the same up to a relabeling of their vertices. This defines an equivalence relation on $\cls_0$. Let $\widehat{f}$ be the equivalence class containing $f$. Let $\widehat{\cls}_0$ be the collection of all these equivalence classes equipped with the \textit{cut metric} $\delta_{\square}$ :
$$ \delta_{\square}\left( \widehat{f},\widehat{g}\right) := \inf_{f^{\prime} \in \widehat{f},g^{\prime} \in \widehat{g}} d_{\square}(f^{\prime},g^{\prime}) = \inf_{f^{\prime} \in \widehat{f}} d_{\square}(f^{\prime},g) = \inf_{g^{\prime} \in \widehat{g}} d_{\square}(f,g^{\prime}). $$	
Under this metric, the set $\widehat{\cls_0}$ is a compact metric space. For more details on the cut metric, we refer to ~\cite{Lovasz,janson}.
 Since the canonical map $f \mapsto \widehat f$ from $\left(\cls_0, d_{\square}\right)$ to $\left(\widehat{\cls_0}, \delta_{\square} \right)$ is Lipschitz continuous, the following corollary is immediate from Theorem~\ref{thm:llnnew}. 

 \begin{cor}{\label{cor:llnnew}}
As $n \to \infty$, we have 
$$ \int_0^T \delta_{\square} \left( \widehat{H_t^n}, \widehat{w_t} \right)dt \stackrel{a.s.}{\longrightarrow} 0.$$
 \end{cor}

We next study large deviations behavior under the cut metric. We first present a LDP for the graphons at fixed time points, under the  topology induced by the cut metric,  in the following theorem. 

	\begin{thm}{\label{thm:ldp-point-cut}}
	Suppose Assumption~\ref{assmp:2}\eqref{assump:1AB} is satisfied. Then for any $t \in (0,T]$, the sequence $\widehat{H^n_t}$ converges to $\widehat{w_t}$ almost surely in the metric space $\left(\widehat{\cls_0}, \delta_{\square} \right)$. Moreover, for any finite $\cT \subseteq (0,T]$, the sequence $\left\{ \left( \widehat{H_t^n}\right)_{t \in \cT}: n \geq 1\right\}$ satisfies a large deviation principle in the $|\cT|$-fold product space of $\left(\widehat{\cls_0}, \delta_{\square} \right)$ with speed $n^2$ and rate function $\widehat I^{(\beta^+,\,\beta^-)}_{\text{point},\cT}$ given by 
	$$ \widehat I^{(\beta^+,\,\beta^-)}_{\text{point},\cT} \left( \left( \widehat{f_t}\right)_{t \in \cT}\right) := \inf \left\{ I^{(\beta^+,\,\beta^-)}_{\text{point},\cT} \left( \left( {f_t}\right)_{t \in \cT}\right)  : f_t \in \widehat{f_t}, \;\forall \; t \in \cT\right\},$$
	for any $\widehat{f_t}\in \widehat{\cS_0}, t \in \cT$.
\end{thm}

The proof of Theorem~\ref{thm:ldp-point-cut} follows from the proof of Proposition~\ref{prop:ldp-point-weak} and the block-approximation  method used in~\cite[Sections 3.2 and 3.3]{Markering2023}, and is therefore  omitted. We now consider a sample path LDP. For any sample path $\phi_{\bcdot}$ in $\cls_0$, we write $\widehat{\phi_{\bcdot}}$ for the corresponding sample path in $\widehat \cls_0$ of the equivalence classes, i.e.,  $\left(\widehat{\phi_{\bcdot}}\right)_t = \widehat{\phi_t}$ for any $t \in [0,T]$. Once again, the highly oscillatory behavior rules out a LDP without any local averaging and we have an analogue of Proposition~\ref{prop:no-ldp-weak}, given below. 

\begin{prop}{\label{prop:no-ldp-cut}}
	Suppose Assumption~\ref{assmp:2}\eqref{assump:1AB} is satisfied. Then $\left\{\widehat{H^n_{\bcdot}} : n \geq 1\right\}$, the sequence of $\mathbb{D}\left([0,T]:\left(\widehat{\mathcal{S}_0}, \delta_{\square}\right)\right)$-valued random variables,
    does not satisfy  a LDP with a non-trivial rate function.
\end{prop}

The proof of Proposition~\ref{prop:no-ldp-cut} is identical to the proof of Proposition~\ref{prop:no-ldp-weak} (only change being we use Theorem~\ref{thm:ldp-point-cut} instead of Proposition~\ref{prop:ldp-point-weak}) and, therefore, omitted.  Proposition~\ref{prop:no-ldp-cut} should be compared with ~\cite[Theorem 1.4]{Braunsteins2023} which proves a LDP for the sample path process under the conditions of that paper. As discussed earlier, this different behavior is due to very high jump rates in our system captured by the condition that $a(n) \to \infty$, see Remark~\ref{rem:cutpathLDP} for a detailed discussion of this point. 
Nevertheless, by a local time-averaging one can smoothen out the oscillatory behavior and establish an LDP for the window-averaged process $\left\{\widehat{A^{n,\epsilon}_{\bcdot}} : n \geq 1\right\}$. Note that (\ref{eq:av-lip}) guarantees the continuity of the sample path for this time-averaged process with respect to the cut metric.  The next pair of results  prove the almost sure convergence and the LDP for this  process in the space $\mathbb{C}\left([0,T]:\left(\widehat{\mathcal{S}_0}, \delta_{\square}\right)\right)$, respectively.

	\begin{prop}{\label{prop:lln-avg-path}}
	For any $\epsilon \in (0,T)$, The sequence $\widehat{A^{n,\epsilon}_{\bcdot}}$ converges almost surely, as $n \to \infty$, to $\widehat{w^{\epsilon}_{\bcdot}} $ 
    in $\mathbb{C}\left([0,T]:\left(\widehat{\mathcal{S}_0},\delta_{\square}\right)\right)$. 
\end{prop}

	\begin{thm}{\label{thm:LDP-avg-path}}
 Suppose that Assumption \ref{assmp:2}\eqref{assump:1AB} is satisfied. Then, for any $\epsilon \in (0,T)$, the sequence of random variables $\{ \widehat{A^{n,\epsilon}_{\bcdot}}: n \geq 1\}$ satisfies a large deviation principle in $\mathbb{C}\left([0,T]:\left(\widehat{\mathcal{S}_0},\delta_{\square}\right)\right)$ with speed $a(n)n^2$ and rate function $\widehat I^{(\beta^+,\,\beta^-)}_{\epsilon}$ defined as
	\begin{align}{\label{def:rate-path-cut}}
	\widehat 	I_{\epsilon}^{(\beta^+,\,\beta^-)}(\widehat{\phi_{\bcdot}}) := \sup_{\eta >0} \inf \left\{ I_{\epsilon}^{(\beta^+, \, \beta^-)}(\phi^*_{\bcdot}) :  \sup_{t \in [0,T]} \delta_{\square} \left(\widehat {\phi_t}, \widehat{\phi^*_t}\right) < \eta\right\},
	\end{align}
	for $\widehat{\phi_{\bcdot}} \in \mathbb{C}\left([0,T]:\left(\widehat{\mathcal{S}_0}, \delta_{\square}\right)\right)$.
\end{thm}

Proofs of Proposition~\ref{prop:lln-avg-path} and Theorem~\ref{thm:LDP-avg-path} are deferred to Section~\ref{lln} and Section~\ref{sec:cutldp} respectively. Let us mention a few words on the proof of Theorem~\ref{thm:LDP-avg-path}. The main idea is to make use of the Dawson-G\"{a}rtner projective limit LDP~\cite{Dembo2018}. Towards that goal, we first prove LDP results for finite-dimensional projections of the sample path, given by the following proposition.

	\begin{prop}{\label{prop:LDP-avg-path-finite}}
	Suppose that Assumption \ref{assmp:2}\eqref{assump:1AB} is satisfied. Then, for any $\epsilon \in (0,T)$, and any finite $\cT \subseteq [0,T]$, the sequence of random variables $\{ \left(\widehat{A^{n,\epsilon}_{t}}\right)_{t \in \cT}: n \geq 1\}$ satisfies a large deviation principle in the $|\cT|$-fold product space of $\left(\widehat{\mathcal{S}_0},\delta_{\square}\right)$ with speed $a(n)n^2$ and rate function $\widehat I^{(\beta^+,\,\beta^-)}_{\epsilon, \cT}$ given by
	\begin{align*}
		\widehat I^{(\beta^+,\,\beta^-)}_{\epsilon,\cT} \left( \left(\widehat{f_t}\right)_{t \in \cT}\right) := \sup_{\eta >0} \inf \left\{ I_{\epsilon}^{(\beta^+, \, \beta^-)}(\phi^*_{\bcdot}) : \max_{t \in \cT} \delta_{\square} \left(\widehat {f_t}, \widehat{\phi^*_t}\right) < \eta \right\},
	\end{align*}
	for  $\widehat{f_t} \in \widehat{\cls_0}$, $t \in \cT$. 
\end{prop}

We then prove an exponential tightness result (see Definition~\ref{def:exp-tight}) for the sample path of the time-averaged process and combine it with Proposition~\ref{prop:LDP-avg-path-finite} to complete the proof of Theorem~\ref{thm:LDP-avg-path}. Details can be found in Section~\ref{sec:pf-LDP-path-cut}.

\begin{rem}{\label{weakandcut}}
    The proof of Proposition~\ref{prop:LDP-avg-path-finite} builds on the LDP with respect to the weak topology in Corollary~\ref{cor:weakLDP-path} and is inspired by  the techniques of~\cite{Chatterjee2011} (see also \cite{Dhara2022}) to lift the LDP from the weak topology to the cut-metric topology. One may wonder if it is possible to directly establish Proposition~\ref{prop:LDP-avg-path-finite} using the weak convergence methods, for example those used in~\cite{Dupuis2022} or by suitably modifying the proof of Theorem~\ref{thm:weakLDP}. One of the main hurdles in the first approach is that since \cite{Dupuis2022} considers a static setting, their basic object of interest is a sequence of (time independent) random graphons giving representations of the underlying inhomogeneous random graphs, in particular the graphons take values in $\left\{0,1\right\}$. As a consequence, the control variables take a relatively simple form and the conditional distribution of each control variable in the proof of \cite[Theorem 4.1]{Dupuis2022} can be characterized by a real-valued parameter.  This allows for a proof of the large deviation principle essentially along the lines of the proof of the LLN for their model, using standard concentration bounds. In contrast, in our setting the time-averaged graphons take values in the whole interval $[0,1]$ and due to the underlying random time evolution of the graphon one needs to consider time-dependent controls; as a result the study of the asymptotic behavior of the controlled averaged graphons in the cut-metric topology becomes much less tractable. 
    Furthermore, the second approach (namely suitably modifying the proof of Theorem~\ref{thm:weakLDP}) also seems challenging, as convergence in the space $\left(\widehat{\cls_0}, \delta_{\square}\right)$ does not easily translate to useful convergence properties in the space $\left(\cls_0, d_{\square}\right)$. Due to these difficulties, we take the approach of first proving the LDP for the weak topology and then suitably lifting it to the cut-metric topology.
    \end{rem}

    \begin{rem} \label{rem:3.17}
    	If 	$\widehat 	I_{\epsilon}^{(\beta^+,\,\beta^-)}(\widehat{\phi_{\bcdot}}) < \infty$ for some $\widehat{\phi_{\bcdot}} \in \mathbb{C}\left([0,T]:\left(\widehat{\mathcal{S}_0}, \delta_{\square}\right)\right)$, then for any $\eta>0$, there exists $\phi^{*,\eta}_{\bcdot}$ satisfying $\sup_{t \in [0,T]} \delta_{\square} \left(\widehat {\phi_t}, \widehat{\phi^{*,\eta}_t}\right) < \eta$ and $I_{\epsilon}^{(\beta^+,\,\beta^-)}(\phi^{*,\eta}_{\bcdot}) < \infty.$ By Remark~\ref{rem:rate-avg-process-domain} and~(\ref{eq:av-lip}), we have $t \mapsto \widehat{\phi_t^{*,\eta}}$ to be $L(T,\epsilon)$-Lipschitz and hence $t \mapsto \widehat{\phi_t}$ is $(L(T,\epsilon)+2\eta)$-Lipschitz. Since this is true for arbitrary $\eta >0$, we conclude that  $\widehat 	I_{\epsilon}^{(\beta^+,\,\beta^-)}(\widehat{\phi_{\bcdot}}) < \infty$ only if $t \mapsto \widehat{\phi_t}$ is $L(T,\epsilon)$-Lipschitz. This shows that domain of the candidate rate function $\widehat 	I_{\epsilon}^{(\beta^+,\,\beta^-)}$ is relatively compact in the space  $\mathbb{C}\left([0,T]:\left(\widehat{\mathcal{S}_0}, \delta_{\square}\right)\right)$. Since, by definition, the function $\widehat 	I_{\epsilon}^{(\beta^+,\,\beta^-)}$ is lower semi-continuous, we conclude that it has compact sub-level sets, hence, a valid rate function. Moreover, Remark~\ref{rem:rate-avg-process-domain} also guarantees that 
    	$$ \left\{ \widehat{\phi_{\bcdot}} : \widehat 	I_{\epsilon}^{(\beta^+,\,\beta^-)}(\widehat{\phi_{\bcdot}}) < \infty\right\} =  \left\{ \widehat{\phi_{\bcdot}} : \widehat 	I_{\epsilon}^{(\beta^+,\,\beta^-)}(\widehat{\phi_{\bcdot}}) \leq \|\beta^+\|_1 + \|\beta^-\|_1 \right\}$$
    	is compact.
    \end{rem}

    \begin{rem}\label{rem:lsc}
It can in fact be shown that, under the assumptions of Theorem \ref{thm:LDP-avg-path}, the function
$$
\mathbb{C} \left([0,T] : \left(\widehat{\cls_0},\delta_{\square}\right)\right)  \ni \widehat{\phi_{\bcdot}}  \mapsto  \inf \left\{ I_{\epsilon}^{(\beta^+, \, \beta^-)}(\phi^*_{\bcdot}) :  \widehat {\phi^*_t} = \widehat{\phi_t}, \; \forall \; t \in [0,T]\right\}
$$
is lower semi-continuous from which it follows that the rate function in \eqref{def:rate-path-cut} takes the simpler form
{\begin{align*}
	\widehat 	I_{\epsilon}^{(\beta^+,\,\beta^-)}(\widehat{\phi_{\bcdot}}) &=  
    \inf \left\{ I_{\epsilon}^{(\beta^+, \, \beta^-)}(\phi^*_{\bcdot}) : \widehat {\phi^*_t} = \widehat{\phi_t}, \; \forall \; t \in [0,T] \right\}\\
    &= \inf_{\substack{\phi^*_{\bcdot} \\ \phi^*_t \in \widehat \phi_t, \,\forall\, t\in [0,T]}}\;\;  \inf_{\substack{\tilde \phi \\ \text{Av}^{\epsilon}(\tilde \phi )= \phi^*_{\bcdot}}}
    \left\{\int_{[0,1]^2 \times [0,T]} 
    \left(\sqrt{\beta^+(1-\tilde \phi)}-\sqrt{\beta^-\tilde \phi}\right)^2\right\}
    \end{align*}}
    for all $\widehat{\phi_{\bcdot}} \in \mathbb{C}\left([0,T]:\left(\widehat{\mathcal{S}_0}, \delta_{\square}\right)\right)$.
Similarly, it also holds true that for any finite $\cT \subseteq [0,T]$
$$	\widehat I^{(\beta^+,\,\beta^-)}_{\epsilon,\cT} \left( \left(\widehat{f_t}\right)_{t \in \cT}\right) := \inf \left\{ I_{\epsilon}^{(\beta^+, \, \beta^-)}(\phi^*_{\bcdot}) :  \widehat {\phi^*_t} = \widehat{f_t}, \; \forall \; t \in \cT\right\},$$
for  $\widehat{f_t} \in \widehat{\cls_0}$, $t \in \cT$.
However, the proof of the lower semi-continuity property  is substantially involved.  It requires a non-trivial extension of the Szemer\'{e}di regularity lemma for time-evolving graphons along with several other technical ingredients.  In order to avoid making the paper longer we have chosen not to present this result here. Instead, we have pursued a less complicated argument for the special case of time-homogeneous jump rates, i.e., $\beta^{\pm}_t = \gamma^{\pm} \in \cls$ for all $t \in [0,T]$. See Section~\ref{validrate:cut-time-hom}.
\end{rem}

\begin{rem}{\label{rem:cutpathLDP}}
 As noted previously, one work that has studied large deviations for time-varying graphons is \cite{Braunsteins2023}.  This work considered edge-flipping rates to be space and time homogeneous and hence the stationary distribution is also (spatially) homogeneous (i.e., an \erdos\;graph).  The paper established a sample path LDP on a fixed time horizon. 
  In the setting of \cite{Braunsteins2023} the jump rates are $O(1)$ and thus the transient (fixed time-horizon) properties of Markov chains for the edge-flips determine the large deviation behavior and the speed of the LDP is $O(n^2)$ (where $n$ is the number of vertices in the graph). In contrast, in the current work, edge-flips occur at rate $a(n)$, where $a(n)\to \infty$ as $n \to \infty$ and we establish an LDP with speed $a(n)n^2$ which is governed by the long-time behavior of the edge-flip Markov chains. Furthermore, in the current work, the edge-flipping rates are allowed to take very general forms that may be inhomogeneous in space and time. Our formulation considers locally time-averaged graphon process, since, as argued in Proposition \ref{prop:no-ldp-cut}, a sample path LDP for the original graphon valued process is impossible in our setting.  
\end{rem}

\begin{rem}{\label{rem:homldp}}
    Let $F$ be a simple directed graph on the vertex set $[k]=\left\{1,\ldots,k\right\}$ and for any $h \in \mathcal{S}_0$, we define
$$ \ft(h,F) := \int_{[0,1]^k} \left( \prod_{(i,j) \in E(F)} h(x_i,x_j) \right) \prod_{i=1}^k dx_i,$$
where $E(F)$ is the edge-set of $F$. In other words, if $h$ is the graphon representation of a simple directed graph $G$, then $\ft(h,F)$, the \textit{homomorphism density} corresponding to $F$, measures the probability that an uniformly random mapping from the vertex set of $F$ into the vertex set of $G$ is a homomorphism (edge-preserving map). It is easy to see that $\ft(h_1,F)=\ft(h_2,F)$ whenever $h_1 \sim h_2$ and hence $\ft(\widehat h, F)$ is well-defined. The graphon limit theory tells us that $\delta_{\square} \left(\widehat{h_n}, \widehat h \right) \to 0$ if and only if $\ft(\widehat{h_n}, F) \to \ft(\widehat h, F)$ for all simple directed graph $F$. Using contraction principle, as a direct corollary of Theorem~\ref{thm:LDP-avg-path}, we obtain that for each $F$, $\left\{\ft(\widehat{A^{n,\epsilon}_t}, F) : 0 \leq t \leq T\right\}$ satisfies a LDP in $\mathbb{C}\left([0,T] : \mathbb{R}\right)$ with speed $a(n) n^2$. Furthermore, a similar argument as for Proposition~\ref{prop:no-ldp-weak} and Proposition~\ref{prop:no-ldp-cut} shows that a path LDP cannot hold for the original process $\left\{\ft(\widehat{H^n_t}, F) : 0 \leq t \leq T\right\}$ due to its high oscillatory behavior. 
Finally, we remark that it would be of interest to establish a
a path LDP for the the locally averaged process of the homomorphism densities, i.e.,  for the process
	\begin{equation}{\label{eq:def-h-process}}
		 \left\{ T^{n,\epsilon,F}_t := \dfrac{1}{|U_{\epsilon}(t)|} \int_{U_{\epsilon}(t)} \ft(\widehat{H_s^n},F)\,ds : 0 \leq t \leq T\right\}
	\end{equation}
	in the space $\mathbb{C}\left([0,T] : \mathbb{R}\right)$. Such a result is not immediate from the results in the current paper. A weak convergence approach with a suitable weak topology similar to the proof of Theorem~\ref{thm:weakLDP} might be useful, but we leave this study for future work. LDP for one particular  case of \eqref{eq:def-h-process} does follow from our results, namely when $F$ is  a single directed edge. In this case, $\ft(\cdot, F)$ gives the edge density and $h \mapsto \ft(h,F)$ being linear we have $\ft(\widehat{A^{n,\epsilon}_t}, F) = T^{n,\epsilon,F}_t $, for all $t$.
\end{rem}

\begin{rem}
	Note that all the law of large numbers and large deviation results presented in this paper require no assumptions on the initial graph configuration and are invariant under the choice of the initial configuration. This is  a consequence of the fast evolving nature of the dynamics which ensure that the process decouples from its initial state within a very short time period. 
	\end{rem}

\subsection{Applications to rare event asymptotics}\label{sec:apps}
This section is devoted to applications of the large deviation principles derived earlier to a few concrete situations where we analyze the corresponding variational problem and deduce more details about the rate of exponential decay and optimal trajectories. The examples we shall consider are inspired by those discussed in~\cite[Section 7]{Dupuis2022} and \cite[Section 2]{Braunsteins2023}. Throughout this section we restrict ourselves to the case of  time-homogeneous jump rates, i.e., $\beta^{\pm}_t =\gamma^{\pm} \in \cS$ for all $t \in [0,T]$.

\subsubsection{Large deviations for local edge densities}
\label{sec:lded}  The focus here will be on the edge densities of the locally time-averaged graphons. In particular, we look at  the process 
$\left\{E^{n,\epsilon}_t := T_t^{n,\epsilon,F}\right\}$ (defined in~(\ref{eq:def-h-process})), where $F$ is just a single edge : 
$$ E^{n,\epsilon}_t = \dfrac{1}{|U_{\epsilon}(t)|} \int_{U_{\epsilon}(t)} \ft(H_s^n,F)\,ds = \dfrac{1}{|U_{\epsilon}(t)|} \int_{U_{\epsilon}(t)} \int_{[0,1]^2} H^n_s(x,y)\,dx\,dy\,ds, \; t \in [0,T].$$
Note that, $\ft(H_s^n,F)$ gives the edge density (total number of edges normalized by $n^2$) at time $t$ and hence  $E^{n,\epsilon}_t$ is the average edge density in 
an $\epsilon$-window around $t$. As observed earlier, when $f$ is a single edge, $E^{n,\epsilon}_t = \ft(\widehat{A^{n,\epsilon}_t},F) = \ft(A^{n,\epsilon}_t,F)$ and the map $\cS_0 \ni h \mapsto \ft(h,F) \in [0,1]$ is continuous with respect to the weak topology and hence we can deduce an LDP for the sequence of random variables $\{ E^{n,\epsilon}_{\bcdot}: n \geq 1\}$  in $\mathbb{C}\left([0,T]: \mathbb{R}\right)$  from Corollary~\ref{cor:weakLDP-path}. In particular, we note that $E_{\bcdot}^{n,\epsilon}$ converges almost surely to the constant function $p^*$ over $[0,T]$, where 
$$ p^* = \int_{[0,1]^2} \dfrac{\gamma^+}{\gamma^++\gamma^-} \in (0,1).$$

The rare event that we shall focus is that the number of edges is unusually large/small at least once during $[0,T]$. For example, fix some closed set $\cT \subseteq [0,T]$ and $\delta \in (0,1-p^*)$. Consider the following atypical event : $ \sup_{t \in \cT} E^{n,\epsilon}_t \geq (p^*+\delta)$. We are interested to understand what are the likely trajectories for the time-averaged graphon when conditioned on this rare event. The next  result answers the question.

\begin{thm}{\label{thm:likely-path-cond-edge}}
	Fix $\delta \in (0, 1- p^*), \epsilon \in (0,T)$. and closed $\cT \subseteq [0,T]$. Let 
	$$ \widehat \cH^{\epsilon}_{\delta, \cT} := \argmin \left\{ \widehat I^{(\beta^+,\, \beta^-)}_{\epsilon}\left(\widehat{\phi_{\bcdot}}\right)  :  \sup_{t \in \cT} \ft \left( \widehat{\phi_t}, F\right) \geq p^*+\delta \right\}. $$
	Then, $ \widehat \cH^{\epsilon}_{\delta, \cT} $ is non-empty and compact in $\mathbb{C}\left([0,T]:\left(\widehat{\mathcal{S}_0}, \delta_{\square}\right)\right)$. Moreover, for any $\varepsilon >0$ and large enough $n$,  
	$$  \log \mathbb{P} \left( \inf_{\widehat{\phi_{\bcdot}} \in \widehat \cH^{\epsilon}_{\delta, \cT}} \sup_{t \in [0,T]} \delta_{\square} \left( \widehat{A^{n,\epsilon}_t},\, \widehat{\phi_t}\right) \geq \varepsilon \,\Bigg \rvert  \sup_{t \in \cT} E_t^{n,\epsilon} \geq p^*+\delta\right) \leq e^{-Ca(n)n^2},$$
	for some positive constant $C$ which depends only on $\gamma^{\pm}, \cT,\epsilon, \varepsilon$ and $\delta$, but not $n$.
\end{thm}

 The next two results provide an exact characterization for $\widehat \cH^{\epsilon}_{\delta, \cT}$ through the solutions of a related variational problem.

\begin{prop}{\label{prop:edge-max-eq}}
Fix $\delta \in (0, 1- p^*), \epsilon \in (0,T)$. Consider the following variational problem :
	\begin{equation}{\label{eq:edge-var-prob}}
		\text{Minimize } \int_{[0,1]^2_T} \cQ_1 \left(\phi, \gamma^+,\gamma^-\right) \text{ over } \phi \in \cW_0 \text{ s.t. } \dfrac{1}{|U_{\epsilon}(t)|} \int_{[0,1]^2 \times U_{\epsilon}(t)} \phi \geq p^*+\delta.
	\end{equation} 
There is an unique (upto a set of measure zero) optimizer $\phi^{*,t,\delta}$ which is given by
	$$ \phi^{*,t,\delta}_s := \begin{cases}
		\dfrac{1}{2} + \dfrac{\gamma^+ - \gamma^-+\lambda_{\delta}}{2\sqrt{(\gamma^+-\gamma^-+\lambda_{\delta})^2 + 4 \gamma^+\gamma^-}}=:f^*_{\delta}, & \text{ if } s \in U_{\epsilon}(t),\\
		\dfrac{\gamma^+}{\gamma^++\gamma^-} = w^*, & \text{ if } s \notin U_{\epsilon}(t).
	\end{cases} $$
	where $\lambda_{\delta} \in \mathbb{R}$ is the unique solution to 
	\begin{equation}{\label{eq:def-lambda}}
		\int_{[0,1]^2} \left( 	\dfrac{1}{2} + \dfrac{\gamma^+ - \gamma^-+\lambda_{\delta}}{2\sqrt{(\gamma^+-\gamma^-+\lambda_{\delta})^2 + 4 \gamma^+\gamma^-}}\right) = p^*+\delta.
	\end{equation}
	Moreover, $\lambda_{\delta}$ is strictly increasing and continuous in $\delta$. The optimal value of the objective function in~(\ref{eq:edge-var-prob}) is given by 
	\begin{equation}{\label{eq:formula-obj}}
|U_{\epsilon}(t)|	\int_{[0,1]^2} \cQ_1 \left( f^*_{\delta}, \gamma^+, \gamma^-\right),
	\end{equation}
	which is also continuous in $\delta.$
\end{prop}

\begin{prop}{\label{prop:edge-max-eq2}}
		Consider the set-up of Theorem~\ref{thm:likely-path-cond-edge}.  Then 
	$$ \widehat \cH^{\epsilon}_{\delta, \cT}  = \left\{\widehat{\text{Av}^{\epsilon}(\phi)} : \phi = \phi^{*,t,\delta} \text{ for some } t \in \cT^* \right\},$$
	where $\cT^* := \argmin \left\{|U_{\epsilon}(t)| : t \in \cT\right\}$. 
\end{prop}

Proofs of the above results are deferred to Section~\ref{sec:applications}. Proposition~\ref{prop:edge-max-eq} should be compared to~\cite[Lemma 7.2]{Dupuis2022} which provides similar characterization for the graphon with optimal value for the rate function satisfying atypically large edge density in the context of stationary graphs. {When $\gamma^+$ and $\gamma^-$ are constant, the minimizer $f^*_{\delta}$ is a constant graphon, so conditioned on the rare event, the evolving graph is locally in time asymptotically \erdos. While this mirrors the optimizer in the static model of \cite{Dupuis2022}, the large deviation cost is fundamentally different: the static model is governed by the Bernoulli relative entropy $\cQ_2$, whereas the present dynamical setting is governed by the Donsker--Varadhan functional $\cQ_1$, namely the Dirichlet-form cost of maintaining an atypical occupation measure for the edge Markov chain.}

The above results  show that conditioned on the event that at least once during the pre-specified time set $\cT$ the locally time-averaged edge density of the graph touches atypically large value $p^*+\delta$, the time-averaged sample path is close to  $\widehat{\text{Av}^{\epsilon}\left(\phi^{*,t,\delta}\right)}$, for some $t \in \cT^*$. The sample path trajectory of the original  graph process is close to $\phi^{*,t,\delta}$, which, as we can see from Proposition~\ref{prop:edge-max-eq}, takes the equilibrium value $w^*$ outside $U_{\epsilon}(t)$ (incurring no cost) and takes the graph value $f^*_{\delta}$ in $U_{\epsilon}(t)$ which helps it achieve local edge density of $p^*+\delta$ at $t$. Clearly, it costs the least to maintain this extremal graph value and achieve this large edge density at time instants with smallest neighborhood size, i.e., for time instants in $\cT^*$.  If $\cT =[0,T]$, then $\cT^*=\left\{0,T\right\}$. In other words, the likely trajectory will either take the extremal value $f^*_{\delta}$ for small time $\epsilon$ at the start and then fall back to equilibrium at $w^*$ or keep being at $w^*$ for the whole time except for the last $\epsilon$ time unit where it should be $f_{\delta}^*$. An interesting point to note is that although the sample path $H^n_{\bcdot}$ is a c\`adl\`ag path, the most likely trajectories (conditioned on the event discussed here) are discontinuous, which again is a consequence of the highly oscillatory sample path behavior as $n$ becomes large. One can also prove simlar results for the event that during $\cT$, the local edge density has been usually low at some time instant, i.e., $\left(\inf_{t \in \cT} E^{n,\epsilon}_t \leq p^*-\delta \right)$ for some $\delta \in (0,p^*)$.

\subsubsection{Most likely path to a given destination}
This application is analogue to the example described in~\cite[Section 2.2]{Braunsteins2023}. Suppose we condition to end the time-averaged process $\widehat{A^{n,\epsilon}_{\bcdot}}$ near a fixed graphon $\widehat f$ at time $T$. A natural question is what are the most likely trajectories for the time-averaged graphon process and whether that is unique or not. The following theorem answers that question.

\begin{thm}{\label{thm:likely-path-end}}
	Fix $\epsilon \in (0,T)$ and $\widehat f \in \widehat{\cS_0}$. Let 
	$$ \widehat \cL^{\epsilon}(\widehat{f}) := \argmin \left\{ \widehat I^{(\beta^+,\, \beta^-)}_{\epsilon}\left(\widehat{\phi_{\bcdot}}\right)  :  \widehat{\phi_T} = \widehat f \right\}. $$
	Then, $ \widehat \cH^{\epsilon}_{\delta, \cT} $ is non-empty and compact in $\mathbb{C}\left([0,T]:\left(\widehat{\mathcal{S}_0}, \delta_{\square}\right)\right)$. Moreover, for any $\varepsilon >0$, small enough $\delta>0$, and large enough $n$,  
	$$  \log \mathbb{P} \left( \inf_{\widehat{\phi_{\bcdot}} \in \widehat \cL^{\epsilon} (\widehat f)}\sup_{t \in [0,T]} \delta_{\square} \left( \widehat{A^{n,\epsilon}_t},\, \widehat{\phi_t}\right) \geq \varepsilon \,\Bigg \rvert  \delta_{\square} \left(\widehat{A^{n,\epsilon}_T}, \widehat f\right) \leq \delta \right) \leq e^{-Ca(n)n^2},$$
	for some positive constant $C$ which depends only on $\gamma^{\pm},,\epsilon, \varepsilon$ and $\widehat f$.
\end{thm}

The proof of Theorem~\ref{thm:likely-path-end} is identical to that of Theorem~\ref{thm:likely-path-cond-edge}, hence omitted. The characterization of the set of optimal paths $\widehat \cL^{\epsilon}(\widehat{f})$ is trickier compared to $\widehat \cH^{\epsilon}_{\delta,\cT}$ since the rate function is getting minimized over the set $(\widehat{\phi_T}=\widehat f)$ which possesses no convexity structure similar to that of $\widehat \cH^{\epsilon}_{\delta,\cT}$. Nevertheless, inspired by~\cite[Theorem 2.8]{Braunsteins2023}, if we restrict our attention to the case of jump rates with block structure, we can analyze $\widehat \cL^{\epsilon}(\widehat{f})$.

\begin{prop}{\label{prop:most-likely-path}}
	Fix $\epsilon \in (0,T)$ and $\widehat f \in \widehat{\cS_0}$. Suppose $\gamma^{\pm}$ has a block structure and so does $f$, for some $f \in \widehat f$. In other words, there exists $0=a_0 < \ldots < a_k =1$ such that 
	$$ \gamma^{\pm}(x,y) = \gamma^{\pm}_{ij}, \; f(x,y) = f_{ij}, \; \forall (x,y) \in [a_{i-1}, a_i) \times [a_{j-1},a_j)$$for all $i,j \in [k]$.  Let $\cS_0^* \subseteq \cS_0$ be defined as 
	$$ \cS^*_0(\widehat f) := \argmin \left\{ \int_{[0,1]^2} \cQ_1 \left( g, \gamma^+, \gamma^-\right) : g \in \widehat f \right\},$$
	and for any $g \in \cS_0$, we define $\phi^{(g)} \in \cW_0$ as follows :
	 $$ \phi^{(g)}_t = \begin{cases}
	 		g, &\text{if } t \in  U_{\epsilon}(T) = [T-\epsilon,T], \\
	 	w^*, & \text{otherwise } 
	 \end{cases}$$
	Then $\cS^*_0(\widehat f)$ is non-empty and 
	$$ \widehat \cL^{\epsilon}(\widehat{f}) = \left\{\widehat{\text{Av}^{\epsilon}(\phi)} : \phi = \phi^{(g)} \text{ for some } g \in \cS^*_0(\widehat f) \right\}.$$
	Moreover, if $\gamma^{\pm},f $ are constant over $[0,1]^2$, then $\widehat \cL^{\epsilon}(\widehat{f})$ is a singleton set. 
\end{prop}

In other words, conditioned on the event that locally time-averaged graphon at time $T$ is close to $\widehat{f}$, the time-averaged sample path is close to  $\widehat{\text{Av}^{\epsilon}\left(\phi^{(g)}\right)}$, for some $g \in \cS_0^*(\widehat f)$. The sample path trajectory of the original  graph process is close to $\phi^{(g)}_{\bcdot}$, which takes the equilibrium value $w^*$ before time $T-\epsilon$ (thus incurring no cost) and takes some graph value $g \in \cS_0^*(\widehat f)$ on $[T-\epsilon, T]$ to minimize the cost while having local time-averaged graph at time $T$ to be $g \in \widehat f$. The proof of Proposition~\ref{prop:most-likely-path}  uses the same arguments as in~\cite[Theorem 2.8, Corollary 2.9]{Braunsteins2023} and hence omitted.

\subsubsection{Phase transition for the terminal observation}
For this subsection, we specialize to the case where the jump rates are constant, i.e., $\beta^{\pm} \equiv \lambda^{\pm} \in [c_{\beta}, \infty)$. In this case, the stationary state of the process is \erdos\;graph with parameter $p^*=\lambda^+/(\lambda^++\lambda^-)$. Consider an atypically large edge density $r \in (p^*,1)$. Condition on the event that the homomorphism density corresponding to a $d$-regular graph $F$ for the observed time-averaged graphon $\widehat{A^{n,\epsilon}_T}$ at time $T$ is at least $r^d$, it is natural to ask whether $\widehat{A^{n,\epsilon}_T}$ is close (in cut metric) to the constant graphon $\mathbf{r}$. The following result is the analogue of ~\cite[Theorem 2.1]{Braunsteins2023} in our set-up. 

\begin{thm}{\label{phase-trans}}
	Fix $\epsilon \in (0,T), r \in (p^*,1)$ and a finite $d$-regular graph $F$ with $|E(F)|$ many edges, with $d>1$.
	\begin{enumerate}[(a)]
		\item If the point $(r^d, \cQ_1(r,\lambda^+,\lambda^-))$ lies on the convex minorant of the function $u \mapsto  \cQ_1(u^{1/d},\lambda^+,\lambda^-)$ on $[0,1]$, then 
		$$ \lim_{n \to \infty} \dfrac{1}{a(n)n^2} \log \mathbb{P} \left( \ft(\widehat{A^{n,\epsilon}_T}, F) \geq r^{|E(F)|} \right) = - \epsilon  \cQ_1(r,\lambda^+,\lambda^-),$$
		and for every $\varepsilon$, there exists $C_1 >0 $ such that
		$$ \mathbb{P} \left( \delta_{\square} \left( \widehat{A^{n,\epsilon}_T}, \widehat{\mathbf{r}}\right) \geq \varepsilon \Big \rvert \ft(\widehat{A^{n,\epsilon}_T}, F) \geq r^{|E(F)|} \right) \leq e^{-C_1a(n)n^2}, \; \forall \; n \geq 1.$$ 
		\item If the point $(r^d, \cQ_1(r,\lambda^+,\lambda^-))$ does not lie on the convex minorant of the function $u \mapsto  \cQ_1(u^{1/d},\lambda^+,\lambda^-)$ on $[0,1]$, then 
		$$ \lim_{n \to \infty} \dfrac{1}{a(n)n^2} \log \mathbb{P} \left( \ft(\widehat{A^{n,\epsilon}_T}, F) \geq r^{|E(F)|} \right) > - \epsilon  \cQ_1(r,\lambda^+,\lambda^-),$$
		and there exists $\varepsilon, C_2 >0 $ such that
		$$ \mathbb{P} \left( \delta_{\square} \left( \widehat{A^{n,\epsilon}_T}, \widehat{\mathbf{r}}\right) \leq \varepsilon \Big \rvert \ft(\widehat{A^{n,\epsilon}_T}, F) \geq r^{|E(F)|} \right) \leq e^{-C_2a(n)n^2}, \; \forall \; n \geq 1.$$
	\end{enumerate}
	Here $\widehat{\mathbf{r}}$ is the equivalence class for the constant graphon $\mathbf{r}$.
\end{thm}

Similar to~\cite{Braunsteins2023}, Theorem~\ref{phase-trans} exhibits a phase transition in the behavior of $\widehat{A^{n,\epsilon}_T}$ as $r$ varies : We have the \textit{Symmetric (S)} phase where the graph is close in cut-metric to the constant graphon $\widehat{\mathbf{r}}$ and the \textit{Symmetry Breaking (SB)} phase where it is away from any uniform graphon. Note a small change in interpretation in Theorem~\ref{phase-trans} from ~\cite[Theorem 2.1]{Braunsteins2023}. Since we are focusing on the time-averaged graphon $\widehat{A^{n,\epsilon}_T}$ which is a weighted graph, it is more appropriate to consider whether this graph is close to constant graphon in cut metric or not, rather than to the corresponding \erdos\;graph; although the two are asymptotically equivalent since  large \erdos\;graphs are close in cut metric to a constant graphon. 
The decay rate here differs from that in \cite{Braunsteins2023}, which, as noted below Theorem \ref{thm:weakLDP}, is given as an action integral over space and time of the local Lagrangian for a two-state jump Markov process.
The key to the proof of Theorem~\ref{phase-trans} is the observation that $\widehat{A^{n,\epsilon}_T}$ satisfies LDP in the space $\left(\widehat{\cS_0}, \delta_{\square}\right)$ with speed $a(n)n^2$ and rate function at $\widehat f$ given by 
$$ \inf_{f \in \widehat f} \epsilon \int_{[0,1]^2} \cQ_1 (f, \lambda^+,\lambda^-), \; \forall \; \widehat f \in \widehat{\cS_0}.$$
This follows from proposition~\ref{prop:LDP-whole-avg-time-hom} with a easy time change argument. The rest of the proof is identical to that of ~\cite[Theorem 2.1]{Braunsteins2023} and, therefore, omitted.

\begin{rem}
	According to Theorem~\ref{phase-trans}, whether we are in S phase or in SB depends on the convex minorant of the function  $u \mapsto \cQ_1(u^{1/d},\lambda^+,\lambda^-)$. Different combination of degree parameter $d$ and rates $\lambda^{\pm}$ produce strikingly different behavior : The map $u \mapsto \cQ_1(u^{1/d},\lambda^+,\lambda^-)$ is either convex on $[0,1]$ or strictly concave on some $(u_1,u_2) \subset (0,1)$ and convex outside of it. Making use of bi-tangents to construct convex minorants, we can write down the following characterization. 

\textit{For any $p^* \in (0,1)$, there exists a critical degree parameter $d_{c}(p^*) \in \mathbb{N}$ such that for $1 < d \leq d_{c}$ we are in the S phase for any $r \in (p^*,1)$, whereas for $d > d_c$ there exists an interval $(r_1,r_2)$ with $p^* < r_1 < r_2 <1$ such that we are in the S phase for $r \in (p^*,r_1] \cup [r_2, 1)$ and in the SB phase for $r \in (r_1,r_2)$. In other words, we observe no phase transition for $1 < d \leq d_c$ and a re-entrant phase transition ($\text{S}\to \text{SB} \to \text{S}$) when $d > d_c$. The critical degree is given by 
	$$ d_c(p^*) := \max \left\{ d \in \mathbb{N}  : \dfrac{1}{2(d-1)} \geq \sup_{x \in [0,1]} (1-x)\left(2x-1-\dfrac{(2p^*-1)\sqrt{x(1-x)}}{\sqrt{p^*(1-p^*)}} \right)\right\}.$$}
For example, $d_c(1/2) = 5$.The formula for the critical dimension above also implies existence of a critical stationary edge density parameter $p_c \in (0,1/2)$ such that for $p^* \geq p_c$ we have $d_c(p^*) \geq 2$ and hence for small $d$ we observe no phase transition and for large $d$ we get a re-entrant phase transition; whereas for $p^*<p_c$ we get $d_c(p^*)=1$ and hence we observe re-entrant phase transition for any $d>1$. This critical $p_c$ is given by the unique solution of the following equation :
$$ \sup_{x \in [0,1]} (1-x)\left(2x-1-\dfrac{(2p-1)\sqrt{x(1-x)}}{\sqrt{p(1-p)}} \right) = \dfrac{1}{2}.$$
The  characterization can be proved using Theorem~\ref{phase-trans}, followed by a straightforward calculus exercise, hence omitted. These observations match quite nicely with the graphical information presented in~\cite[Figure 2,3]{Braunsteins2023} in the context of $d=2$. Since we are in the regime of rapid evolution, we should compare our observations to those in~\cite{Braunsteins2023} for large planning horizon $T$. The authors there also noted same re-entrant phase transition for small $p^*$, but no phase transition for large $p^*$. Similar properties for static \erdos\, random graphs were established in~\cite{Lubetzsky2015}.
\end{rem}

\subsection{Interacting particle systems defined on the random graphs}
In this section we consider an interacting particle system for which the interactions are governed by the random graph that is evolving with time according to the dynamics described before. While there is extensive literature on such systems, for the specific class of models considered in this paper—particularly those motivated by neuronal dynamics and incorporating techniques from dense graph limits—see, for example, \cite{medvedev2014nonlinear, medvedev2019continuum, chiba2019mean, Dupuis2022} for a more in-depth discussion of their motivation and relevance. For the simplicity of exposition, we consider the state space of the particles to be $\RR$, but the proofs easily extend to the case of $\RR^d$.    Each vertex $i \in [n]$ is viewed as a particle whose state evolves according to the following differential equation. 
\begin{equation}{\label{model:interact1}}
	\dfrac{\partial}{\partial t} u^{i,n}(t) = F_i^n\left(u^n_t,t \right) + \dfrac{1}{n} \sum_{j=1}^n X_{i,j}^n(t) D_{i,j}^n(u^n_t,t), \;\;	u_i^n(0) = z_i^n,\; t \in [0,T].
\end{equation}
Here, $u^{i,n} : [0,T] \to \mathbb{R}$ describes  the state (position) trajectory of the particle $i$ and $u^n : [0,T] \times [0,1] \to \mathbb{R}$ is the random field describing the path profiles for all the particles, defined as follows:
$$ u^n(t,x) := \sum_{i \in [n]} u^{i,n}(t) \mathbbm{1}_{Q_i^n}(x), \;\; u^n_t(\cdot) := u^n(t,\cdot)\; \; \forall \; (t,x) \in [0,T] \times [0,1].$$ The functions $F_i^n$ and and $D_{i,j}^n$ are maps from 
$L^2\left([0,1]\right) \times [0,T]$ to $\mathbb{R}$, the first 
describing the intrinsic dynamics for particle $i$ the second modeling the degree of pairwise interactions between the particles $i$ and $j$. 
The connectivity of the pairwise interaction graph at time $t$ is given by the graph $G^n(t)$ defined on the vertex set $[n]$ which has the edge connecting $i$ to $j$ (for $i,j \in [n]$) if and only if $X_{i,j}^n(t)=1$.
Finally, $z_i^n$ denotes the initial state of the particle $i$. This model allows the intrinsic dynamics and interactions between the particles to depend not only on their current positions, but also on the particles themselves. 

As $n\to \infty$, the system of equations in \eqref{model:interact1} formally leads to continuum equations of the form 
\begin{align}
	\dfrac{\partial}{\partial t} v(t,x) &= F\left(v_t,t,x \right) + \int_0^1 W\left(x,y,t\right) D(v_t,t,x,y)\;dy, \; \;  \; (t,x) \in [0,T] \times [0,1],\label{model:interact2}\\
	v(0,x) &= z(x), \; \;  \; x \in [0,1], \label{model:interact3}
\end{align}
where 
$$ F : L^2\left([0,1]\right) \times [0,T] \times [0,1] \to \mathbb{R}, \; \; D : L^2\left([0,1]\right) \times [0,T] \times [0,1]^2 \to \mathbb{R},$$
represent the intrinsic dynamics and interactions of the system respectively, whereas $W \in \cW_0$ is a graphon process representing the asymptotic connectivity of the system. $z\in L^2[0,1]$ is the initial state of the system. As before,  we write $v_t(\cdot) := v(t,\cdot)$, for $t \in [0,T]$. We now introduce conditions that will ensure the well-posedness of the equations in \eqref{model:interact1} and \eqref{model:interact2}.

\begin{assump}{\label{assump:FD}}
	We say that $(F,D)$ satisfies the regularity conditions with parameter $L=L(F,D) \in (0,\infty)$, if the following hold.
	\begin{enumerate}[(i)]
		\item For any $t \in [0,T], (x,y) \in [0,1]^2$ and $\psi \in L^2\left([0,1]\right)$, we have 
		$$ |F(\psi,t,x)|, |D(\psi,t,x,y)| \leq L.$$
		\item For any $t \in [0,T], (x,y) \in [0,1]^2$ and $\psi,\psi^{\prime} \in L^2\left([0,1]\right)$, we have 
		$$ |F\left(\psi,t,x\right) - F\left(\psi^{\prime},t,x\right)| \leq L \|\psi - \psi^{\prime}\|_2,$$
		and 
		$$ |D\left(\psi,t,x,y\right) - D\left(\psi^{\prime},t,x,y\right)| \leq L \|\psi - \psi^{\prime}\|_2.$$
		\end{enumerate}
\end{assump}

We now introduce the space in which we consider solutions to \eqref{model:interact2}. For any $K \in (0,\infty)$, let $L^2\left([0,1],K\right)$ denote the closed ball of radius $K$ in $L^2\left([0,1]\right)$ and  $L_w^2\left([0,1],K\right)$ be the same $L^2\left([0,1],K\right)$ space but equipped with the weak topology. Clearly $L_w^2\left([0,1],K\right)$ is a compact space. Let $C\left([0,T];L_w^2\left([0,1],K\right)\right)$ be the set of all measurable functions $v :[0,T]\times [0,1]\to \mathbb{R}$, with $v_t \in L^2_w\left([0,1],K\right)$ for all $t$ and such that the map $[0,T] \ni t \mapsto v_t \in L^2_w\left([0,1],K\right)$ is continuous.
We equip $C\left([0,T];L_w^2\left([0,1],K\right)\right)$ with the norm 
$$ \big \lVert  v \big \rVert_{C\left([0,T];L_w^2\left([0,1],K\right)\right)} := \sup_{t \in [0,T]} \| u_t \|_{w,K}.$$
Here $\| \cdot \|_{w,K}$ is a norm on $L^2 \left( [0,1],K\right)$ which induces the weak topology on  it. Without loss of generality, we take 
$$ \| \psi \|_{w,K} := \sum_{k \geq 1} 2^{-k} \bigg \rvert \int_{[0,1]} \psi \varphi_k \bigg \rvert,$$
where $\left\{ \varphi_k : k \geq 1\right\}$ is an orthonormal basis of $L^2\left([0,1]\right)$. Note that, $\| \psi \|_{w,K} \leq \|\psi\|_2$ for any $\psi \in L^2\left([0,1],K\right)$. The following result is an immediate consequence of the Arzel\`a--Ascoli theorem for metric-space-valued functions
(e.g. \cite[Theorem~7.2]{billingsley1999convergence}).
\begin{prop}{\label{compact:lipschitz}}
	Fix any $K, M \in (0,\infty)$ and let $C_M\left([0,T];L_w^2\left([0,1],K\right)\right)$ be the subset of the space $C\left([0,T];L_w^2\left([0,1],K\right)\right)$ containing those $v$ which satisfies $ \|v_t-v_s\|_2 \leq M|t-s|$, for all $s,t \in [0,T]$. Then $C_M\left([0,T];L_w^2\left([0,1],K\right)\right)$ is a compact metric space.
\end{prop}
The next result gives unique solvability of \eqref{model:interact2}-\eqref{model:interact3}.
\begin{prop}{\label{prop:unique}}
For any $z \in L^2\left([0,1],K\right)$ and $(F,D)$ satisfying the regularity conditions stated in Assumption~\ref{assump:FD}, equation \eqref{model:interact2}, with initial condition given by \eqref{model:interact3}, has an unique solution in $C\left([0,T];L_w^2\left([0,1], K+2LT\right)\right)$. Moreover, the unique solution lies in $C_{2L}\left([0,T];L_w^2\left([0,1], K+2LT\right)\right)$. 
\end{prop}

The proof of the above proposition is given in Section~\ref{appn}. Proposition~\ref{prop:unique} guarantees that \eqref{model:interact2} is uniquely solvable for any initial state $z \in L^2\left([0,1]\right)$. To connect \eqref{model:interact2} to the system of equations in \eqref{model:interact1}, we define for all $\psi \in L^2\left([0,1]\right)$ and $(t,x,y) \in [0,T] \times [0,1]^2$, 
\begin{equation}{\label{def:FDn}}
	F^n(\psi,t,x) := \sum_{i=1}^n F_i^n(\psi,t) \mathbbm{1}_{Q_i^n}(x), \;D^n(\psi,t,x,y) := \sum_{i,j=1}^n D_{i,j}^n(\psi,t)\mathbbm{1}_{Q_{i,j}^n}(x,y),
\end{equation}  
and 
$$z^n(x) := \sum_{i=1}^n z_i^n\mathbbm{1}_{Q_i^n}(x).$$
Note that for each $n$, $z^n \in L^2\left([0,1]\right)$.
The system of equations in \eqref{model:interact1} can be rewritten as 
\begin{align}
	\dfrac{\partial}{\partial t} u^n(t,x) &= F^n\left(u^n_t,t,x \right) + \int_0^1 H^n\left(x,y,t\right) D^n(u^n_t,t,x,y)\;dy, \; \;  \; (t,x) \in [0,T] \times [0,1],\label{model:interact4}\\
	u^n(0,x) &= z^n(x), \; \;  \; x \in [0,1]. \label{model:interact5}
\end{align}
Since $H^n \in \cW_0$ for each $n$, we have from Proposition~\ref{prop:unique}, that if $(F^n, D^n)$ satisfy 
 the regularity conditions in Assumption~\ref{assump:FD}, then there exists an unique solution $u^n$ to \eqref{model:interact4}-\eqref{model:interact5} (equivalently, for the system \eqref{model:interact1}).

 We will need the following assumption for the next two results.

\begin{assump}{\label{assump:FDconv}} 
	Suppose that $(F^n,D^n)$ and $(F,D)$ satisfy Assumption~\ref{assump:FD} with parameter $L \in (0,\infty)$. Moreover, suppose that
\begin{enumerate}[(i)]
		\item \label{assump:FDconv1} For any $K \in (0,\infty)$, as $n \to \infty$, we have 
		$$ d_1(F^n,F; K) :=  \int_0^T \int_0^1 \sup_{\psi \in L^2\left([0,1],K\right)} \big \rvert F^n(\psi,t,x) - F(\psi,t,x)\big \rvert\, dx\,dt \longrightarrow 0,$$  
		$$ d_2(D^n,D; K) :=  \int_0^T \int_{[0,1]^2} \sup_{\psi \in L^2\left([0,1],K\right)}  \big \rvert D^n(\psi,t,x,y) - D(\psi,t,x,y)\big \rvert\, dx\,dy\,dt \longrightarrow 0.$$  
		\item  \label{assump:FDconv2}   If $\psi_n \in L^2\left([0,1]\right)$ converges weakly to $\psi \in L^2\left([0,1]\right)$, then 
		\begin{equation*}
			F(\psi_n,t,x) \longrightarrow F(\psi,t,x), \; \text{a.e.}[t,x], \;\;\;
			D(\psi_n,t,x,y) \longrightarrow D(\psi,t,x,y), \; \text{a.e.}[t,x,y].
		\end{equation*}
	\end{enumerate}
\end{assump}

The following theorem establishes an LDP for the sequence $\left\{u^n : n \geq 1\right\}$. Note that, we allow the initial data to be random (but independent of the graph process).
\begin{thm}{\label{thm:interactldp}}
	Suppose that Assumption~\ref{assump:FDconv} holds and the sequence $u^n$  is given by \eqref{model:interact4}, with initial conditions given by \eqref{model:interact5}, where $z^n$ is independent of $\left\{H^n_t : t \in [0,T]\right\}$, for every $n$. Suppose, for some $K<\infty$,
    $z^n \in L^2\left([0,1],K\right)$ for each $n$ and $\left\{z^n : n \geq 1\right\}$ satisfies an LDP in $L_w^2\left([0,1],K\right)$ with speed $a(n)n^2$ and rate function $I_{\text{initial}} : L_w^2\left([0,1],K\right) \to [0,\infty]$. Moreover suppose that  Assumption~\ref{assmp:2}\eqref{assump:1A} is satisfied. Then $\left\{u^n : n \geq 1\right\}$, as a sequence of random variables in $C_{2L}\left([0,T]; L_w^2\left([0,1],K+2LT\right)\right)$, satisfies an LDP with speed $a(n)n^2$ and rate function $I_{\text{solution}} : C_{2L}\left([0,T]; L_w^2\left([0,1],K+2LT\right)\right) \to [0,\infty]$, defined as follows :
	$$ I_{\text{solution}}(v) := \inf_{\left(\phi,z\right) \in \mathcal{C}(v)} \left[ J^{(\beta^+, \, \beta^-)}\left(\phi \right) + I_{\text{initial}}(z)\right], $$
	where $J^{(\beta^+, \, \beta^-)}$ is as in Theorem \ref{thm:weakLDP} and $\mathcal{C}(v)$ is the set of all $(\phi,z) \in \cW_0 \times L_w^2\left([0,1],K\right)$ such that $v$ solves the equation:
	\begin{align}
		\dfrac{\partial}{\partial t} v(t,x) &= F\left(v_t,t,x \right) + \int_0^1 \phi\left(x,y,t\right) D(v_t,t,x,y)\;dy, \; \;  (t,x) \in [0,T] \times [0,1],\label{model:interact6}\\
		v(0,x) &= z(x), \; \;  x \in [0,1]. \label{model:interact7}
	\end{align}
\end{thm}
The proof of Theorem~\ref{thm:interactldp}  is given in Section~\ref{appn}. 
\begin{rem}
\label{rem:3.21}
In \cite[Theorem 4.3]{Dupuis2022}, an analogous large deviation principle for dynamical systems, in which interactions between particles are governed by a {\em static} random graph, has been studied. Our work allows for a general class of time-varying random graphs with jump rates that can be inhomogeneous in both time and space. In addition, Theorem \ref{thm:interactldp}
     allows the intrinsic dynamics and interactions between the particles to depend not only on their current positions (as was the case in the model analyzed in \cite{Dupuis2022}), but also the particles themselves, the path profile of all the particles at that time and the time itself. We note, however, that our LDP is given in a weaker topology, in that we consider $L_w^2\left([0,1],K\right)$ instead of $L^2[0,1]$. This shortcoming is a consequence of the fact that we  only established an LDP for the locally time-averaged graphon process rather than the original graphon sample path in the  cut-metric topology (see also Remark~\ref{rem:cutpathLDP}).
\end{rem}

\begin{rem}
\label{rem:3.22}
In this work we considered the large deviation behavior of dynamical systems with interactions governed by certain types of dense dynamic random networks.  Stochastic dynamical systems with interactions given by certain sparse random network models were studied in \cite{bhamidi2019weakly} (see Example 2.2 therein) and the law of large numbers for the empirical measures of the states of the stochastic dynamical systems was established. It would be interesting to see if the methods from the current work can be used to prove large deviation results from the LLN behavior characterized in \cite{bhamidi2019weakly}.
\end{rem}

\begin{example}
	One basic example of $(F,D)$ that satisfies Assumption~\ref{assump:FD} and the condition~\ref{assump:FDconv2} in Assumption~\ref{assump:FDconv}, is the following. Let $\kappa_i \in L^2\left([0,1]^2\right)$, $i=1,2,3$ be such that 
    $$ \sup_{x \in [0,1]} \| \kappa_i(\cdot,x)\|_2 < \infty.$$ 
	Let $D_0:  \RR\times \RR \times [0,T]  \to \RR$ and $F_0:  \RR \times [0,T]  \to \RR$ be such that for some $C=C(D_0,F_0) \in (0,\infty)$
	$$ |D_0(x,y,t) - D_0(x',y',t)| + |F_0(x,t)- F_0(x',t)| \leq C(D_0, F_0)  (|x-x'| + |y-y'|),$$
	and 
	$$ |D_0(x,y,t)|+ |F_0(x,t)| \leq C(D_0, F_0),$$
	for any $x,x^{\prime},y,y^{\prime} \in \RR$ and $t \in [0,T]$. 
	For $\psi \in L^2\left([0,1]\right)$, define $\psi\ast \kappa_i: [0,1] \to \RR$ as
	$$\psi \ast \kappa_i(x) = \int_0^1 \psi(x') \kappa_i(x',x) dx', \; x \in [0,1].$$
	Note that $\psi*\kappa_i(x)$ is finite for all $x$ and $\psi*\kappa_i \in L^2\left([0,1]\right)$. Moreover, 
    $$ \|\psi*\kappa_i(x) - \psi^{\prime}*\kappa_i(x)\|_{\infty} \leq \| \psi-\psi^{\prime}\|_2 \left( \sup_{x \in [0,1]} \|\kappa_i(\cdot,x) \|_2\right).$$
    for any $\psi,\psi^{\prime} \in L^2\left([0,1]\right)$. Finally, if $\psi_n$ converges weakly to $\psi$ in $L^2\left([0,1]\right)$, then $\psi^n*\kappa_i$ converges pointwise to $\psi*\kappa_i$. Define for all $x,y \in [0,1], \psi \in L^2\left([0,1]\right), t\in [0,T],$ 
	$$D(\psi, x,y,t) = D_0(\psi * \kappa_1(x), \psi\ast \kappa_2(y),t), \;\; F(\psi,x,t) = F_0(\psi\ast \kappa_3(x), t).$$
	Then it follows that $D$ and $F$ satisfy Assumption~\ref{assump:FD} and the condition~\ref{assump:FDconv2} in Assumption~\ref{assump:FDconv}.
\end{example}

\begin{example}
	A simple example of  a sequence $\left\{(F^n,D^n) : n \geq 1\right\}$ of the form~\eqref{def:FDn} and satisfying \ref{assump:FDconv1} in Assumption~\ref{assump:FDconv} is as follows. Assume that $(F,D)$ satisfies Assumption~\ref{assump:FD}. Moreover, assume that for any $t \in [0,T], (x,y), (x^{\prime},y^{\prime}) \in [0,1]^2$ and $\psi \in L^2\left([0,1]\right)$, we have 
	$$ |F\left(\psi,t,x\right) - F\left(\psi,t,x^{\prime}\right)| \leq C |x-x^{\prime}|$$
    and 
    $$ |D\left(\psi,t,x,y\right) - D\left(\psi,t,x^{\prime},y^{\prime}\right)| \leq C \left(|x-x^{\prime}| + |y-y^{\prime}|\right),$$
	for some $C \in (0,\infty)$. For any $\psi \in L^2\left([0,1]\right)$ and $t \in [0,T]$, define, 
	$$ F_i^n(\psi,t) := n\int_{Q_i^n} F(\psi,t,x)\,dx,\;\; D_{i,j}^n(\psi,t) := n^2 \int_{Q_{i,j}^n} D(\psi,t,x,y)\, dx\,dy,$$
	and then set $F^n$ and $D^n$ according to~(\ref{def:FDn}). In other words, $(F^n,D^n)$ is the block approximation of $(F,D)$. It is immediate that $(F^n,D^n)$ satisfies Assumption~\ref{assump:FD}. Moreover, 
	\begin{align*}
		\Big \rvert F^n(\psi,t,x) - F(\psi,t,x)\Big \rvert &= \sum_{i=1}^n  \Big \rvert F_i^n(\psi,t) - F(\psi,t,x)\Big \rvert \mathbbm{1}_{Q_i^n}(x) \\
		&\leq \sum_{i=1}^n  n\int_{Q_i^n} \Big \rvert F(\psi,t,x^{\prime}) - F(\psi,t,x)\Big \rvert\, dx^{\prime} \mathbbm{1}_{Q_i^n}(x) \\
		& \leq Cn\sum_{i=1}^n  \int_{Q_i^n} \big \rvert x^{\prime} - x\big \rvert\, dx^{\prime} \mathbbm{1}_{Q_i^n}(x) \leq C/n,
	\end{align*}
	for any $\psi,t$. Thus $d_1(F^n,F;K) \to 0$ for any $K\in(0,\infty)$. Similar estimates also show that $d_2(D^n,D;K) \to 0$.
\end{example}

\subsection{Organization of Proofs}
\label{sec:orgpf}
 The remaining paper is organized as follows. In Section \ref{lln} we prove the main LLN results: Theorems \ref{cor2:thmlin}, \ref{thm:llnnew}, and Proposition 
 \ref{prop:lln-avg-path}. Section \ref{ldp} contains the proofs of 
 Propositions \ref{prop:no-ldp-weak}, \ref{prop:ldp-point-weak}, Corollary~\ref{cor:weakLDP-path}, and Theorem \ref{thm:weakLDP}    giving results on  large deviation principles in the weak topology. Specifically, the proof of the latter result
 is completed in Sections \ref{sec:isrf}--\ref{sec:pflowbd}.  Section \ref{sec:cutldp} lifts the LDP in the weak topology to an LDP in the cut metric. This section contains proofs of Theorem~\ref{thm:LDP-avg-path} and Proposition~\ref{prop:LDP-avg-path-finite}.
Section \ref{sec:applications} contains the proofs of the results on explicit applications of the LDP results to some rare events given in Section \ref{sec:apps}.
  Proof of the LDP for the dynamical systems stated in Theorem \ref{thm:interactldp} is given in Section \ref{appn}.  Finally, Section \ref{sec:aux} collects some auxiliary results. A more detailed organization of each section is provided at the beginning of that section.

\section{Law of Large Numbers}{\label{lln}}
   In this section we prove Theorem~\ref{cor2:thmlin}, Theorem~\ref{thm:llnnew} and Proposition~\ref{prop:lln-avg-path}. We start with the following technical lemmas. The first is a standard concentration estimate; we provide a proof for completeness.
	
	\begin{lem}{\label{es11}}
		Let $\left\{B_{i,j} : i,j \in [n]\right\}$ are independent $[0,1]$-valued random variables and let 
		$$ T^n(x,y) := \sum_{i,j=1}^n B_{i,j} \mathbbm{1}_{Q_{i,j}^n}(x,y), \;\; \E T^n := \sum_{i,j=1}^n \E(B_{i,j}) \mathbbm{1}_{Q_{i,j}^n}(x,y), \;\; x,y \in [0,1].$$
		Then, for any $\delta >0$, we have 
		$$ \mathbb{P} \left[ d_{\infty \to 1} \left(T^n, \E T^n \right) \geq \delta \right] \leq 2^{2n} \exp\left(-n^2\delta^2/2\right).$$
	\end{lem}

	\begin{proof}
			From the definition of the cut metric $d_{\infty \to 1}$, it is immediate that 
		$$ d_{\infty \to 1}\left(T^n,\E T^n \right) = \sup_{a,b \in  \left\{-1,1\right\}^n} \frac{1}{n^2} \sum_{i,j=1}^n a_ib_j \left(B_{i,j} - \E B_{i,j}\right).$$
		For any $a,b \in \left\{-1,1\right\}^n$ and $i,j \in [n]^2$, the random variable $a_ib_j\left(B_{i,j}-\E B_{i,j}\right)$ lies in $[-1,1]$ and hence it is sub-Gaussian with variance factor $1$, see \cite[Chapter 2.3]{boucheron} for definition and properties of sub-Gaussian variables. The collection $\left\{B_{i,j} : i,j \in [n]\right\}$ being independent, we have $\sum_{i,j} a_ib_j\left(B_{i,j}-\E B_{i,j}\right)$ to be sub-Gaussian with variance factor $n^2$. Hence,
		$$ \mathbb{P} \left( \dfrac{1}{n^2} \sum_{i,j} a_ib_j\left(B_{i,j}-\E B_{i,j}\right) \geq \delta \right) \leq \exp\left(-n^2\delta^2/2\right),$$
		and therefore, union bound implies
		\begin{equation*}
			\mathbb{P} \left(d_{\infty \to 1}\left(T^n,\E T^n \right) \geq \delta\right) \leq 2^{2n}\exp\left(-n^2\delta^2/2\right).
		\end{equation*} 
        The result follows.
	\end{proof}
	The next lemma gives the convergence of the expected value of the graphon process.
	\begin{lem}{\label{lem:es12}}
	Define $\mu^n \in \cW_0$ as
		$$ \mu^n_t := \E H_t^n  = \sum_{i,j=1}^n \E X_{i,j}^n (t) \, \mathbbm{1}_{Q_{i,j}^n} := \sum_{i,j=1}^n P_{i,j}^n(t) \, \mathbbm{1}_{Q_{i,j}^n},\;\; t \in [0,T],$$
		where $P_{i,j}^n(t) := \E X_{i,j}^n(t).$ 
        Then, as $n\to \infty$, $\mu^n$ converges  to $w$ in $L^2\left([0,1]^2_T\right)$. 
	\end{lem}
	
		\begin{proof}
        By a standard subsequential argument it suffices to show that every subsequence of $\{n\}$ has a further subsequence along which $\mu^n$ converges to $w$ in $L^2\left([0,1]^2_T\right)$. Now fix a subsequence and using Assumption \ref{assmp:1}\eqref{assump:1Aweak} choose a further subsequence (denoted again as $\{n\}$) along which
        \begin{equation} \label{eq:1135} \int_0^T |\beta_s^{n,\pm}(x,y) - \beta_s^{\pm}(x,y)| ds \to 0 \mbox{ for a.e. } (x,y) \in [0,1]^2.\end{equation}
Since $\mu^n=\mathbb E H^n$ and $w$ both take values in $[0,1]$, we have
\[
|\mu^n_s(x,y)-w_s(x,y)|^2
\leq
|\mu^n_s(x,y)-w_s(x,y)|,
\]
and therefore, using dominated convergence applied over $(x,y)\in[0,1]^2$, we see that it suffices to show that along the above subsequence
        \begin{equation}\label{eq:1248}
        \int_0^T |\mu^n_s(x,y) - w_s(x,y)| ds \to 0 \mbox{ for a.e. } (x,y) \in [0,1]^2.
        \end{equation}

        Kolmogorov equation for the Markov process $\{X_{i,j}^n(t) : t \in [0,T]\}$ says that
			\begin{align*}
				P_{i,j}^n(t) &= P_{i,j}^n(0) + a(n)\int_{0}^t \beta^{n,+}_{i,j}(s)(1- P_{i,j}^n(s))\, ds - a(n)\int_{0}^t \beta^{n,-}_{i,j}(s) P_{i,j}^n(s)\, ds. 
			\end{align*}
            Equivalently,
			\begin{equation}{\label{boundP}}
				P_{ij}^n(t) = \exp(-a(n)B^{n}_{i,j}(t)) \left[ \int_0^t a(n)\beta_{i,j}^{n,+}(s)\exp\left(a(n)B^n_{i,j}(s)\right)ds + P_{i,j}^n(0)\right],
			\end{equation} 
			where 
			$$ B_{i,j}^n(t) := \int_{0}^t \left(\beta_{i,j}^{n,+}(s) + \beta_{i,j}^{n,-}(s)\right) ds,
            $$
            Let, for $t \in [0,T]$,
            $$B^n_t := \sum_{i,j=1}^n B^{n}_{i,j}(t)\mathbbm{1}_{Q_{i,j}^n} = \int_0^t \left( \beta^{n,+}_s+\beta^{n,-}_s\right)ds, \; \; 
             B_t:= \int_0^t \left( \beta^{+}_s+\beta^{-}_s\right)ds.
            $$ 
			Then we can rewrite \eqref{boundP} as 
			\begin{align}{\label{boundP1}}
				\E H_t^n = \exp \left( -a(n)B^n_t\right) \left[ \int_0^t a(n)\beta^{n,+}_s\exp\left(a(n)B^n_s\right)ds +  H_0^n\right].
			\end{align}
			From Assumption \ref{assmp:1}\eqref{assump:1B}, we have $B_t^n \geq c_{\beta}t$ and hence $\exp \left( -a(n)B^n_t\right) H_0^n$ converges to $0$, for all $(x,y,t) \in [0,1]^2 \times [0,T]$, as $n \to \infty$. The remaining term of \eqref{boundP1} can be written as 
			\begin{align}\label{eq:1136}
				&\exp \left( -a(n)B^n_t\right) \int_0^t a(n)\beta^{n,+}_s\exp\left(a(n)B^n_s\right)ds \nonumber \\
                &= \int_0^{\infty} \beta^{n,+}_{t-u/a(n)} \exp \left[ -a(n)\left(B_t^n - B_{t-u/a(n)}^n\right) \right]du, 
			\end{align}
			where $\beta^{n,\pm}_u, \beta^{\pm}_u, B^n_u$ and $B_u$ are defined to be the zero function on $[0,1]^2$ whenever $u <0$. 
            
           Using Assumption \ref{assmp:1}\eqref{assump:1B}, for a.e. $(x,y)$,
            \begin{align}
            &\int_0^T \int_0^t a(n) |\beta^{n,+}_s - \beta^{+}_s| \exp\{-a(n)(B^n_t - B^n_s)\} ds dt \nonumber \\
            &\le \int_0^T \int_0^t a(n) |\beta^{n,+}_s - \beta^{+}_s| \exp\{-a(n)c_{\beta}(t-s)\} ds dt\nonumber\\
            &= \int_0^T |\beta^{n,+}_s - \beta^{+}_s|\int_s^T a(n)  \exp\{-a(n)c_{\beta}(t-s)\} dt ds\nonumber\\
            & \le \frac{1}{c_{\beta}} 
            \int_0^T |\beta^{n,+}_s - \beta^{+}_s| ds \to 0, \mbox{ as } n \to \infty,  \label{eq:1244}
            \end{align}
            where in the last line we have used \eqref{eq:1135}.  
            
            Also, using a similar calculation as \eqref{eq:1136}, and denoting $\sup_{t\in[0,T]} \beta^+_t(x,y) = c_1(x,y)$ (which is finite a.e. from Assumption \ref{assmp:1}\eqref{assump:1D}),
            \begin{align}
&\int_0^T a(n)\left| \int_0^t [\exp\{-a(n)(B^n_t-B^n_s)\} - \exp\{-a(n)(B_t-B_s)\}]\beta^+_s \,ds\right| dt \nonumber \\
&=\int_0^T \left| \int_0^{ta(n)} \beta^+_{t-u/a(n)} \left[e^{-a(n)\left(B^n_t-B^n_{t-u/a(n)}\right)} - e^{-a(n)\left(B_t-B_{t-u/a(n)}\right)}\right] du \right| dt\nonumber\\
& \le c_1(x,y) \int_0^T \int_0^{ta(n)} e^{-c_{\beta} u} a(n) \int_{t-u/a(n)}^t \left(|\beta^{n, +}_s - \beta^{+}_s| + |\beta^{n, -}_s - \beta^{-}_s|\right) ds du  dt\nonumber\\
& \leq c_1(x,y) \int_0^T \int_0^{\infty} \int_{s}^{T \wedge (s+u/a(n))} a(n)e^{-c_{\beta}u}\, \left(|\beta^{n, +}_s - \beta^{+}_s| + |\beta^{n, -}_s - \beta^{-}_s|\right) dt du  ds\nonumber\\
& \le c_1(x,y) \int_0^T \int_0^{\infty}  ue^{-c_{\beta}u}\, \left(|\beta^{n, +}_s - \beta^{+}_s| + |\beta^{n, -}_s - \beta^{-}_s|\right) du  ds\nonumber\\
&=c_1(x,y) \left(\int_0^{\infty} ue^{-c_{\beta} u} du\right) \int_0^T \left(|\beta^{n, +}_s - \beta^{+}_s| + |\beta^{n, -}_s - \beta^{-}_s|\right) ds \to 0, \mbox{ as } n \to \infty, \label{eq:1245}
\end{align}
for a.e. $(x,y)$, where in the last line we have once more used \eqref{eq:1135}.

Finally,
\begin{align*}
\int_0^t a(n) \exp\{-a(n) (B_t-B_s) \} \beta^+_s ds &= \int_0^{\infty} \beta^{+}_{t-u/a(n)} \exp \left[ -a(n)\left(B_t - B_{t-u/a(n)}\right) \right]du.
\end{align*}
By Assumption \ref{assmp:1}\eqref{assump:1D}, we have the following a.e. $(x,y) \in [0,1]^2$ : For all $t \in (0,T]$, $\beta^{+}_{t-u/a(n)} \to \beta^{+}_{t}$, as $n \to \infty$ and 
$$ a(n) \left(B_t - B_{t-u/a(n)}\right) \longrightarrow u\beta^{+}_t + u\beta^-_t. $$
Also, $\beta^{+}_{t-u/a(n)} \exp \left[ -a(n)\left(B_t - B_{t-u/a(n)}\right) \right] \le c_1(x,y) \exp\{-c_{\beta}u\}$. Thus, by the dominated convergence theorem
\begin{equation}{\label{eq:mu-t-es1}}
	\int_0^t a(n) \exp\{-a(n) (B_t-B_s) \} \beta^+_s ds \to w_t, \forall \; t \in (0,T], \text{ a.e. } (x,y) \in [0,1]^2.
\end{equation}
Moreover, for any $n \geq 1$,
\begin{align*}
    \int_0^t a(n) \exp\{-a(n) (B_t-B_s) \} \beta^+_s ds & \leq \int_0^t a(n)\beta_s^+ \exp \left(-a(n) \int_s^t \beta_u^+ du \right) ds \\
    &= 1 - \exp \left( -a(n) \int_0^t \beta_u^+du \right) \leq 1.
\end{align*}   
Another application of the dominated convergence theorem now shows that, as $n\to \infty$,
$$\int_0^T \left|\int_0^t a(n) \exp\{-a(n) (B_t-B_s) \} \beta^+_s ds - w_t\right| dt \to 0 \mbox{ for a.e. } (x,y) \in [0,1]^2.$$
Combining the above with \eqref{eq:1244} and \eqref{eq:1245} we now have that \eqref{eq:1248} holds which completes the proof of the lemma.
\end{proof}
			
Combining the last two lemmas we can now complete the proof of 	Theorem~\ref{thm:llnnew}.	

\begin{proof}[Proof of Theorem~\ref{thm:llnnew}]
Recall $\mu_t^n$ from Lemma~\ref{lem:es12}.
\begin{align*}
 \int_0^1 d_{\square} \left(\mu^n_t, w_t \right) \, dt &\leq \int_0^1 \| \mu^n_t - w_t \|_1  \, dt =  \| \mu^n - w \|_1 \leq  \| \mu^n - w \|_2.
\end{align*}
From Lemma~\ref{lem:es12}, the last term in the above display converges to $0$. It is, therefore, enough to show that 
\begin{equation}{\label{ts:new1}}
 \int_{0}^{1} d_{\infty \to 1} \left(H^{n}_t,\,\mu^{n}_t\right)\, dt \stackrel{a.s.}{\longrightarrow} 0,
\end{equation}
as $n \to \infty$. Since the metric $d_{\infty \to 1}$ takes value between $0$ and $1$ on $\cS_0$, we have the following for any $\delta >0$.
$$ \int_{0}^{1} d_{\infty \to 1} \left(H^{n}_t,\,\mu^{n}_t\right)\, dt \leq \int_{0}^{1} \mathbbm{1}_{\left(d_{\infty \to 1} \left(H^{n}_t,\,\mu^{n}_t\right)\geq  \delta \right)} dt + \delta,$$
whereas, 
            \begin{align}
				 \mathbb{P} \left( 	\int_{0}^{1} \mathbbm{1}_{\left(d_{\infty \to 1} \left(H^{n}_t,\,\mu^{n}_t\right)\geq  \delta \right)} dt > \delta \right) 
				& \leq \dfrac{1}{\delta} \int_{0}^1 \mathbb{P} \left(d_{\infty \to 1} \left(H^{n}_t,\,\mu^n_t\right) \geq \delta \right)\, dt.\label{es1}
			\end{align} 
			Applying Lemma~\ref{es11} in \eqref{es1}, we obtain that for any $t \in [0,1]$, 
            $$\mathbb{P} \left(d_{\infty \to 1} \left(H^{n}_t,\,\mu^n_t\right) \geq \delta \right) \leq 2^{2n} \exp\left(-n^2\delta^2/2 \right)$$ and hence,
            \begin{align*}
            \sum_{n \geq 1} \mathbb{P} \left( 	\int_{0}^{1} d_{\infty \to 1} \left(H^{n}_t,\,\mu^{n}_t\right) dt \geq 2\delta \right) &\leq \sum_{n \geq 1} \mathbb{P} \left( 	\int_{0}^{1} \mathbbm{1}_{\left(d_{\infty \to 1} \left(H^{n}_t,\,\mu^{n}_t\right) \geq  \delta \right)} dt \geq \delta \right)  \\
            &\leq \dfrac{1}{\delta}\sum_{n \geq 1} 2^{2n}\exp\left(-n^2\delta^2/2\right) < \infty.
				\end{align*}
				This shows \eqref{ts:new1}, which completes the proof. 
            \end{proof}
	
	\begin{proof}[Proof of Proposition~\ref{prop:lln-avg-path}]
    We start by observing that 
	\begin{align*}
	&	\sup_{t \in [0,T]} d_{\infty \to 1} \left( A^{n, \epsilon}_t, w^{\epsilon}_t\right) \\
		&\leq \sup_{t \in [0,T]} \dfrac{1}{|U_{\epsilon}(t)|}\sup_{\substack{a,b \\ \|a\|_{\infty}, \|b\|_{\infty} \leq 1}} \int_{[0,1]^2} a(x)b(y) \left[ \int_{U_{\epsilon}(t)} \left(H^n_s(x,y)\, ds - w_s(x,y) \right)ds  \right]\, dx \, dy \\
		& \leq  \sup_{t \in [0,T]} \dfrac{1}{|U_{\epsilon}(t)|}  \int_{U_{\epsilon}(t)} \left[\sup_{\substack{a,b \\ \|a\|_{\infty}, \|b\|_{\infty} \leq 1}} \int_{[0,1]^2} a(x)b(y) (H_s^n(x,y)-w_s^n(x,y))\, dx \, dy \right] ds \\
		&=   \sup_{t \in [0,T]} \dfrac{1}{|U_{\epsilon}(t)|}  \int_{U_{\epsilon}(t)}  d_{\infty \to 1} \left(H^n_s, w^n_s \right) ds \leq \dfrac{1}{\epsilon} \int_0^1 d_{\infty \to 1} \left(H^n_s, w^n_s \right) ds.
	\end{align*}
        The last term in the above display converges almost surely to $0$ by Theorem~\ref{thm:llnnew}. This completes the proof.
			\end{proof}
			
			\begin{proof}[Proof of Theorem~\ref{cor2:thmlin}]
				By Lemma~\ref{es11}, for any $\delta>0$, 
				$$ \sum_{n \geq 1} \mathbb{P} \left(d_{\infty \to 1} \left(H^{n}_s,\,\E H^n_s\right) \geq \delta\right) \leq \sum_{n \geq 1} 2^{2n}\exp\left(-n^2\delta^2/2\right).$$
			Therefore, we have $d_{\infty \to 1}\left(H_s^n, \E H_s^n \right)$ converges to $0$ almost surely for any $s \in [0,T]$. This shows that for any bounded  $a,b[0,1]\to \mathbb{R}$ and $\psi:[0,T] \to \mathbb{R}$, we have, 
				$$ \int_{[0,1]^2_T} \psi(s)a(x)b(y)\left(H^n(x,y,s)-\E H^n(x,y,s)\right)\,dx\,dy\, ds \stackrel{a.s.}{\longrightarrow} 0.
                $$
            By Lemma~\ref{lem:es12}, we also have $\E H^n  \longrightarrow w$ in $L^1\left([0,1]^2_T\right)$ and therefore,
            for $\psi, a, b$ as above
				$$ \int_{[0,1]^2_T} \psi(s)a(x)b(y)\left(\E H^n(x,y,s)-w(x,y,s)\right)\,dx\,dy\,ds {\longrightarrow} 0.$$
				The result follows on combining the two displays.
			\end{proof}

\section{LDP under weak topology}\label{ldp}
	 In this section, we present the proofs of all the results related to large deviation behavior with respect to the weak topology stated in Section~\ref{sec:results-weak}, namely Theorem~\ref{thm:weakLDP}, Proposition~\ref{prop:no-ldp-weak}, Corollary~\ref{cor:weakLDP-path}, and Proposition~\ref{prop:ldp-point-weak}. 
     We begin in Sections \ref{sec:pfcor-weakLDP-path}--\ref{sec:pfno-ldp-path-weak} with the proof of Corollary~\ref{cor:weakLDP-path}, followed by Proposition~\ref{prop:ldp-point-weak} and  Proposition~\ref{prop:no-ldp-weak}. The proof of Theorem~\ref{thm:weakLDP} is completed in Sections \ref{sec:isrf}--\ref{sec:pflowbd}. The proof employs weak convergence methods (see~\cite{Budhirajaweakconv}), with the upper bound proof given in Section \ref{sec:pfofldpupp} and the lower bound shown in Section \ref{sec:pflowbd}. The fact that $J^{(\beta^+, \beta^-)}$ is a rate function is proved in Section \ref{sec:isrf}. Section \ref{sec:varrep} presents the key variational representation which is used in the proofs of both the upper and the lower bound. Sections \ref{sec:char} and \ref{sec:eval} provide some intermediate results used in the proofs of the upper bound and lower bound, respectively.

	\subsection{Proof of Corollary~\ref{cor:weakLDP-path}}\label{sec:pfcor-weakLDP-path}
	By contraction principle, Theorem~\ref{cor2:thmlin} and Theorem~\ref{thm:weakLDP}, it is enough to establish that the map $\text{Av}^{\epsilon} : \left(\mathcal{W}_0, d_{\cW_0, \text{weak}}\right) \to \mathbb{C} \left([0,T] : \left(\mathcal{S}_0, d_{\cS_0, \text{weak}}\right)\right)$ is continuous. Consider any sequence $\phi^n \in \cW_0$ converging to $\phi \in \cW_0$ in weak topology and sequence of time instants $t_n \in [0,T]$ converging to $t \in [0,T]$. It is enough to show that $\text{Av}^{\epsilon}(\phi^n)_{t_n}$ converges weakly to $\text{Av}^{\epsilon}(\phi)_t$. In light of the Lipschitz property in~\eqref{eq:av-lip}, we only need to establish that $\text{Av}^{\epsilon}(\phi^n)_{t}$ converges weakly to $\text{Av}^{\epsilon}(\phi)_t$. Towards that goal, fix any $\psi \in L^2 \left([0,1]^2 \right)$. Then, 
	\begin{align*}
		\int_{[0,1]^2} \psi \text{Av}^{\epsilon}(\phi^n)_{t} &= \dfrac{1}{|U_{\epsilon}(t)|} \int_{[0,1]^2_T} \psi(x,y) \mathbbm{1}_{U_{\epsilon}(t)}(s)  \phi^n(x,y,s)\,dx\,dy\,ds \\
		& \longrightarrow \dfrac{1}{|U_{\epsilon}(t)|} \int_{[0,1]^2_T} \psi(x,y) \mathbbm{1}_{U_{\epsilon}(t)}(s)  \phi(x,y,s)\,dx\,dy\,ds = 	\int_{[0,1]^2} \psi \text{Av}^{\epsilon}(\phi)_{t},
	\end{align*}
	where the convergence follows from the fact that $\phi^n$ converges weakly to $\phi$. This completes the proof.
	\qed
	
\subsection{Proof of Proposition~\ref{prop:ldp-point-weak}} {\label{sec:pfpropldp-point-weak}}

Recall that $\mu_t^n := \E H^n_t$ for any $t \in [0,T]$. If we prove that $\| \mu_t^n - w_t\|_1 \to 0$ as $n \to \infty$ for any $t \in (0,T]$, then it will prove the almost sure convergence in the statement of the proposition by virtue of Lemma~\ref{es11}. Moreover, if we can establish that both $\log w_t$ and $\log (1-w_t)$ are integrable on $[0,1]^2$ and $\| \log \mu_t^n - \log w_t\|_1, \| \log (1-\mu_t^n) - \log (1-w_t)\|_1 \to 0$, then we can conclude the LDP for the sequence $\left\{H_t^n : n \geq 1\right\}$ with respect to the weak topology. This follows from the arguments and references in~\cite[Section 3.3]{Markering2023}. Finally one can conclude the joint LDP,  as stated in Proposition \ref{prop:ldp-point-weak}, by exploiting the Markov property of the graphon process and combining the marginal and conditional LDPs using \cite[Theorem 5]{Biggins2004}. In particular, similar to~(\ref{boundP1}), one can show that 
$$ 	\mu^n_{t|s} := \E \left(H_t^n \mid  H_s^n\right)  = e^{-a(n)(B^n_t-B_s^n)} \left[ \int_s^t a(n)\beta^{n,+}_u e^{a(n)(B^n_u-B^n_s)}du +  H_s^n\right].$$
The formula above can now be used to establish the LDP for the conditional distribution of $H^n_t$ given $H_s^n$ (using arguments identical to the marginal LDP summarized earlier) and it turns out to be the exact same LDP as the unconditional distribution of $H^n_t$.  \cite[Theorem 5]{Biggins2004} is then applied to get the joint LDP of $(H_t^n,H_s^n)$. This technique can be extended to get any finite-dimensional joint LDP.

Returning to complete the proof of marginal LDP, fix $t \in (0,T]$. We start by proving $\| \mu_t^n - w_t\|_1 \to 0$ as $n \to \infty$. {Consider an arbitrary subsequence; it is enough to prove a.e. convergence, along a further subsequence, of $\mu_t^n$ to $w_t$ on $[0,1]^2$.} The technique follows closely to the proof of Lemma~\ref{lem:es12}. In particular, the estimates in~(\ref{eq:1244}) and~(\ref{eq:1245}) can be modified to the following : For a.e. $(x,y) \in [0,1]^2$, 
$$ \int_0^t a(n) |\beta^{n,+}_s - \beta^{+}_s| \exp\{-a(n)(B^n_t - B^n_s)\} ds \leq \dfrac{1}{c_{\beta}} \sup_{s \in [0,T]} |\beta_s^{n,+} - \beta_s^+|$$
and 
\begin{align*}
	&a(n)\left| \int_0^t [\exp\{-a(n)(B^n_t-B^n_s)\} - \exp\{-a(n)(B_t-B_s)\}]\beta^+_s \,ds\right| \\
	& \leq \left(\sup_{s \in [0,T]} \beta_s^+ \right)  \left(\int_0^{\infty} ue^{-c_{\beta} u} du\right) \sup_{s \in [0,T]}  \left(|\beta^{n, +}_s - \beta^{+}_s| + |\beta^{n, -}_s - \beta^{-}_s|\right).
\end{align*}
These two estimates, along with~(\ref{boundP1}) and~(\ref{eq:mu-t-es1}) and  Assumption~\ref{assmp:2}\eqref{assump:1AB} prove the {a.e. convergence of $\mu_t^n$ to $w_t$ along a further subsequence of the chosen subsequence}.

Now we prove that $\log w_t$ is integrable and $\| \log \mu_t^n - \log w_t\|_1 \to 0$. The proof of the other two assertions are exactly same. Note that, $|\log w_t| \leq |\log \beta_t^+| + |\log (\beta_t^++\beta_t^-)|$. Since, $\beta_t^{\pm} \geq c_{\beta} >0$ and are integrable, the integrability of $\log w_t$ follows. It is now enough to prove uniform integrability of $\left\{\log \mu_t^n : n \geq 1 \right\}.$ Recall the notation of $B_t^n$ from the proof of Lemma~\ref{lem:es12}. From~(\ref{boundP1}), we have
\begin{align*}
	0 \leq - \log \mu_t^n &\leq -\log \int_0^t a(n) \beta_s^{n,+}\exp \left[ -a(n)\left(B_t^n - B_s^n\right)\right]\,ds \\
	& \leq - \log a(n) c_{\beta} \int_0^t \exp \left[ - a(n) (t-s) \sup_{u \in [0,T]} \left(\beta_u^{n,+}+\beta_u^{n,-}\right)\right]\,ds \\
	& = - \log \dfrac{c_{\beta} \left[1- \exp\left(-a(n)t \sup_{u \in [0,T]} \left(\beta_u^{n,+}+\beta_u^{n,-}\right)\right)\right]}{\sup_{u \in [0,T]} \left(\beta_u^{n,+}+\beta_u^{n,-}\right)} \\
	& \leq - \log c_{\beta} \left[1-e^{-2a(n)tc_{\beta}}\right] + \log_+ \sup_{u \in [0,T]} \left(\beta_u^{n,+}+\beta_u^{n,-}\right),
\end{align*}
where $\log_+ (x) := \log (x) \vee 0.$ The uniform integrability of the last expression above  follows from Assumption~\ref{assmp:2}\eqref{assump:1AB}, since
$$ \limsup_{n \to \infty} \int_{[0,1]^2} \sup_{s \in [0,T]} \beta_s^{n,\pm} \leq \int_{[0,1]^2} \sup_{s \in [0,T]} \beta_s^{\pm} +  \limsup_{n \to \infty} \int_{[0,1]^2} \sup_{s \in [0,T]} \big \lvert \beta_s^{n,\pm} - \beta_s^{\pm} \big \rvert < \infty. $$
\qed

\subsection{Non-existence of sample-path LDP in weak topology}{\label{sec:pfno-ldp-path-weak}}
In this section, we shall prove Proposition~\ref{prop:no-ldp-weak} using Proposition~\ref{prop:ldp-point-weak}. Towards that goal, we recall the concept of exponential tightness.

\begin{defn}{\label{def:exp-tight}}
	Let $\left\{b(n) : n \geq 1 \right\}$ be a sequence of positive real numbers. A sequence of $(S,d_S)$-valued random variables $\left\{X_n : n \geq 1\right\}$ is said to be exponentially tight with speed $b(n)$ if for any $\alpha \in (0,\infty)$, there exists a compact set $K_{\alpha} \subseteq S$ such that 
	$$ \limsup_{n \to \infty}  \dfrac{1}{b(n)} \log \mathbb{P} \left( X_n \notin K_{\alpha}\right) \leq - \alpha.$$
\end{defn}

The following two results about exponential tightness will be crucial in our current proof and later in other arguments. 

\begin{thm}{\cite[Lemma 3.5]{feng2006large}}{\label{thm:ldp-to-exp-tight}}
	Suppose a sequence of $(S,d_S)$-valued random variables $\left\{X_n : n \geq 1\right\}$ satisfies an LDP with speed sequence $\left\{b(n) : n \geq 1 \right\}$ and rate function $I$. Then $\left\{X_n : n \geq 1\right\}$ is exponentially tight with speed $b(n)$.
\end{thm}

\begin{thm}{\cite[Theorem 4.1]{feng2006large}}{\label{thm:exp-tight-D-space}}
	Let $\left\{(X^n(t))_{t \in [0,T]} : n \geq 1\right\}$ be a sequence of $\mathbb{D} \left([0,T]:(S,d_S)\right)$-valued random variables. Suppose $\mathcal{T} \subseteq [0,T]$ is a dense subset such that $\left\{X^n(t) : n \geq 1\right\}$ is exponentially tight in the space $(S,d_S)$ with speed sequence  $\left\{b(n) : n \geq 1 \right\}$ for every $t \in \mathcal{T}$. Then $\left\{(X^n(t))_{t \in [0,T]} : n \geq 1\right\}$ is exponentially tight with speed $b(n)$ in the space $\mathbb{D} \left([0,T]:(S,d_S)\right)$ if and only if for every $\varepsilon >0$, 
	$$ \lim_{\delta \downarrow 0} \limsup_{n \to \infty} \dfrac{1}{b(n)}\log \mathbb{P} \left[ \text{Osc}^{\prime}\left(X^n, \delta; (S,d_S)\right) > \varepsilon\right] = - \infty.$$
	In the above statement, if the $\mathbb{D}$-space is replaced by $\mathbb{C}$-space, then $\text{Osc}^{\prime}$ can be replaced by $\text{Osc}$.
\end{thm}

Before proceeding to the proof of Proposition~\ref{prop:no-ldp-weak}, we make an observation regarding the rate function in Proposition~\ref{prop:ldp-point-weak}. Note that, $0 \leq \cQ_2(u,v) \leq - \log v(1-v)$, for any $u \in [0,1]$ and $v \in (0,1)$. Thus, for any $f \in \cS_0$ and $t \in (0,T]$, we have 
\begin{align*}
	\int_{[0,1]^2} \cQ_2(f,w_t) \leq -\int_{[0,1]^2} \log w_t(1-w_t) \leq -2 \log c_{\beta} + 2\int_{[0,1]^2} \log (\beta_t^++\beta_t^-),
\end{align*}
which yields the following uniform bound by Assumption~\ref{assmp:2}\eqref{assump:1AB} : 
\begin{equation}{\label{eq:I-point-unif-ub}}
	 I^{\beta^+,\beta^-}_{\text{point},\cT} \leq -2|\cT| \log c_{\beta} + 2 |\cT| \int_{[0,1]^2} \sup_{s \in [0,T]} (\beta_s^++\beta_s^-) < \infty, \; \forall \; \text{finite }\cT \subseteq (0,T].
\end{equation}

We shall also need the  following elementary lemma which establishes the uniqueness for the choice of speed sequence (upto a scaling). We include a proof in  Section~\ref{sec:aux} for completeness.

\begin{lem}{\label{lem:unique-speed}}
	Suppose $\left\{Y_n : n \geq 1\right\}$, a sequence of $S$-valued random variables, satisfies LDP in $S$ with speed sequences  $b(n)$ and $b^*(n)$ and the corresponding rate functions are given by $I$ and $I^*$. Suppose $b(n)/b^*(n) \to \infty$ as $n \to \infty$ and
 that $I^{-1}\left(\left\{0\right\}\right)$ is a singleton. Then  $I^*(x)=\infty$ whenever $I(x) >0$, namely $I^*$ is a trivial rate function.
\end{lem}

\begin{proof}[Proof of Proposition~\ref{prop:no-ldp-weak}]
	We prove by contradiction. Suppose $\left\{H^n_{\bcdot} : n \geq 1\right\}$  satisfies an LDP in the space $\mathbb{D} \left( [0,T] : \left(\cS_0, d_{\cS_0, \text{weak}}\right)\right)$ with speed $b(n)$ and non-trivial rate function $I$. First we establish that $b(n) /n^2 \to 0$ as $n \to \infty$. By Theorem~\ref{thm:ldp-to-exp-tight}, we have 
	$\left\{H^n_{\bcdot} : n \geq 1\right\}$ is exponentially tight with speed $b(n)$. On the otherhand, since $\left(\cS_0, d_{\cS_0, \text{weak}}\right)$ is compact, we can conclude that for any $t \in (0,T]$, $\left\{H_t^n : n \geq 1\right\}$ is exponentially tight (for any speed sequence). Thus, Theorem~\ref{thm:exp-tight-D-space} yields that for every $\varepsilon >0$, 
	\begin{equation}{\label{eq:no-ldp-1}}
		 \lim_{\delta \downarrow 0} \limsup_{n \to \infty} \dfrac{1}{b(n)}\log \mathbb{P} \left[ \text{Osc}^{\prime}\left(H^n_{\bcdot}, \delta; (\cS_0,d_{\cS_0, \text{weak}})\right) > \varepsilon\right] = - \infty.
	\end{equation}
Since the graph updates at most one edge at any time instant almost surely, we have 
\begin{equation}{\label{eq:jump-es}}
	\text{jump}\left(H^n_{\bcdot}; (\cS_0,d_{\cS_0, \text{weak}})\right) \leq \sup_{t \in (0,T]} \| H^n(t) - H^n(t-)\|_2 \leq \dfrac{1}{n},
\end{equation}
and hence for any $0 < t \leq  T-\delta$, 
\begin{align*}
	\text{Osc}^{\prime}\left(H^n_{\bcdot}, \delta; (\cS_0,d_{\cS_0, \text{weak}})\right)  &\geq \dfrac{1}{2}\text{Osc}\left(H^n_{\bcdot}, \delta; (\cS_0,d_{\cS_0, \text{weak}})\right)  - \dfrac{1}{2n} \\
	& \geq \dfrac{1}{2} d_{\cS_0, \text{weak}} (H_{t+\delta}^n, H_t^n) - \dfrac{1}{2n}.
\end{align*}
Equation~(\ref{eq:no-ldp-1}) then yields
$$  \lim_{\delta \downarrow 0} \limsup_{n \to \infty} \dfrac{1}{b(n)}\log \mathbb{P} \left[ d_{\cS_0, \text{weak}}(H_{t+\delta}^n, H_t^n) > \varepsilon\right] = - \infty, \; \forall \; \varepsilon >0, $$
whereas, by Proposition~\ref{prop:ldp-point-weak} and~(\ref{eq:I-point-unif-ub}), for small enough $\varepsilon >0$ and any $\delta >0$, 
$$  \liminf_{n \to \infty} \dfrac{1}{n^2}\log \mathbb{P} \left[ d_{\cS_0, \text{weak}}(H_{t+\delta}^n, H_t^n) > \varepsilon\right] > -\infty.$$
This shows that $b(n)/n^2 \to 0$ as $n \to \infty$.

Since the $\text{jump}$ function is continuous in Skorokhod topology, we have, by~(\ref{eq:jump-es}), that $I(\phi_{\bcdot}) =\infty$ whenever $\phi_{\bcdot}$ does not have continuous path. Therefore, by the \textit{extended contraction principle} (see~\cite[Lemma 3.11]{feng2006large}), we conclude that for any $t \in (0,T]$ , the sequence $\left\{H_t^n : n \geq 1\right\}$ satisfies an LDP in $\left(\cS_0, d_{\cS_0, \text{weak}}\right)$ with speed $b(n)$ and rate function $I_t$ given by $I_t(f)=\inf \left\{I(\phi_{\bcdot}) : \phi_t=f\right\}$.  From Proposition \ref{prop:ldp-point-weak}, we already have an LDP for this sequence with speed $n^2$ and non-trivial rate function $I^{\beta^+,\beta^-}_{\text{point},t}$ which vanishes only at $w_t$. By Lemma~\ref{lem:unique-speed}, this guarantees that $I_t(f)=\infty$ whenever $f \neq w_t$. Hence, $I(\phi_{\bcdot})=\infty$ if $\phi_t \neq w_t$ for some $t \in (0,T]$. In other words, $I(\phi_{\bcdot})=\infty$, possibly except for $\phi_{\bcdot}=w_{\bcdot}$; i.e., $I$ is trivial, a contradiction.
\end{proof}

	\subsection{Compactness of level sets}\label{sec:isrf}
	In this section we argue that $J^{(\beta^+,\,\beta^-)}$ is indeed a rate function on $\left(\mathcal{W}_0, d_{\cW_0, \text{weak}}\right)$, namely it has compact sub-level sets. Since $\cW_0$ is compact, we simply need to show that the function $\phi \mapsto 	J^{(\beta^+,\,\beta^-)}(\phi) $ is lower semi-continuous with respect to the weak topology on $\cW_0$. The following lemma will be useful. When $L^1$ is replaced by $L^2$ in the statement of the lemma, the result is immediate from the definition of the weak topology. The fact that we can relax the space to be $L^1$ is a consequence of the fact that the functions in $\cW_0$ take values in $[0,1]$.
	
		\begin{lem}{\label{lem:weaksemicont1}}
		Fix $\phi^* \in L^1\left([0,1]^2_T\right)$. The map $\phi \mapsto \int_{[0,1]^2_T} \phi\phi^*$ is continuous on the space  $\left(\mathcal{W}_0, d_{\cW_0, \text{weak}}\right)$.
	\end{lem}
	\begin{proof}
		For any $\phi\in \cW_0$, we have $\int_{[0,1]^2_T} |\phi\phi^*| \leq \|\phi^*\|_1$ and hence the map is well defined. It is enough to prove continuity for only non-negative $\phi^*$. Take $\left\{\phi_n : n \geq 1\right\} \subseteq \cW_0$ converging to $\phi \in\cW_0$ in weak topology as $n \to \infty$. For any $M \in (0,\infty)$, 
		$$ \liminf_{n \to \infty} \int_{[0,1]^2_T} \phi^*\phi_n \geq \liminf_{n \to \infty} \int_{[0,1]^2_T} (\phi^* \wedge M)\phi_n =  \int_{[0,1]^2_T} (\phi^* \wedge M)\phi.$$
		Taking $M \uparrow \infty$, we conclude that $ \liminf_{n \to \infty} \int_{[0,1]^2_T} \phi^*\phi_n \geq  \int_{[0,1]^2_T} \phi^*\phi.$ Since, the sequence $\left\{(1-\phi_n) : n \geq 1\right\}$ also converges to $1-\phi$ with respect to weak topology in $\cW_0$, we have $ \liminf_{n \to \infty} \int_{[0,1]^2_T} \phi^*(1-\phi_n) \geq  \int_{[0,1]^2_T} \phi^*(1-\phi),$ which can be re-written as $ \limsup_{n \to \infty} \int_{[0,1]^2_T} \phi^*\phi_n \leq  \int_{[0,1]^2_T} \phi^*\phi.$ This completes the proof.
	\end{proof}
	
		\begin{prop}{\label{lem:weaksemicont2}}
		The function $J^{(\beta^+,\,\beta^-)} : \left(\mathcal{W}_0, d_{\cW_0, \text{weak}}\right) \to [0,\infty]$ is lower semi-continuous.
	\end{prop}
	
	\begin{proof}
		Since $\beta^{\pm} \in L^1\left([0,1]^2_T\right)$, we have $	J^{(\beta^+,\,\beta^-)}(\phi) < \infty$ for any $\phi \in \cW_0$. Note that, for any $a,b \geq 0$, we have $\inf_{c \in (0,\infty)} (ca+b/c) = 2\sqrt{ab}$ and for any $\varepsilon \in (0,1)$,
		\begin{equation}{\label{es}}
			\inf_{c \in [\varepsilon,1/\varepsilon]} (ca+b/c) \leq \inf_{c \in (0,\infty)} (ca+b/c) + \varepsilon (a \vee b).
		\end{equation}
		Also, suppressing  $(x,y)\in [0,1]^2$ from the notation, 
		\begin{align*}
			\left(\sqrt{\beta_s^+(1-\phi_s)} - \sqrt{\beta_s^-\phi_s}\right)^2
			&= \beta_s^+(1-\phi_s) + \beta_s^-\phi_s-2\sqrt{\beta_s^+\beta_s^-\phi_s(1-\phi_s)} \\
			&= \beta_s^+(1-\phi_s) + \beta_s^-\phi_s-\inf_{c\in (0,\infty)} \left( c\beta_s^+(1-\phi_s) + \dfrac{\beta_s^-\phi_s}{c}\right)\\
			&= \sup_{c \in (0,\infty)} \left[ (1-c)\beta_s^+(1-\phi_s)+(1-1/c)\beta_s^-\phi_s\right].
		\end{align*}
		Therefore,
		\begin{align*}
			J^{(\beta^+,\,\beta^-)}(\phi) &= \int_{[0,1]^2_T} \left( \sqrt{\beta^+(1-\phi)}-\sqrt{\beta^-\phi}\right)^2 \\
			&= \int_{[0,1]^2_T} \sup_{c \in (0,\infty)} \left[ (1-c)\beta^+(1-\phi)+(1-1/c)\beta^-\phi\right] \\
			& \geq \sup_{\substack{g \text{ measurable}\\ g:[0,1]^2_T \to (0,\infty)}} \int_{[0,1]^2_T} \left[ (1-g)\beta^+(1-\phi)+(1-1/g)\beta^-\phi\right] =:	\tilde J^{(\beta^+,\,\beta^-)}(\phi),
		\end{align*}
		where the last supremum is taken over all measurable functions $g :[0,1]^2_T \to (0,\infty).$ On the otherhand, for any $\varepsilon \in (0,1)$, 
		\begin{align*}
			&J^{(\beta^+,\,\beta^-)}(\phi) \\
            &= \int_{[0,1]^2_T} \sup_{c \in (0,\infty)} \left[ (1-c)\beta^+(1-\phi)+(1-1/c)\beta^-\phi\right] \\
			& \leq \int_{[0,1]^2_T} \sup_{c \in [\varepsilon,1/\varepsilon]} \left[ (1-c)\beta^+(1-\phi)+(1-1/c)\beta^-\phi\right] + \varepsilon \int_{[0,1]^2_T} \left[\beta^+(1-\phi) \vee \beta^-\phi \right] \\
			&= \sup_{\substack{g \text{ measurable}\\ g:[0,1]^2_T \to [\varepsilon,1/\varepsilon]}} \int_{[0,1]^2_T} \left[ (1-g)\beta^+(1-\phi)+(1-1/g)\beta^-\phi\right] +\varepsilon \int_{[0,1]^2_T} \left[\beta^+(1-\phi) \vee \beta^-\phi \right] \\
			& \leq \sup_{\substack{g \text{ measurable}\\ g:[0,1]^2_T \to [\varepsilon,1/\varepsilon]}} \int_{[0,1]^2_T} \left[ (1-g)\beta^+(1-\phi)+(1-1/g)\beta^-\phi\right]+ \varepsilon \int_{[0,1]^2_T} \left[\beta^++ \beta^-\right]\\
			&=: 	J_{\varepsilon}^{(\beta^+,\,\beta^-)}(\phi) + \varepsilon \int_{[0,1]^2_T} \left[\beta^++ \beta^-\right] \leq 	\tilde J^{(\beta^+,\,\beta^-)}(\phi) + \varepsilon \int_{[0,1]^2_T} \left[\beta^++ \beta^-\right],
		\end{align*}
		where the third equality above follows from the observation that the infimum in the expression $\inf_{c \in [\varepsilon,1/\varepsilon]} (ca+b/c)$ is always attained at $c=(\sqrt{b/a} \wedge 1/\varepsilon) \vee \varepsilon$ for any $a,b \geq 0$. Taking $\varepsilon \downarrow 0$, we conclude that $	J^{(\beta^+,\,\beta^-)}(\phi) =	\tilde J^{(\beta^+,\,\beta^-)}(\phi) =\lim_{\varepsilon \downarrow 0} 	J_{\varepsilon}^{(\beta^+,\,\beta^-)}(\phi) =\sup_{\varepsilon >0} 	J_{\varepsilon}^{(\beta^+,\,\beta^-)}(\phi) $. In other words,
		$$ J^{(\beta^+,\,\beta^-)}(\phi) = \sup_{\varepsilon >0}   \sup_{\substack{g \text{ measurable}\\ g:[0,1]^2_T \to [\varepsilon,1/\varepsilon]}} \int_{[0,1]^2_T} \left[ (1-g)\beta^+(1-\phi)+(1-1/g)\beta^-\phi\right].$$
		From Lemma \ref{lem:weaksemicont1}, for any measurable $g:[0,1]^2_T \to [\varepsilon,1/\varepsilon]$ , the map $$\phi \mapsto \int_{[0,1]^2_T}  \left[ (1-g)\beta^+(1-\phi)+(1-1/g)\beta^-\phi\right]$$
        is continuous on $\cW_0$ with respect to weak topology. Thus $J^{(\beta^+,\,\beta^-)}$ is lower semi-continuous by virtue of being pointwise supremum of a family of continuous functions. 
	\end{proof}

	\subsection{Variational representation}\label{sec:varrep}
	One can give a representation for the evolution of the Markov processes $\{X^n_{i,j}(t), t \in [0,1]\}_{i,j \in [n]}$ using suitable Poisson random measures (PRM) as follows.    
    
	For $i,j \in [n]$ and $t \in [0,T]$
	\begin{align}
		X^n_{i,j}(t) &= x^n_{i,j}+ \int_{[0,t]\times [0,\infty)} (1-X^n_{i,j}(s-)) \mathbbm{1}_{[0, a(n)\beta^{n,+}_{i,j}(s)]}(r) N_{i,j}^{n,+}(ds\, dr)  \nonumber \\ 
		&\hspace{1 in}- \int_{[0,t]\times [0,\infty)} X^n_{i,j}(s-) \mathbbm{1}_{[0, a(n)\beta^{n,-}_{i,j}(s)]}(r) N_{i,j}^{n,-}(ds\, dr),  \label{evolution}
	\end{align}
	where $\left\{N_{i,j}^{n,\pm} : i,j \in [n]\right\}$ is a collection of $2n^2$ mutually independent PRM on $[0,T]\times [0,\infty)$, with intensity measure being the Lebesgue measure, on some probability space $(\Omega, \cF, P)$. Let $\mathscr{F}^n_t$  denote the $\sigma$-algebra generated by the collection 
    $$\left\{N_{i,j}^{n,\pm}(A) : A \subseteq [0,t]\times [0,\infty), \; i,j \in [n]\right\}.$$ 
	Let $\cla^n(C)$ be the collection of all $\Phi^n \equiv (\varphi^{n, \pm}_{i,j})_{i,j \in [n]}$, where each $\varphi^{n, \pm}_{i,j}$ is a $\{\mathscr{F}^n_t\}$-predictable map from $\Omega \times [0,T] \to [0, \infty)$ that is bounded above and below, away from $0$, uniformly in $(\omega,t) \in \Omega\times [0,T]$.
	Define $\ell: [0,\infty) \to [0,\infty)$ as
	\begin{align*}
		\ell(x) := x\log x - x +1, \; x \ge 0.
	\end{align*}

	Then, using the variational representation for functionals of PRM from \cite{BDM2} (see e.g. \cite[Theorem 8.12]{Budhirajaweakconv}), for any continuous and bounded $\Psi : \cW_0  \to [0,\infty)$ with respect to the weak topology, we have
		\begin{align}
			&-\frac{1}{a(n)n^2} \log \E\left[\exp\left(-a(n)n^2\Psi \left(H^n\right)\right)\right] \label{eq:varrep1}\\ 
			&= \inf_{\Phi^n \in \cla^n}\E\Big[ \Psi\left(\bar{H}^n\right)
			+ \frac{1}{n^2}\sum_{i,j=1}^n \int_0^T\left(\ell(\varphi^{n, +}_{i,j}(s)) \beta^{n,+}_{i,j}(s)
			+\ell(\varphi^{n, -}_{i,j}(s))\beta^{n,-}_{i,j}(s)\right) ds
			\Big], \nonumber 
		\end{align}
	where $\bar H^n$ is a $\cW_0$ valued random variable defined as
	\begin{align}\label{eq:barhn}
	\bar{H}^n_t := \sum_{i,j=1}^n \bar{X}^n_{i,j}(t)\mathbbm{1}_{Q^n_{i,j}},
	\end{align}
	and
	\begin{align}
		\bar{X}^n_{i,j}(t) &= \bar{X}^n_{i,j}(0)+ \int_{[0,t]\times [0,\infty)} (1-\bar{X}^n_{i,j}(s-)) \mathbbm{1}_{[0, a(n)\varphi^{n,+}_{i,j}(s)\beta^{n,+}_{i,j}(s)]}(r) N_{i,j}^{n,+}(ds\, dr)\nonumber \\ 
		&\hspace{1 in}- \int_{[0,t]\times [0,\infty)} \bar{X}^n_{i,j}(s-) \mathbbm{1}_{[0, a(n)\varphi^{n,-}_{i,j}(s)\beta^{n,-}_{i,j}(s)]}(r) N_{i,j}^{n,-}(ds\, dr). \label{eq:evol}
	\end{align}
	
	Fix $\delta \in (0,1)$ and let $\Phi^n_{\delta}$ be $\delta$-optimal for the right side in \eqref{eq:varrep1}. We omit $\delta$ from the notation for the sake of brevity.
	Then
	\begin{align}
		\|\Psi\|_{\infty}+\delta &\geq -\frac{1}{a(n)n^2} \log \E\left[\exp\left(-a(n)n^2\Psi\left(H^n\right)\right)\right] +\delta\nonumber \\ 
		&\ge \E\left[ \Psi \left(\bar{H}^n\right)
		+ \frac{1}{n^2}\sum_{i,j=1}^n \int_0^T\left(\ell(\varphi^{n, +}_{i,j}(s)) \beta^{n,+}_{i,j}(s)
		+\ell(\varphi^{n, -}_{i,j}(s))\beta^{n,-}_{i,j}(s)\right) ds
		\right]. \label{delopt}
	\end{align}
	Using Assumption \ref{assmp:1}\eqref{assump:1B}, we have 
	\begin{align}\label{eq:costbd}
		\E \left[  \frac{1}{n^2}\sum_{i,j=1}^n \int_0^T\left(\ell(\varphi^{n, +}_{i,j}(s)) 
		+\ell(\varphi^{n, -}_{i,j}(s)) \right) ds
		\right] \le (2\|\Psi\|_{\infty}+\delta)c_{\beta}^{-1} < \infty.
	\end{align}
	Define $\clv := [0,\infty)^2 \times \left\{0,1\right\}  \times [0,1]^2 \times [0,T]$ (equipped with the product topology) and define $\clp(\clv)$ valued random variable (i.e., a random measure on $\clv$) $\Lambda^n$ as follows : 
	\begin{align*}
		\Lambda^n(dv^+\, dv^-\, du\,  dx\, dy\,ds)
		&:= 
		 \sum_{i,j=1}^n \delta_{(\varphi^{n, +}_{i,j}(s), \varphi^{n, -}_{i,j}(s), \bar X^n_{i,j}(s)}(dv^+\, dv^-\, du)\mathbbm{1}_{Q_{i,j}^n}(x,y) dx\, dy\,ds.
	\end{align*}
    Henceforth, we will abbreviate $(dv^+\, dv^-\, du\,  dx\, dy\,ds) := d\theta$. Then, with $\bar{\Lambda}^n= \E \Lambda^n$,
	\begin{equation}{\label{tight}}
		\int_{\clv} \left(\ell(v^+)+\ell(v^-)\right) \bar{\Lambda}^n(d\theta) = \E \int_{\clv} \left(\ell(v^+)+\ell(v^-)\right)\, \Lambda^n(d\theta) \leq  (2\|\Psi\|_{\infty}+\delta)c_{\beta}^{-1} < \infty,
	\end{equation} 
	It is easy to see that the function $\psi : \clv \to [0,\infty)$, defined as $\psi(v^+,v^-,u,x,y,s)=\ell(v^+)+\ell(v^-),$ for all $(v^+,v^-,u,x,y,s) \in \clv$, is a tightness function, i.e., it has pre-compact level sets. It then follows, (cf.  \cite[Lemma 2.9, Theorem 2.11]{Budhirajaweakconv}),  that $\left\{\Lambda^n : n \geq 1\right\}$ is a tight sequence of random variables in $\clp(\clv)$.

	Furthermore, the right side of \eqref{delopt} can be re-written as
	\begin{align}
		& \E\Big[ \Psi\Big(\bar{H}^n\Big)
		+ \int_{\clv}\left(\ell(v^+) \beta^{n,+}_s(x,y)
		+\ell(v^{-}) \beta^{n,-}_s(x,y)\right) \Lambda^n(d\theta) \label{eq:deloptmz}
		\Big],
	\end{align}
	and hence
	\begin{equation}{\label{tightbound2}}
		\sup_{n \geq 1} \E  \int_{\clv}\left(\ell(v^+) \beta^{n,+}_s(x,y)
		+\ell(v^{-}) \beta^{n,-}_s(x,y)\right) \Lambda^n(d\theta) \leq 2\|\Psi \|_{\infty} + \delta< \infty.
	\end{equation}
    
	\subsection{Characterization of the weak limit points}\label{sec:char}
    Let $\{(\bar H^n, \Lambda^n): n\ge 1\}$ be as in Section \ref{sec:varrep} associated with the $\delta$-optimal $\Phi^n_{\delta}$ introduced there.
	Since, $\left\{\Lambda^n : n \geq 1\right\}$ is tight,
    and $\cW_0$ is compact in weak topology, so is $\left\{\left(\bar{H}^n,\Lambda^n\right) : n \geq 1\right\}$.  The latter is viewed as a sequence of random variables in $\cW_0 \times \clp(\clv)$ which we endow with the product topology of the weak topology inherited from $L^2\left([0,1]^2_T\right)$ for $\cW_0$ and the topology of weak convergence for $\clp(\clv)$. Let $\left(\bar H,\Lambda \right)$ be a weak limit point of $\{\left(\bar{H}^n,\Lambda^n\right), n \ge 1\}$. We now observe some properties of $(\bar H, \Lambda)$.
	\begin{enumerate}[(a)]
		\item For any $n \geq 1$, the marginal distribution $\Lambda^n_{(4,5,6)}$ is the Lebesgue measure on $[0,1]^2_T$. Hence, almost surely,  $\Lambda_{(4,5,6)}(dx\, dy\,ds) = dx\, dy\,ds$ on $[0,1]^2 \times [0,T]$. 
		
		\item Disintegrating $\Lambda$ as
		$$\Lambda(d\theta) = \tilde \Lambda_{x,y,s} (dv^+\, dv^-\, du)\,dx\,dy\,ds = \widehat \Lambda_{x,y,s,u}(dv^+\, dv^-) \Lambda^*_{x,y,s}(du) dx \, dy\,ds,$$
		we have that, almost surely,
		\begin{equation}\label{eq:charh}
			\bar H(x,y,s) = \int_{ \left\{0,1\right\}} u \Lambda^*_{x,y,s}(du), \; \mbox{ for a.e.}[x,y,s].
		\end{equation}
		To see \eqref{eq:charh}, 
		assume without loss of generality (by appealing to Skorokhod representation theorem) that $(\bar H^n, \Lambda^n) \to (\bar H, \Lambda)$ almost surely. 
		Then for any $f :[0,1]^2 \times [0,T] \to \RR$ continuous, we have
		$$\int_{[0,1]^2_T} f(x,y,s) \bar H^n(x,y,s) dx\, dy\,ds \stackrel{a.s.}{\longrightarrow} \int_{[0,1]^2_T} f(x,y,s) \bar H(x,y,s) dx\, dy\,ds.$$
		Also,
		\begin{align*}
			\int_{\clv} f(x,y,s) u \Lambda^n(d\theta) &\stackrel{a.s.}{\longrightarrow} \int_{\clv} f(x,y,s) u \Lambda(d\theta) \\
            &= \int_{[0,1]^2_T} f(x,y,s) \left(\int_{\left\{0,1\right\}} u \Lambda^*_{x,y,s}(du)\right) dx \, dy\,ds.
		\end{align*}
		Note that
        \begin{align*}
            \int_{[0,1]^2_T} f(x,y,s) \bar H^n(x,y,s) dx\, dy\,ds &= \int_{0}^T \sum_{i,j=1}^n \int_{Q_{i,j}^n}  f(x,y,s)\bar{X}_{i,j}^n(s)\, dx\,dy\,ds \\
            &=\int_{\clv}f(x,y,s) u \Lambda^n(d\theta).
        \end{align*}
		Combining the last three displays we have \eqref{eq:charh}.
		
		\item The condition in~\eqref{tightbound2} implies that 
		\begin{equation}{\label{tightbound}}
		\E  \int_{\clv}\ell(v^{\pm}) \beta_s^{\pm}(x,y)
			 \Lambda(d\theta) \leq
           \liminf_{n \to \infty}\E \int_{\clv}  \ell(v^{\pm})\beta_s^{n,\pm}(x,y)\, \Lambda^n(d\theta) \le 2\|\Psi\|_{\infty} + \delta.
		\end{equation}
        Proof of the above statement involves standard approximation argument and is deferred to Section~\ref{sec:sub7.4}.
		\item Now rewrite \eqref{eq:evol} as
		\begin{align}
			\bar X^n_{i,j}(t) &= a(n)\int_0^t (1-\bar X^n_{i,j}(s))\beta^{n,+}_{i,j}(s) \varphi^{n, +}_{i,j}(s) ds \label{eq:contevol-cm} \\
			& \hspace{1 in} - a(n)\int_0^t \bar X^n_{i,j}(s) \beta^{n,-}_{i,j}(s) \varphi^{n, -}_{i,j}(s) ds +
			L^n_{i,j}(t), \nonumber
		\end{align}
		where
		\begin{align*}
			L^n_{i,j}(t) &= x_{i,j}^n+ \int_{[0,t]\times (0,\infty)} (1-\bar X^n_{i,j}(s-)) \mathbbm{1}_{[0, a(n) \beta^{n,+}_{i,j}(s)
				\varphi^{n, +}_{i,j}(s)]}(r) \tilde N_{i,j}^{n,+}(ds\, dr)\nonumber\\ 
			&\hspace{1 in}- \int_{[0,t]\times (0,\infty)} \bar X^n_{i,j}(s-) \mathbbm{1}_{[0, a(n) \beta^{n,-}_{i,j}(s)\varphi^{n, -}_{i,j}(s)]}(r) \tilde N_{i,j}^{n,-}(ds\, dr)
		\end{align*}
		is a martingale defined using compensated  PRM 
		$$\tilde N_{i,j}^{n,\pm}(ds\, dr) = N_{i,j}^{n,\pm}(ds\, dr) - ds\, dr.$$
		Note that, for any bounded continuous $g :[0,T] \to \mathbb{R}$ and $t \in [0,T]$,
		\begin{align*}
			&\operatorname{Var}\left(\int_0^t g(s)\,dL_{i,j}^n(s)\right) \\
            & = \E  \int_{[0,t]\times (0,\infty)} \left(g(s)\right)^2(1-\bar X^n_{i,j}(s-))^2 \mathbbm{1}_{[0, a(n) \beta^{n,+}_{i,j}(s)	\varphi^{n, +}_{i,j}(s)]}(r) ds\,dr \\
			& \hspace{0.5 in} + \E  \int_{[0,t]\times (0,\infty)} \left(g(s)\right)^2(\bar X^n_{i,j}(s-))^2 \mathbbm{1}_{[0, a(n) \beta^{n,-}_{i,j}(s)	\varphi^{n, -}_{i,j}(s)]}(r) ds\,dr \\
			&  \leq a(n) \|g\|^2_{\infty} \E \int_{0}^t \left[\beta^{n,+}_{i,j}(s)\varphi^{n,+}_{i,j}(s) + \beta^{n,-}_{i,j}(s)\varphi^{n,-}_{i,j}(s)\right]ds \\
			&  \leq a(n) \|g\|^2_{\infty} \E \int_{0}^t \left[ \left(\ell( \varphi^{n,+}_{i,j}(s))+2\right)\beta^{n,+}_{i,j}(s) + \left(\ell( \varphi^{n,-}_{i,j}(s))+2\right)\beta^{n,-}_{i,j}(s) \right]\,ds,
		\end{align*}
		where we have used that $x \leq \ell(x)+2$ for any $x \geq 0$. 
		Using \eqref{eq:contevol-cm}, for all continuous $f:[0,1]^2 \to \RR$, $\psi \in \mathcal{C}^1([0,T])$ and $t \in [0,T]$, 
		\begin{align}
			&\int_{[0,1]^2} \sum_{i,j=1}^n f(x, y)\left(\psi(t)\bar X^n_{i,j}(t) - \psi(0)\bar X_{i,j}^n(0)\right)\mathbbm{1}_{Q_{i,j}^n}(x,y)dx\, dy \label{eq:contevol-cmp} \\ 
			&	 =\int_{[0,1]^2} \sum_{i,j=1}^n f(x, y) \int_{0}^t \left[\psi(s)d\bar X^n_{i,j}(s) + \psi^{\prime}(s)\bar X_{i,j}^n(s)\,ds \right] \mathbbm{1}_{Q_{i,j}^n}(x,y)dx\, dy \nonumber\\ 
			& =
			a(n)\int_{0}^t \int_{[0,1]^2} \sum_{i,j=1}^n f(x, y) \psi(s)(1-\bar X^n_{i,j}(s))\beta^{n,+}_{i,j}(s) \varphi^{n, +}_{i,j}(s) \mathbbm{1}_{Q_{i,j}^n}(x,y) dx\, dy\, ds \nonumber
			\\ 
			&\hspace{0.2 in} - a(n)\int_0^t \int_{[0,1]^2} \sum_{i,j=1}^n f(x, y) \psi(s)\bar X^n_{i,j}(s) \beta^{n,-}_{i,j}(s) \varphi^{n, -}_{i,j}(s)\mathbbm{1}_{Q_{i,j}^n}(x,y) dx\, dy\, ds \nonumber \\
			&\hspace{0.2 in} + \int_{[0,1]^2} \int_0^t \sum_{i,j=1}^n f(x,y) \psi^{\prime}(s)\bar X^n_{i,j}(s) \mathbbm{1}_{Q_{i,j}^n}(x,y)\,dx\,dy\,ds +
			R^n(t), \nonumber 
		\end{align}
		where
		$$R^n(t) =   \int_{[0,1]^2} \sum_{i,j=1}^n f(x, y)\left[\int_0^t \psi(s)\,dL^n_{i,j}(s)\right] \mathbbm{1}_{Q_{i,j}^n}(x,y) dx\, dy = \dfrac{1}{n^2}\sum_{i,j=1}^n f^n_{i,j} L^{n,\psi}_{i,j}(t),$$
		and
		$$ f^n_{i,j} := n^2\int_{Q_{i,j}^n} f(x,y)\, dx\,dy,\;\; L^{n,\psi}_{i,j}(t) := \int_0^t \psi(s)\,dL^n_{i,j}(s).$$
		Note that $|f^n_{i,j}| \leq \|f\|_{\infty}$.
		Also, $\E R^n(t) =0$, and
		\begin{align*}
			&\operatorname{Var} \left( R^n(t)\right)   
			 = \dfrac{1}{n^4} \sum_{i,j=1}^n (f_{i,j}^n)^2\operatorname{Var}\left(L^{n,\psi}_{i,j}(t)\right) \\
			& \leq \dfrac{a(n)\|f \psi \|_{\infty}^2}{n^4}\sum_{i,j=1}^n \E \int_{0}^T \left[ \left(\ell( \varphi^{n,+}_{i,j}(s))+2\right)\beta^{n,+}_{i,j}(s) + \left(\ell( \varphi^{n,-}_{i,j}(s))+2\right)\beta^{n,-}_{i,j}(s) \right]ds \\
			& \le \frac{a(n)}{n^2}\|f\|_{\infty}^2 \|\psi\|^2_{\infty} \left[2\|\Psi\|_{\infty}+\delta + 2 \int_{0}^T \int_{[0,1]^2} (\beta_s^{n,+}(x,y) + \beta_s^{n,-}(x,y))\,dx\,dy\,ds\right],
		\end{align*}
		where the last inequality follows from \eqref{delopt}. Using Assumption \ref{assmp:2}\eqref{assump:1A}, we now conclude that $ R^n(t)/a(n) \stackrel{p}{\longrightarrow} 0$. 
		Also,
		\begin{align*}
			&\int_0^t \int_{[0,1]^2} \sum_{i,j=1}^n f(x, y)\psi(s) \left[(1-\bar X^n_{i,j}(s))\beta^{n,+}_{i,j}(s)\varphi^{n, +}_{i,j}(s) -\bar X^n_{i,j}(s) \beta^{n,-}_{i,j}(s) \varphi^{n, -}_{i,j}(s)\right] \\
            &  \hspace{4 in}  \mathbbm{1}_{Q_{i,j}^n}(x,y)\, dx\, dy\, ds \\  
			&=\int_0^t \int_{[0,1]^2} \sum_{i,j=1}^n f(x, y)\psi(s) \left[(1-\bar X^n_{i,j}(s))\beta_s^{n,+}(x,y)\varphi^{n, +}_{i,j}(s) \right]\mathbbm{1}_{Q_{i,j}^n}(x,y)\, dx\, dy\, ds \\
            & \hspace{0.5 in} - \int_0^t \int_{[0,1]^2} \sum_{i,j=1}^n f(x, y)\psi(s) \left[ \bar X^n_{i,j}(s) \beta^{n,-}_s(x,y) \varphi^{n, -}_{i,j}(s)\right]\mathbbm{1}_{Q_{i,j}^n}(x,y)\, dx\, dy\, ds \\
			&=\int_{[0,\infty)^2 \times \left\{0,1\right\} \times [0,1]^2 \times [0,t]} f(x,y)\psi(s)[(1-u)\beta_s^{n,+}(x,y)v^+ - u\beta_s^{n,-}(x,y)v^-] \Lambda^n(d\theta).
		\end{align*}
		Since the absolute value of the left hand side and the third term in the right hand side of~\eqref{eq:contevol-cmp} are bounded above by $2\|f\|_{\infty}\|\psi\|_{\infty}$ and $\|f\|_{\infty}\|\psi^{\prime}\|_{\infty}$ respectively, we have, by dividing by $a(n)$ in \eqref{eq:contevol-cmp} and sending $n\to \infty$, that
		\begin{equation}{\label{eq:conv}}
			\int_{\clv} f(x,y)\psi(s)\left[(1-u)\beta_s^{n,+}(x,y)v^+ - u\beta_s^{n,-}(x,y)v^-\right] \Lambda^n(d\theta) \to 0,
		\end{equation}
        in probability.
		We claim that this implies 
		\begin{equation}{\label{eq:conv2}}
			\int_{\clv} f(x,y)\psi(s)\left[(1-u)\beta_s^{+}(x,y)v^+ - u\beta_s^{-}(x,y)v^-\right] \Lambda(d\theta) = 0, \;\; \text{almost surely}.
		\end{equation}
		The proof of \eqref{eq:conv2} involves standard approximation arguments and is deferred to Section~\ref{sec:sub7.5}. Recalling the definition of $\tilde \Lambda_{x,y,s}$, and since $f$ and $\psi$ are arbitrary,   we now have that, a.s., for a.e.  $(x,y,s) \in [0,1]^2 \times [0,T]$, 
		\begin{equation}\label{eq:stst}
			\int_{[0,\infty)^2 \times \left\{0,1\right\}} \left[ (1-u)\beta_s^+(x,y)v^+ - u\beta_s^-(x,y)v^-\right] \tilde \Lambda_{x,y,s}(dv^+\,dv^-\,du)= 0.
		\end{equation}
	\end{enumerate}

	\subsection{Proof of the LDP upper bound}\label{sec:pfofldpupp}
	Now we complete the proof of the upper bound of the LDP in Theorem \ref{thm:weakLDP}. We begin by giving a more convenient representation for the rate function.

   Denote by $\clp^*$ the collection of all $\lambda \in \clp(\clv)$ such that $\lambda_{(4,5,6)}(dx\, dy\,ds) = dx\, dy\,ds$. 
	For $\phi \in \cW_0$, let $\clp^*(\phi)$ denote the collection of all $\lambda \in \clp^*$ such that, with
	\begin{equation}\label{eq:disinla}
		\lambda(dv^+\, dv^-\, du\,  dx\, dy\,ds) = \tilde \lambda_{x,y,s}(dv^+\, dv^-\, du) \, dx \, dy\,ds,
	\end{equation}
	we have the following : 
	\begin{equation}\label{eq:char1}
		\phi(x,y,s) = \int_{[0,\infty)^2 \times \left\{0,1\right\}} u \tilde \lambda_{x,y,s}(dv^+\, dv^-\, du), \; \; \text{a.e.}[x,y,s],
	\end{equation}
	and
	\begin{equation}\label{eq:char2}
		\int_{[0,\infty)^2 \times \left\{0,1\right\}} \left[ (1-u)\beta_s^+(x,y)v^+ - u\beta_s^-(x,y)v^-\right] \tilde \lambda_{x,y,s}(dv^+\,dv^-\,du)= 0,\;\; \text{a.e.}[x,y,s].
	\end{equation}
	From \eqref{eq:charh} and  \eqref{eq:stst}, for any weak limit point $(\bar H, \Lambda)$ of the sequence $\left\{\left(\bar H^n, \Lambda^n\right) : n \geq 1\right\}$, we have that $\Lambda \in \clp^*(\bar H)$ almost surely.

    Define $J_1^{(\beta^+,\,\beta^-)} : \cW_0 \to [0,\infty]$  as follows :
	\begin{equation}{\label{def:I1}}
		J_1^{(\beta^+,\,\beta^-)}(\phi) := \inf_{\lambda \in \clp^*(\phi)} \int_{\clv} [\ell(v^+)\beta_s^+(x,y)  +\ell(v^-) \beta_s^{-}(x,y)] \lambda(d\theta), \; \forall \; \phi \in \cW_0.
	\end{equation}
    It will be shown in Lemma \ref{lem:alleq} that $J_1^{(\beta^+,\,\beta^-)}= J^{(\beta^+,\,\beta^-)}$.
    In view of this result, for the upper bound, it suffices
	to show that (cf. \cite[Proposition 1.10(a)]{Budhirajaweakconv}) 
	$$ \limsup_{n \to \infty} \dfrac{1}{a(n)n^2}\log \E \left[ \exp\left(-a(n)n^2 \Psi(H^n)\right)\right] \leq - \inf_{\phi \in \cW_0} \left(\Psi(\phi) + J_1^{(\beta^+,\,\beta^-)}(\phi)\right).$$
Recall from \eqref{delopt}, \eqref{eq:deloptmz} that 
	\begin{align}
		&-\dfrac{1}{a(n)n^2}\log \E \left[ \exp\left(-a(n)n^2\Psi(H^n)\right)\right]+\delta \nonumber\\
		&\hspace{1 in} \ge \E\Big[ \Psi\Big(\bar H^n\Big)
		+ \int_{\clv}\left(\ell(v^+) \beta_s^{n,+}(x,y)
		+\ell(v^{-}) \beta_s^{n,-}(x,y)\right) \Lambda^n(d\theta) \Big].  \label{eq:deloptmz1}
	\end{align}
	Also recall  that the sequence $\{(\bar H^n, \Lambda^n) : n \geq 1\}$ is tight. Consider any subsequence along which $\left\{(\bar H^n, \Lambda^n) : n \geq 1 \right\}$ converges (weakly) and denote the limit by $(\bar H, \Lambda)$. Using Assumption \ref{assmp:2}\eqref{assump:1A}, we can further assume (without loss of generality) that $\beta^{n,\pm}$ converges almost everywhere to $\beta^{\pm}$ on $[0,1]^2_T$ along this subsequence.  Henceforth relabel the subsequence as $\{n\}$. From the boundedness and continuity of $\Psi$, the inequality in \eqref{tightbound}, and the observation that $\Lambda \in \clp^*\left(\bar H \right)$, we have
	\begin{align*}
		&\liminf_{n\to \infty} \E \left[\Psi\Big(\bar H^n\Big)+ \int_{\clv}\left(\ell(v^+) \beta_s^{n,+}(x,y)
		+\ell(v^{-}) \beta_s^{n,-}(x,y)\right) \Lambda^n(d\theta)\right]\\ 
		& \hspace{1 in} \ge \E \left[\Psi \Big(\bar H\Big)+\int_{\clv}\left(\ell(v^+) \beta_s^+(x,y)
		+\ell(v^{-}) \beta_s^-(x,y)\right) \Lambda(d\theta)\right]\\ 
		& \hspace{ 1in}  \ge \E \left[\Psi\Big(\bar H \Big)+ J_1^{(\beta^+,\,\beta^-)}(\bar H)\right] \ge \inf_{\phi \in \cW_0}[\Psi(\phi) + J_1^{(\beta^+,\,\beta^-)}(\phi)].
	\end{align*}
	Since the convergent subsequence is arbitrary, this, along with \eqref{eq:deloptmz1}, proves that 
	$$ \liminf_{n \to \infty} -\dfrac{1}{a(n)n^2}\log \E \left[ \exp\left(-a(n)n^2\Psi(H^n)\right)\right]+\delta \geq \inf_{\phi \in \cW_0}[\Psi(\phi) + J_1^{(\beta^+,\,\beta^-)}(\phi)].$$
	Since, $\delta >0$ is arbitrary, this completes the proof of the upper bound.

	\subsection{Equivalent representations for the rate function}\label{sec:eval} We now proceed to establish that $J_1^{(\beta^+,\,\beta^-)}$ equals $J^{(\beta^+,\,\beta^-)}$. 
	Towards that goal, we now give yet another representation for the rate function. Let 
	$$\clu := \{\alpha = (\alpha^+, \alpha^-): \alpha^{\pm} \in \cW \},$$
	where, similar to $\beta^{\pm}$, we define $\alpha^{\pm}_s(\cdot, \cdot) := \alpha^{\pm}(\cdot, \cdot, \, s)$ on $[0,1]^2$ for $s \in [0,T]$.
	For any $\phi \in \cW_0$,
	\begin{align*}
		\clu(\phi) :=  \{\alpha=(\alpha^+,\alpha^-) \in \clu \,:\, \alpha^+_s\beta^+_s(1-\phi_s)=\alpha^-_s\beta^-_s\phi_s, \; \text{a.e.}[x,y,s] \}.
	\end{align*}
	Define, $J_2^{(\beta^+,\,\beta^-)}: \cW_0  \to [0,\infty]$ as follows : For all $\phi \in \cW_0$, 
	\begin{equation}
		J_2^{(\beta^+,\,\beta^-)}(\phi) :=  \inf_{\alpha \in \clu(\phi)} \int_{[0,1]^2_T} \left[ \ell(\alpha^+)\beta^+(1-\phi)+\ell(\alpha^-)\beta^-\phi \right].
	\end{equation}
    The following lemma, which was used in the proof of the upper bound in Section \ref{sec:pfofldpupp}, shows that $J^{(\beta^+,\,\beta^-)}$,$J_1^{(\beta^+,\,\beta^-)}$, and $J_2^{(\beta^+,\,\beta^-)}$ are the same.
	\begin{lem} \label{lem:alleq} For all $\phi \in \cW_0$, $J^{(\beta^+,\,\beta^-)}(\phi)= J_1^{(\beta^+,\,\beta^-)}(\phi) = J_2^{(\beta^+,\,\beta^-)}(\phi)$.
	\end{lem}
	\begin{proof}
		Fix $\phi \in \cW_0$. We first argue that $J_1^{(\beta^+,\,\beta^-)}(\phi) \le J_2^{(\beta^+,\,\beta^-)}(\phi)$. Fix $\delta>0$ and 
		let $\alpha=(\alpha^+,\alpha^-) \in \clu(\phi)$ be $\delta$-optimal for $J_2^{(\beta^+,\,\beta^-)}(\phi)$, i.e.,
		$$
	\int_{[0,1]^2_T} \left[ \ell(\alpha^+)\beta^+(1-\phi)+\ell(\alpha^-)\beta^-\phi \right]
		\le J_2^{(\beta^+,\,\beta^-)}(\phi) + \delta .
		$$
		Define $\lambda \in \clp^*$ as
		\begin{align*}
			\lambda(dv^+\, dv^-\, du\,  dx\, dy\,ds)
			& = \delta_{1}(dv^+)\delta_{\alpha_s^-(x,y)}(dv^-) \delta_1(du) \phi_s(x,y)\, dx \,dy\,ds \\
			& \hspace{0.5 in} + \delta_{\alpha_s^+(x,y)}(dv^+)\delta_{1}(dv^-) \delta_0(du) (1-\phi_s(x,y)) dx\, dy\,ds.\end{align*}
		Then, disintegrating $\lambda$ as in \eqref{eq:disinla}, almost everywhere on $[0,1]^2_T$, 
		\begin{align*}
			&\int_{[0,\infty)^2 \times \left\{0,1\right\}} u \tilde{\lambda}_{x,y,s}\left(dv^+\,dv^-\,du\right) \\
            &= 	\phi_s(x,y)\int_{[0,\infty)^2 \times \left\{0,1\right\}} u  	\delta_{1}(dv^+)\delta_{\alpha_s^-(x,y)}(dv^-) \delta_1(du) \\
			& \hspace{1 in} + (1-\phi_s(x,y))\int_{[0,\infty)^2 \times \left\{0,1\right\}} u  \delta_{\alpha_s^+(x,y)}(dv^+)\delta_{1}(dv^-) \delta_0(du) \\
			& = \phi_s(x,y),
		\end{align*}
		and 
		\begin{align*}
			&\int_{[0,\infty)^2 \times \left\{0,1\right\}} \left[(1-u)\beta_s^+(x,y)v^+-u\beta_s^-(x,y)v^-\right] \tilde \lambda_{x,y,s}(dv^+\,dv^-\, du)\\ 
			&\hspace{0.5 in} =  \beta_s^+(x,y)\phi_s(x,y)\int_{[0,\infty)^2 \times \left\{0,1\right\}} (1-u)v^+ 	\delta_{1}(dv^+)\delta_{\alpha_s^-(x,y)}(dv^-) \delta_1(du) \\
			& \hspace{0.8 in} + \beta_s^+(x,y)(1-\phi_s(x,y))\int_{[0,\infty)^2 \times \left\{0,1\right\}} (1-u)v^+ 	\delta_{\alpha_s^+(x,y)}(dv^+)\delta_{1}(dv^-) \delta_0(du) \\
			&\hspace{0.8 in} -  \beta_s^-(x,y)\phi_s(x,y)\int_{[0,\infty)^2 \times \left\{0,1\right\}} uv^-	\delta_{1}(dv^+)\delta_{\alpha_s^-(x,y)}(dv^-) \delta_1(du) \\
			& \hspace{0.8 in} - \beta_s^-(x,y)(1-\phi_s(x,y))\int_{[0,\infty)^2 \times \left\{0,1\right\}} uv^- 	\delta_{\alpha_s^+(x,y)}(dv^+)\delta_{1}(dv^-) \delta_0(du) \\
			& \hspace{0.5 in} = \beta_s^+(x,y)(1-\phi_s(x,y))\alpha_s^+(x,y)-\beta_s^-(x,y)\phi_s(x,y)\alpha_s^-(x,y) =0,
		\end{align*}
		where the last equality follows from the fact that $\alpha \in \clu(\phi)$. This shows that $\lambda \in \clp^*(\phi)$. Consequently, since $\ell(1)=0$, we have
		\begin{align*}
			J_1^{(\beta^+,\,\beta^-)}(\phi)  &\le \int_{\clv} \left[\ell(v^+)\beta_s^+(x,y)  + \ell(v^-)\beta_s^-(x,y)\right] \lambda(d\theta) \\
            & = \int_{[0,1]^2_T} \beta^+ (1-\phi)\ell(\alpha^+) + \int_{[0,1]^2_T} \beta^-\phi\ell(\alpha^-) 
			\le J_2^{(\beta^+,\,\beta^-)}(\phi) + \delta.
		\end{align*}
		Since $\delta>0$ is arbitrary, the inequality $J_1^{(\beta^+,\,\beta^-)}(\phi) \le J_2^{(\beta^+,\,\beta^-)}(\phi)$ follows.
		
		Now we show the reverse inequality $J_2^{(\beta^+,\,\beta^-)}(\phi) \le J_1^{(\beta^+,\,\beta^-)}(\phi)$. We consider the case when $J_1^{(\beta^+,\,\beta^-)}(\phi)< \infty$, because otherwise the inequality is trivial. Fix $\lambda \in \clp^*(\phi)$ that is 
		$\delta$-optimal for $J_1^{(\beta^+,\,\beta^-)}(\phi)$, i.e.,
		$$\int_{\clv} \left[\ell(v^+)  \beta_s^+(x,y) + \ell(v^-) \beta_s^{-}(x,y)\right] \lambda(d\theta) \le J_1^{(\beta^+,\,\beta^-)}(\phi) +\delta.$$
		Now disintegrate $\lambda$ as
		\begin{align*}
			\lambda(dv^+\, dv^-\, du\,  dx\, dy\,ds) &=\tilde \lambda_{x,y,s}\left(dv^+\,dv^-\,du\right)\,dx\,dy\,ds \\
			&= \widehat \lambda_{x,y,s,u}(dv^+\, dv^-) \lambda^*_{x,y,s}(du)\,dx\, dy\,ds\\ &= \widehat \lambda_{x,y,s,0}(dv^+\, dv^-)\lambda^*_{x,y,s}(0)\delta_0(du)\,dx\, dy\,ds \\
            & \hspace{0.5 in} +
			\widehat \lambda_{x,y,s,1}(dv^+\, dv^-)\lambda^*_{x,y,s}(1)\delta_1(du)\,dx\, dy\,ds.
		\end{align*}
	Define, for all $(x,y) \in [0,1]^2$ and $s \in [0,T]$, $\phi(x,y,s) \equiv \phi_s(x,y) := \lambda^*_{x,y,s}(1) = 1- \lambda^*_{x,y,s}(0)$ and 
		\begin{align*}
			&\alpha_s^+(x,y) := \mathbbm{1}_{(\phi_s(x,y)<1)}\int_{[0,\infty)^2} v^+\widehat \lambda_{x,y,s,0}(dv^+\, dv^-),\\
			& \alpha_s^-(x,y) := \mathbbm{1}_{(\phi_s(x,y)>0)}\int_{[0,\infty)^2} v^-\widehat \lambda_{x,y,s,1}(dv^+\, dv^-),
		\end{align*}
		with the convention that $0\cdot \infty=0$.
		We claim that $\alpha = (\alpha^+, \alpha^-) \in \clu(\phi)$.  
		Indeed, using $\beta^+,\beta^- \geq c_{\beta}>0$, we get
		\begin{align*}
			\dfrac{1}{c_{\beta}} \left(J_1^{(\beta^+,\,\beta^-)}(\phi) +\delta \right) + 4 & \geq \int_{\clv} \left[ (\ell(v^+)+2)+(\ell(v^-)+2)\right]\lambda(d\theta) \\
            & \geq \int_{\clv} \left(v^++v^-\right)\, \lambda(d\theta) \\
			& \geq  \int_{[0,1]^2_T} \left[ (1-\phi_s(x,y))\int_{[0,\infty)^2} v^+\,\widehat \lambda_{x,y,s,0}\left(dv^+\, dv^-\right)\right]\,dx\,dy \,ds\\
			& \hspace{0.2 in} + \int_{[0,1]^2_T} \left[ \phi_s(x,y) \int_{[0,\infty)^2} v^-\,\widehat \lambda_{x,y,s,1}\left(dv^+\, dv^-\right)\right]\,dx\,dy\,ds \\
			& = \int_{[0,1]^2_T} (1-\phi)\alpha^+ + \int_{[0,1]^2_T} \phi\alpha^-.
		\end{align*}
		Using $J_1^{(\beta^+,\,\beta^-)}(\phi)< \infty$, this shows that $\alpha^+ < \infty$ almost everywhere on $(\phi<1)$. By definition, $\alpha^+\equiv 0$ on $(\phi=1)$ and hence we have $\alpha^+ <\infty$ almost everywhere. The same conclusion holds true for $\alpha^-$ and thus we have $\alpha=(\alpha^+,\alpha^-) \in \clu$.
		Moreover, since $\lambda \in \clp^*(\phi)$, we have a.e.$[x,y,s]$,
		\begin{align*}
		0	&= \int_{[0,\infty)^2 \times \left\{0,1\right\}} \left[ (1-u)\beta_s^+(x,y)v^+ - u\beta_s^-(x,y)v^-\right] \tilde \lambda_{x,y,s}(dv^+\,dv^-\,du) \\
		& = (1-\phi_s(x,y)) \beta_s^+(x,y)\int_{[0,\infty)^2} v^+\,\widehat \lambda_{x,y,s,0}\left(dv^+\, dv^-\right) \\
        & \hspace{1 in} - \phi_s(x,y) \beta_s^-(x,y)\int_{[0,\infty)^2} v^-\,\widehat \lambda_{x,y,s,1}\left(dv^+\, dv^-\right) \\
		&= (1-\phi_s(x,y)) \alpha_s^+(x,y) \beta_s^+(x,y) - \phi_s(x,y)\alpha^-_s(x,y) \beta_s^-(x,y),
		\end{align*}
	and thus $\alpha \in \clu(\phi)$.

			Next, by convexity of $\ell$, we see that
			$$\mathbbm{1}_{(\phi_s(x,y)<1)}\ell(\alpha_s^{+}(x,y)) \le 
			\mathbbm{1}_{(\phi_s(x,y)<1)} \int_{[0,\infty)^2} \ell(v^{+})\widehat \lambda_{x,y,s,0}(dv^+\, dv^-), $$
			$$\mathbbm{1}_{(\phi_s(x,y)>0)}\ell(\alpha_s^{-}(x,y)) \le 
			\mathbbm{1}_{(\phi_s(x,y)>0)} \int_{[0,\infty)^2} \ell(v^{-})\widehat \lambda_{x,y,s,1}(dv^+\, dv^-),
			$$
			almost everywhere. It then follows that
			\begin{align*}
				&J_2^{(\beta^+,\,\beta^-)}(\phi) \\
                &\le \int_{[0,1]^2_T} \left[\beta_s^+(x,y)(1-\phi_s(x,y)) \ell(\alpha_s^+(x,y)) + \beta_s^-(x,y)\phi_s(x,y)\ell(\alpha_s^-(x,y))\right]dx\, dy\,ds \\ 
				& \leq \int_{[0,1]^2_T} \beta_s^+(x,y)(1-\phi_s(x,y)) \int_{[0,\infty)^2} \ell(v^{+})\widehat \lambda_{x,y,s,0}(dv^+\, dv^-)\,dx\,dy\,ds \\ 
				& \hspace{1 in} + \int_{[0,1]^2_T} \beta_s^-(x,y)\phi_s(x,y) \int_{[0,\infty)^2} \ell(v^{-})\widehat \lambda_{x,y,s,1}(dv^+\, dv^-)\,dx\,dy\,ds \\
				& = \int_{[0,1]^2_T} \beta_s^+(x,y) \lambda^*_{x,y,s}(0) \left(\int_{[0,\infty)^2} \ell(v^{+})\widehat \lambda_{x,y,s,0}(dv^+\, dv^-)\right)\,dx\,dy\,ds \\ 
				& \hspace{1 in} + \int_{[0,1]^2_T} \beta_s^-(x,y)  \lambda^*_{x,y,s}(1)\left(\int_{[0,\infty)^2} \ell(v^{-})\widehat \lambda_{x,y,s,1}(dv^+\, dv^-)\right)\,dx\,dy\,ds \\
				& = \int_{\clv} \beta_s^+(x,y)\ell(v^+) \widehat \lambda_{x,y,s,0}(dv^+\,dv^-)\lambda^*_{x,y,s}(0)\delta_0(du)\,dx\, dy\,ds \\
				&\hspace{1 in} + \int_{\clv}\beta_s^-(x,y)\ell(v^-) \widehat \lambda_{x,y,s,1}(dv^+\,dv^-)\lambda^*_{x,y,s}(1)\delta_1(du)\,dx\, dy\,ds \\ 
				&\le \int_{\clv} \left[\beta_s^+(x,y) \ell(v^+) + \beta_s^-(x,y)\ell(v^-)\right] \lambda(d\theta) \le J_1^{(\beta^+,\,\beta^-)}(\phi) +\delta.
			\end{align*}
			Since $\delta > 0$ is arbitrary, we conclude $J_2^{(\beta^+,\,\beta^-)}(\phi)\leq J_1^{(\beta^+,\,\beta^-)}(\phi)$. To complete the proof we now prove that $J_2^{(\beta^+,\,\beta^-)}(\phi) = J^{(\beta^+,\,\beta^-)}(\phi)$. In order to show this we first compute, for a given   $\mathfrak{h} \in [0,1]$ and $b^+,b^- \in (0,\infty)$,
			\begin{align*}
				\zeta(\mathfrak{h},b^+,b^-)&:= \inf \left\{\ell(a^+)b^+(1-\mathfrak{h}) + \ell(a^-)b^-\mathfrak{h} \, : \, a^{\pm} \in [0, \infty), a^+b^+(1-\mathfrak{h})  = a^-b^-\mathfrak{h} \right\}.
			\end{align*}
			Note first that when $\mathfrak{h}=0$, we have $\zeta(\mathfrak{h},b^+,b^-) = b^+$ and when $\mathfrak{h}= 1$, we get
			$\zeta(\mathfrak{h},b^+,b^-) = b^-$. In both cases, the infimum is attained (though not uniquely) at $(a^+,a^-)=(0,0)$.  Now assume that $\mathfrak{h} \in (0,1)$.
			Then
			\begin{align}{\label{zetainf}}
				\zeta(\mathfrak{h},b^+,b^-)&= \inf_{c\geq 0} \left[\ell\left(\frac{cb^-}{1-\mathfrak{h}}\right)b^+(1-\mathfrak{h}) + \ell\left(\frac{cb^+}{\mathfrak{h}}\right)b^-\mathfrak{h}\right] =: \inf_{c > 0} \zeta_0(\mathfrak{h},b^+,b^-,c),
			\end{align}    
			where the infimum can be restricted to only $c>0$ since the map $c \mapsto \zeta_0(\mathfrak{h},b^+,b^-,c)$ is continuous on $[0,\infty)$. Note that $\ell$ is strictly convex on $[0,\infty)$ and continuously differentiable on $(0,\infty)$. Moreover,
			$$\dfrac{\partial}{\partial c}\zeta_0(\mathfrak{h},b^+,b^-, c) = b^+b^- \log \left(\frac{cb^-}{1-\mathfrak{h}}\right) + b^+b^-\log \left(\frac{cb^+}{\mathfrak{h}}\right).$$
			Setting $\dfrac{\partial}{\partial c}\zeta_0(\mathfrak{h},b^+,b^- c) =0$ and by strict convexity of $\zeta_0$ in $c$, we conclude that the unique infimum in \eqref{zetainf} is attained at
			$$ c= \left(\frac{\mathfrak{h}(1-\mathfrak{h})}{b^+b^-}\right)^{1/2} \in (0,\infty).$$
			Substituting this value of $ c$ in the definition of $\zeta_0(\mathfrak{h},b^+,b^-, c)$, we see that
			\begin{align}
				\zeta(\mathfrak{h},b^+,b^-) &= 
				\dfrac{1}{2}\sqrt{b^+b^-\mathfrak{h}(1-\mathfrak{h})}
				\log \left(\frac{\mathfrak{h}b^-}{(1-\mathfrak{h})b^+}\right) - 	\sqrt{b^+b^-\mathfrak{h}(1-\mathfrak{h})} + b^+(1-\mathfrak{h})\nonumber \\ 
				&\hspace{0.4 in} + 	\dfrac{1}{2}\sqrt{b^+b^-\mathfrak{h}(1-\mathfrak{h})}
				\log \left(\frac{(1-\mathfrak{h})b^+}{\mathfrak{h}b^-}\right)
				-	\sqrt{b^+b^-\mathfrak{h}(1-\mathfrak{h})} + b^-\mathfrak{h} \nonumber \\ 
				&= \left(\sqrt{b^+(1-\mathfrak{h})} - \sqrt{b^-\mathfrak{h}}\right)^2. \label{zeta}
			\end{align}
			The optimal value of $(a^+,a^-)$ giving the infimum in the definition of $\zeta$ can easily seen to be 
			$$(a^+,a^-)= \left(\sqrt{\frac{\mathfrak{h}b^-}{(1-\mathfrak{h})b^+}}\,,\, \sqrt{\frac{(1-\mathfrak{h})b^+}{\mathfrak{h}b^-}}\right).$$
			Note that the value of $ \zeta(\mathfrak{h},b^+,b^-)$ computed in \eqref{zeta} agrees with the case when $\mathfrak{h}\in \left\{0,1\right\}$. Therefore,
			\begin{align*}
			&	J_2^{(\beta^+,\,\beta^-)}(\phi) \\
			&=  \inf \left\{ \int_{[0,1]^2_T} \left[\ell(\alpha^+)\beta^+(1-\phi)+\ell(\alpha^-)\beta^-\phi\right] : \alpha^{\pm} \geq 0, \alpha^+\beta^+(1-\phi)=\alpha^-\beta^-\phi \text{ a.e.}\right\}\\
				& =\int_{[0,1]^2_T} \zeta\left(\phi_s(x,y),\beta_s^+(x,y),\beta_s^-(x,y)\right)\,dx\,dy\,ds \\
				&=  \int_{[0,1]^2_T} \left( \sqrt{\beta^+(1-\phi)}-\sqrt{\beta^-\phi}\right)^2 = J^{(\beta^+,\,\beta^-)}(\phi),
			\end{align*}
			where the second equality is true since the infimum in the definition of $\zeta(\mathfrak{h},b^+,b^-)$ is always attained at some point whose value depends on $(\mathfrak{h},b^+,b^-)$ in a measurable way. This completes the proof. 
			\end{proof}

		\subsection{Proof of the LDP lower bound}\label{sec:pflowbd}
		We now give the proof of the LDP lower bound in Theorem \ref{thm:weakLDP}. 
		The approach is as follows : Let $\mu_n$ be the Borel probability measure on $\cW_0$ induced by the random variable $H^n$. 
		Fix a $\phi \in \cW_0$ such that  $\phi$ is continuous on $[0,1]^2 \times [0,T]$ and $\delta \leq \phi \leq 1-\delta$, for some $\delta \in (0,1/2)$. Suppose, for any such $\phi$ we can define another sequence of Borel probability measures $\left\{\nu_{n,\phi} : n \geq 1\right\}$ on $\cW_0$ satisfying the following two conditions.
		\begin{enumerate}
			\item \label{condition:1} For any $r>0$,
			$$ \nu_{n,\phi} \left( \cW_0 \setminus B_{\cW_0,\text{weak}}(\phi,r)  \right) \longrightarrow 0, \;\text{ as } n \to \infty,$$
			where $B_{\cW_0, \text{weak}}(\phi,r) :=  \left\{\phi^{\prime} \in \cW_0 \, :  d_{\cW_0, \text{weak}}(\phi,\phi^{\prime}) < r\right\}$, the open ball of radius $r$ around $\phi$.
			\item \label{condition:2} Let $\mathcal{R} \left(\nu_{n,\phi} || \mu_n\right)$ denotes the relative entropy between the probability measures $\nu_{n,\phi}$ and $\mu_n$. Then 
			$$ \limsup_{n \to \infty} \dfrac{1}{n^2a(n)}\mathcal{R} \left(\nu_{n,\phi} || \mu_n\right) \leq J^{(\beta^+,\,\beta^-)}(\phi).$$
		\end{enumerate}
        Then we can complete the proof of the lower bound as follows.
		 Fix a bounded continuous (with respect to the weak topology) $\Psi : \cW_0 \to \mathbb{R}$. Then for $\phi$ as above, applying the variational formula for relative entropy (see \cite[Proposition 2.2]{Budhirajaweakconv}), 
		\begin{align*}
			- \log \E \exp\left[ -a(n)n^2 \Psi(H^n)\right] & = - \log \int_{\cW_0} \exp\left[ - a(n)n^2\Psi(\phi^{\prime})\right] d\mu_n(\phi^{\prime}) \\
			&\leq \mathcal{R} \left(\nu_{n,\phi} || \mu_n\right) + n^2a(n)\int_{\cW_0} \Psi(\phi^{\prime})\, d\nu_{n,\phi}(\phi^{\prime}).
			\end{align*}
		By continuity of $\Psi$ and property~(\ref{condition:1}), we have 
		$$ \int_{\cW_0} \Psi(\phi^{\prime})\, d\nu_{n,\phi}(\phi^{\prime}) \longrightarrow \Psi(\phi), \; \text{ as } n \to \infty,$$
		and hence 
			\begin{align}
			\limsup_{n \to \infty} - \dfrac{1}{n^2a(n)} \log \E \exp\left[ -a(n)n^2 \Psi(H^n)\right] &\leq  \limsup_{n \to \infty} \dfrac{1}{n^2a(n)}\mathcal{R} \left(\nu_{n,\phi} || \mu_n\right) + \Psi(\phi) \nonumber\\
			& \leq  J^{(\beta^+,\,\beta^-)}(\phi) + \Psi(\phi).	\label{eq:953}		
            \end{align}
	We now argue that the above inequality  in fact holds for any $\phi \in \cW_0$.
Fix $\phi \in \cW_0$. Recall that, from Assumption \ref{assmp:2}\eqref{assump:1A}, $\beta^{\pm} \in L^{1+\eta}\left([0,1]^2_T\right)$ for some $\eta>0$. Without loss of generality we take $\eta \in (0, 1/2)$. Set $p = 1+1/\eta, q = 1+\eta$. Fix $\delta \in (0,1/2)$ and let $\tilde \phi :[0,1]^2_T \to [0,1]$ be continuous such that 
	$$\|\phi - \tilde\phi\|_p := \left( \int |\phi - \tilde \phi \,|^p \right)^{1/p} \leq \delta.$$ 
	Set $ \tilde \phi^{(\delta)} := \left( \tilde \phi \wedge (1-\delta) \right) \vee \delta.$
		Clearly, $|\tilde \phi- \tilde \phi^{(\delta)}| \leq \delta$, and hence  $\| \phi -\tilde \phi^{(\delta)}\|_p \leq (T^{1/p}+1)\delta \leq 2(T \wedge 1)\delta$. Moreover, since the function $x \mapsto x(1-x)$ is $1$-Lipschitz on $[0,1]$ and has a maximum value of $\delta(1-\delta)$ on $[\delta, 1-\delta]$, 
		$$ \Big \rvert \sqrt{\phi(1-\phi)}-\sqrt{\tilde \phi^{(\delta)}(1-\tilde \phi^{(\delta)})} \Big \rvert = \dfrac{|\phi(1-\phi)-\tilde \phi^{(\delta)}(1-\tilde \phi^{(\delta)})| }{\sqrt{\phi(1-\phi)}+\sqrt{\tilde \phi^{(\delta)}(1-\tilde \phi^{(\delta)})}} \leq \dfrac{|\phi - \tilde \phi^{(\delta)}|}{\sqrt{\delta(1-\delta)}},$$
		and thus 
		$$ \Big \rvert \!\Big \rvert \sqrt{\phi(1-\phi)}-\sqrt{\tilde \phi^{(\delta)}(1-\tilde \phi^{(\delta)})}  \Big \rvert \! \Big \rvert_p \leq \dfrac{2(T \wedge 1)\sqrt{\delta}}{\sqrt{1-\delta}} \leq 4(T \wedge 1) \sqrt{\delta}. $$
		This yields,
		\begin{align*}
	&	\int_{[0,1]^2_T}\Bigg \rvert  \left( \sqrt{(1-\phi)\beta^+} - \sqrt{\phi \beta^-}\right)^2  -  \left( \sqrt{(1- \tilde \phi^{(\delta)})\beta^+} - \sqrt{\tilde \phi^{(\delta)} \beta^-}\right)^2  \Bigg \rvert \\
	& \leq  	\int_{[0,1]^2_T} \left( \beta^+ + \beta^- \right) \big \rvert \phi - \tilde \phi^{(\delta)} \big \rvert +  	\int_{[0,1]^2_T} (\beta^++\beta^-)\; \bigg \rvert \sqrt{\phi(1-\phi)}-\sqrt{\tilde \phi^{(\delta)}(1-\tilde \phi^{(\delta)})} \bigg \rvert \\
	&   \leq   \|\beta^{+}+\beta^{-}\|_{q}\| \phi -\tilde \phi^{(\delta)}\|_p +   \|\beta^{+}+\beta^{-}\|_{q} \Big \rvert \!\Big \rvert \sqrt{\phi(1-\phi)}-\sqrt{\tilde \phi^{(\delta)}(1-\tilde \phi^{(\delta)})}  \Big \rvert \! \Big \rvert_p \\
	& \leq 6(T \wedge 1)\;\|\beta^{+}+\beta^{-}\|_{q} \sqrt{\delta},
		\end{align*} 
			where the first inequality uses the bound $2\sqrt{xy} \le x+y$ for $x,y \ge 0$.
    
	Hence, since $\tilde \phi^{(\delta)}$ is continuous and takes values in  $[\delta, (1-\delta)]$, using \eqref{eq:953}, and recalling that $p\ge 3$, we  conclude that for any $\phi \in \cW_0$ and $\delta \in (0,1/2)$, 
		\begin{align*}
				&\limsup_{n \to \infty} - \dfrac{1}{n^2a(n)} \log \E \exp\left[ -a(n)n^2 \Psi(H^n)\right] \\
				&\leq  J^{(\beta^+,\,\beta^-)}(\phi)+ \Psi(\phi) + 6(T \wedge 1)\;\|\beta^{+}+\beta^{-}\|_{q} \sqrt{\delta} + \sup_{\phi^* : \| \phi - \phi^* \|_3 \leq 2 \delta} |\Psi(\phi)-\Psi(\phi^*)|.
		\end{align*}
		Taking $\delta \downarrow 0$, recalling that $\Psi$ is continuous (with respect to the weak topology and hence with respect to the $L^3$-metric) and $\beta^{\pm} \in L^{q}\left([0,1]^2_T\right)$, we have the large deviation lower bound (cf.  \cite[Proposition 1.10(a)]{Budhirajaweakconv}).

Thus it only remains now  to construct $\left\{\nu_{n,\phi} : n \geq 1\right\}$ satisfying properties ~(\ref{condition:1}) and~(\ref{condition:2}) for any $\phi \in \cW_0$ which is continuous and takes values in $[\delta, (1-\delta)]$ for some $\delta \in (0, 1/2)$.

Towards that goal, recall the representation for the evolution of  $\{X^n_{i,j}(t), t \in [0,1]\}_{i,j \in [n]}$ given in \eqref{evolution}. It is easy to see that one can give the following distributionally equivalent representation. 
		\begin{align}
			X_{i,j}^n(t) 
            &=x_{i,j}^n + \int_{[0,t]\times [0,\infty) \times \left\{0,1\right\}} (-1)^{X_{i,j}^n(s-)}\mathbbm{1}_{(X_{i,j}^n(s-)=u)} \label{pois2} \\
            & \hspace{1 in} \mathbbm{1}_{[0,a(n)\left((1-u)\beta^{n,+}_{i,j}(s)+u\right)\left(u\beta^{n,-}_{i,j}(s)+1-u\right)]}(r) N_{i,j}^n(ds\,dr\,du),   \nonumber 
		\end{align}
		where $\left\{N_{i,j}^n : i,j \in [n]\right\}$ is a collection of $n^2$ many independent Poisson random measures on $[0,T]\times [0,\infty) \times \left\{0,1\right\}$ with intensity measure being $\operatorname{Leb}\left([0,T]\right) \otimes \operatorname{Leb}\left([0,\infty)\right) \otimes \operatorname{Count}\left(\left\{0,1\right\}\right)$. Here $\operatorname{Count}\left(\left\{0,1\right\}\right)$ denotes the counting measure on $\left\{0,1\right\}$. Set $\clx:=[0,T]\times \left\{0,1\right\}$ and $\lambda:= \operatorname{Leb}\left([0,T]\right) \otimes \operatorname{Count}\left(\left\{0,1\right\}\right)$. The expression in \eqref{pois2} can be further simplified to 
		\begin{align}
			X_{i,j}^n(t)=x_{i,j}^n + \int_{[0,t]\times \left\{0,1\right\}} \left(-1\right)^{X_{i,j}^n(s-)} \mathbbm{1}_{(X_{i,j}^n(s-)=u)} \mathfrak{N}^n_{i,j}(ds\,du) \label{pois3},
		\end{align}
		where $\left\{\mathfrak{N}_{i,j}^n : i,j \in [n]\right\}$ is a collection of $n^2$ many independent Poisson random measures on $\clx$ with intensity measure $\lambda_{i,j}^n$. Here $\lambda_{i,j}^n$ is a measure with Radon-Nikodym derivative with respect to $\lambda$  given as follows : For any $i,j \in [n], s \in [0,T], u \in \left\{0,1\right\},$ 
		$$ \dfrac{d \lambda_{i,j}^n}{d \lambda} (s,u) = a(n)\left((1-u)\beta^{n,+}_{i,j}(s)+u\right)\left(u\beta^{n,-}_{i,j}(s)+1-u\right) := U_{i,j}^n (s,u).$$ 
		Recall $H^n$ from \eqref{eq:hndefn}
		and that $\mu_n$ is the probability measure induced by $H^n$. Now fix $\phi \in \cW_0$ which is continuous and takes values in $[\delta, (1-\delta)]$ for some $\delta \in (0, 1/2)$.
		 Set,
		$$ \alpha^+ := \sqrt{\dfrac{\phi\beta^-}{(1-\phi)\beta^+}}, \; \alpha^- := \sqrt{\dfrac{(1-\phi)\beta^+}{\phi\beta^-}},$$
		with $0/0:= 0$ and for any $i,j \in [n]$, define
		$$   \alpha_{i,j}^{n,+}(s) := \sqrt{\dfrac{\phi^n_{i,j}(s)\beta^{n,-}_{i,j}(s)}{\left(1-\phi^n_{i,j}(s)\right)\beta^{n,+}_{i,j}(s)}}, \;  \alpha_{i,j}^{n,-}(s) := \sqrt{\dfrac{\left(1-\phi^n_{i,j}(s)\right)\beta^{n,+}_{i,j}(s)}{\phi^n_{i,j}(s)\beta^{n,-}_{i,j}(s)}},$$
        and 
        $$\alpha_s^{n,\pm} := \sum_{i,j=1}^n  \alpha_{i,j}^{n,\pm}(s) \mathbbm{1}_{Q_{i,j}^n},$$
		where 
		$$ \phi_{i,j}^n(s) := n^2 \int_{Q_{i,j}^n} \phi_s, \; \; \phi^n_s := \sum_{i,j=1}^n \phi^n_{i,j}(s)\mathbbm{1}_{Q_{i,j}^n}, \;\; \phi^n(\cdot,\cdot,s) := \phi^n_s(\cdot,\cdot), \; \forall \; s \in [0,T].$$
		The  above quantities are  well-defined since $\delta \leq \phi, \phi^n \leq 1-\delta$ and $\beta^{n,\pm} \geq c_{\beta}$ almost everywhere. Clearly, $\alpha:=(\alpha^+,\alpha^-) \in \clu(\phi)$ and $\alpha^+\beta^+/(\alpha^+\beta^++\alpha^-\beta^-) = \phi$.  
        Since $\phi$ is continuous, we have  $\phi^n \stackrel{L^{\infty}}{\longrightarrow} \phi$. 
        Now set, 
		$$  \upsilon_{i,j}^{n,\pm}(s) := \alpha^{n,\pm}_{i,j}(s)\beta^{n,\pm}_{i,j}(s), \;\; \upsilon^{n,\pm}_{s} :=  \sum_{i,j=1}^n \upsilon^{n,\pm}_{i,j}(s) \mathbbm{1}_{Q_{i,j}^n}, \; \upsilon^{n,\pm}(\cdot,\cdot,s):=\upsilon^{n,\pm}_s(\cdot,\cdot).$$
		Since  $\sqrt{\beta^{n,+}\beta^{n,-}}$ converges in $L^1$ to $\sqrt{\beta^+\beta^-}$ from Assumption~\ref{assmp:2}\eqref{assump:1A} and
         $\phi^n/(1-\phi^n)$ converges in $L^{\infty}$ to $\phi/(1-\phi)$ (which follows since $\phi^n \stackrel{L^{\infty}}{\longrightarrow}\phi$ and $\delta \leq \phi,\phi^n \leq 1-\delta$), we have
		\begin{align*}
		\upsilon^{n,+} = \sqrt{\dfrac{\phi^n \beta^{n,+}\beta^{n,-}}{(1-\phi^n)}} \stackrel{L^1}{\longrightarrow}  \sqrt{\dfrac{\phi \beta^{+}\beta^{-}}{(1-\phi)}} =: \upsilon^+,
		\end{align*}
		 The same is true for $\upsilon^{n,-}$ as well, i.e.,
         $$ \upsilon^{n,-} = \sqrt{\dfrac{(1-\phi^n) \beta^{n,+}\beta^{n,-}}{\phi^n}} \stackrel{L^1}{\longrightarrow}  \sqrt{\dfrac{(1-\phi) \beta^{+}\beta^{-}}{\phi}} =: \upsilon^-.$$
         Also, $ \upsilon_{i,j}^{n,\pm}(s) \geq c_{\beta}\sqrt{\delta/(1-\delta)} >0,$ for all $i,j \in [n]$ and $s \in [0,T]$, and
        $$ \sup_{t \in [0,T]} \upsilon^{\pm}_t \le \dfrac{\sqrt{1-\delta}}{2\sqrt{\delta}} \sup_{t \in [0,T]} \left(\beta_t^+ + \beta_t^- \right), \; \sup_{t \in [0,T]} \upsilon^{n,\pm}_{i,j}(t) \le \dfrac{\sqrt{1-\delta}}{2\sqrt{\delta}} \sup_{t \in [0,T]} \left(\beta_{i,j}^{n,+}(t) + \beta_{i,j}^{n,-}(t) \right). $$ 
        Continuity and boundedness of $\phi$ implies that $t \mapsto \phi_{i,j}^n(t)$ is continuous. Since, $t \mapsto \beta_{i,j}^{n,\pm}(t)$ is left-continuous (by Assumption~\ref{assmp:1}\eqref{assump:1D}), the same is true for $t \mapsto \upsilon^{n,\pm}_{i,j}(t)$.  Finally, since $\phi$ is continuous and $\beta^{\pm}$ is left-continuous in the time variable, the latter property also holds for $\upsilon^{\pm}$. All these observations together imply that Assumption~\ref{assmp:1} is satisfied with $(\beta^{n, \pm}, \beta^{\pm})$ replaced with $(\upsilon^{n, \pm}, \upsilon^{\pm})$. We will now apply Theorem \ref{cor2:thmlin} for $H^n$ replaced by $\bar H^n$, where the latter is defined as in \eqref{eq:barhn} with $\bar X^n_{i,j}(t)$, for any $i,j \in [n]$ and $t \geq 0$, defined as
		\begin{align}
			\bar X_{i,j}^n(t)&= x_{i,j}^n + \int_{[0,t]\times [0,\infty) \times \left\{0,1\right\}} (-1)^{\bar X_{i,j}^n(s-)}\mathbbm{1}_{(\bar X_{i,j}^n(s-)=u)} \label{pois4} \\
            & \hspace{1 in}\mathbbm{1}_{\left[0,a(n)\left((1-u)\upsilon^{n,+}_{i,j}(s)+u\right)\left(u \upsilon^{n,-}_{i,j}(s)+1-u\right)\right]}(r)
			  N_{i,j}^n(ds\,dr\,du). \nonumber
		\end{align}
        By Theorem~\ref{cor2:thmlin}, we now have that 
		\begin{equation}\label{eq:1101}
        \bar H^n \stackrel{a.s.}{\longrightarrow}  \dfrac{\upsilon^+}{\upsilon^+ + \upsilon^-} =  \phi,\end{equation}
		with respect to the weak topology. 
        
        Let $\nu_{n,\phi}$ be the probability measure induced by $\bar H^n$.
       Then, from \eqref{eq:1101}, we have that  $\nu_{n,\phi}\left( \cW_0 \setminus B_{\cW_0, \text{weak}}(\phi,r) \right)\longrightarrow 0$ as $n \to \infty$ for any $r>0$. 
        This shows that property \eqref{condition:1}, stated at the beginning of this section, holds with the above choice of $\nu_{n,\phi}$.
        
		We now verify the second property for $\{\nu_{n,\phi}\}$. For this we will write a simpler form of \eqref{pois4} similar to \eqref{pois3}. Define for all $i,j \in [n], s\in[0,T], u \in \left\{0,1\right\}$, 
		\begin{align}
			\bar U_{i,j}^n(s,u) &:= a(n)\left((1-u)\upsilon^{n,+}_{i,j}(s)+u\right) \mathbbm{1}_{\left(\bar X_{i,j}^n(s-)=u=0\right)} \label{barudef} \\
            &\quad + a(n)\left(u\upsilon^{n,-}_{i,j}(s)+1-u\right) \mathbbm{1}_{\left(\bar X_{i,j}^n(s-)=u=1\right)}   + U_{i,j}^n(s,u)  \mathbbm{1}_{\left(\bar X_{i,j}^n(s-)\neq u\right)}. \nonumber
		\end{align} 
		One can check that for any $(s,r,u) \in [0,T]\times [0,\infty) \times \left\{0,1\right\}$, we have 
		\begin{align*}
			&(-1)^{\bar X_{i,j}^n(s-)}\mathbbm{1}_{(\bar X_{i,j}^n(s-)=u)} \mathbbm{1}_{\left[0,a(n)\left((1-u)\upsilon^{n,+}_{i,j}(s)+u\right)\left(u\upsilon^{n,-}_{i,j}(s)+1-u\right)\right]}(r) \\ &\hspace{ 1in }= (-1)^{\bar X_{i,j}^n(s-)}\mathbbm{1}_{(\bar X_{i,j}^n(s-)=u)} \mathbbm{1}_{\left[0,\bar U_{i,j}^n(s,u)\right]}(r),
		\end{align*}  
		and hence 
		\begin{align}
			\bar X_{i,j}^n(t) & =  x_{i,j}^n + \int_{[0,t]\times [0,\infty)\times \left\{0,1\right\}} (-1)^{\bar X_{i,j}^n(s-)}\mathbbm{1}_{(\bar X_{i,j}^n(s-)=u)} \mathbbm{1}_{\left[0,\bar U_{i,j}^n(s,u)\right]}(r)\, N_{i,j}^n(ds\,dr\,du), \nonumber \\
			& = x_{i,j}^n + \int_{[0,t]\times \left\{0,1\right\}} \left(-1\right)^{\bar X_{i,j}^n(s-)} \mathbbm{1}_{(\bar X_{i,j}^n(s-)=u)} \bar{\mathfrak{N}}^n_{i,j}(ds\,du), \nonumber
		\end{align}
		where $\left\{\bar{\mathfrak{N}}_{i,j}^n : i,j \in [n]\right\}$ is a collection of $n^2$ many independent Poisson random measures on $[0,T]\times \left\{0,1\right\}$ with intensity measure $\bar \lambda_{i,j}^n$ with 
		$$ \dfrac{d \bar \lambda_{i,j}^n}{d \lambda} (s,u) = \bar U_{i,j}^n (s,u), \;\; \forall \; s \in [0,T], u \in \left\{0,1\right\}.$$

       Thus
		\begin{multline*}
			\mathcal{R} \left(\nu_{n,\phi}\|\mu_n\right) \leq \mathcal{R} \left( \mathcal{L}\left( \left\{\bar{\mathfrak{N}}^n_{i,j} : i,j \in [n]\right\}\right)\,\| \,\mathcal{L}\left( \left\{\mathfrak{N}^n_{i,j} : i,j \in [n]\right\}\right)\right) \\
			 = \sum_{i,j=1}^n \mathcal{R} \left( \mathcal{L}\left(\bar{\mathfrak{N}}^n_{i,j}\right)\, \| \, \mathcal{L}\left(\mathfrak{N}^n_{i,j}\right)\right) =\sum_{i,j=1}^n \E \int_{[0,T]\times \left\{0,1\right\}} \ell\left(\dfrac{\bar U^n_{i,j}(s,u)}{U^n_{i,j}(s,u)}\right)\, \lambda_{i,j}^n(ds\,du)  \\
			=\sum_{i,j=1}^n \E \int_{[0,T]\times \left\{0,1\right\}} \ell\left(\dfrac{\bar U^n_{i,j}(s,u)}{U^n_{i,j}(s,u)}\right)\, U_{i,j}^n(s,u)\,\lambda(ds\,du). 
            \end{multline*}
            Consequently, 
            \begin{align}
            &\mathcal{R} \left(\nu_{n,\phi}\|\mu_n\right) \nonumber \\
            &\leq \sum_{i,j=1}^n \E \left[ \int_{[0,T]} \ell\left(\dfrac{\bar U^n_{i,j}(s,0)}{U^n_{i,j}(s,0)}\right)\, U_{i,j}^n(s,0)\,ds + \int_{[0,T]} \ell\left(\dfrac{\bar U^n_{i,j}(s,1)}{U^n_{i,j}(s,1)}\right)\, U_{i,j}^n(s,1)\,ds\right] \nonumber \\	
			&= a(n)\sum_{i,j=1}^n \E  \int_{[0,T]}  \left[\ell\left(\alpha^{n,+}_{i,j}(s)\right)\, \mathbbm{1}_{\left(\bar X_{i,j}^n(s-)=0\right)}\beta_{i,j}^{n,+}(s) \right]\,ds \nonumber \\
            & \hspace{0.5 in} + a(n)\sum_{i,j=1}^n \E  \int_{[0,T]}  \left[\ell\left( \alpha^{n,-}_{i,j}(s)\right)\, \mathbbm{1}_{\left(\bar X_{i,j}^n(s-)=1\right)}\beta_{i,j}^{n,-}(s)\right]\,ds  \nonumber \\
			&= a(n)\sum_{i,j=1}^n \int_{[0,T]} \left[  (1-\bar P_{i,j}^n(s))\ell\left( \alpha_{i,j}^{n,+}(s)\right) \beta_{i,j}^{n,+}(s) + \bar P_{i,j}^n(s)\ell\left( \alpha_{i,j}^{n,-}(s)\right) \beta_{i,j}^{n,-}(s)\right]\,ds \nonumber   \\
			&= n^2a(n) \int_{[0,1]^2_T} \left[ (1-\bar \mu^n) \ell\left(\alpha^{n,+}\right)\beta^{n,+} + \bar \mu^n \ell\left(\alpha^{n,-}\right)\beta^{n,-} \right],\label{entropybound1}
		\end{align} 
		where $\bar P^n_{i,j}(s) = \E \bar X_{i,j}^n(s)$ and $\bar \mu^n_s = \sum_{i,j=1}^n \bar P_{i,j}^n(s)\mathbbm{1}_{Q_{i,j}^n}$. 
        If we apply Lemma~\ref{lem:es12} with $(X^n, H^n)$ replaced with $(\bar X^n, \bar H^n)$ and recall that Assumption~\ref{assmp:1} is satisfied in the latter setting (i.e., with $(\beta^{n, \pm}, \beta^{\pm})$ replaced with $(\upsilon^{n, \pm}, \upsilon^{\pm})$), we can conclude that  $\bar \mu^n$ converges in measure to $\phi$. Also, by Assumption~\ref{assmp:1}, $\beta^{n,\pm}$ converges in measure to $\beta^{\pm}$ and therefore $\alpha^{n,\pm}$ also converges in measure to $\alpha^{\pm}$ (here we have also used $\phi^n \stackrel{L^{\infty}}{\longrightarrow} \phi$, $\delta \le \phi \le 1-\delta$ and $\beta^{\pm} \ge c_{\beta}$).
        Hence the integrand in \eqref{entropybound1} converges in measure to 
		\begin{equation}
        (1-\phi) \ell\left( \alpha^{+}\right) \beta^+ + \phi \ell \left( \alpha^-\right) \beta^- = \left( \sqrt{(1-\phi)\beta^+} - \sqrt{\phi \beta^-}\right)^2.\label{eq:1211}
        \end{equation}
	We claim that  this convergence also happens in $L^1$. We present the argument for the first integrand in \eqref{entropybound1}. The argument for the second integrand in \eqref{entropybound1} is exactly the same. Recall that, $0 \leq \bar \mu^n \leq 1$ for all $n \geq 1$. It is, therefore, enough to show that the sequence of random variables $\left\{\ell\left(\alpha^{n,+}\right)\beta^{n,+} : n \geq 1\right\}$ is uniformly integrable. From Assumption \ref{assmp:2}\eqref{assump:1A}, we have $\left\{(\beta^{n,\pm})^{1+2\varepsilon} : n \geq 1\right\}$ is uniformly integrable for some $ \varepsilon \in (0,1/2)$. Moreover, for some constants $c_i$ independent of $(t,x,y)$ and $n$,
    \begin{multline*}
\ell( \alpha^{n,+}) \beta^{n,+} \le \left(\ell\left( c_1 \frac{(\beta^{n,-})^{1/2}}{(\beta^{n,+})^{1/2}}\right) +1\right) \beta^{n,+} \le
c_2\left(1+ \frac{(\beta^{n,-})^{\varepsilon + 1/2}}{(\beta^{n,+})^{\varepsilon + 1/2}}\right) \beta^{n,+}\\
\le c_3 (\beta^{n,+} + (\beta^{n,-})^{\varepsilon + 1/2}(\beta^{n,+})^{1/2 -\varepsilon}) \le c_4(1+ \beta^{n,+} +  (\beta^{n,-})^{2\varepsilon + 1}).
\end{multline*}
The claimed uniform integrability is now immediate.

Taking limit superior in \eqref{entropybound1}, and recalling \eqref{eq:1211}, we now have 
	\begin{equation*}
		 \limsup_{n \to \infty} \dfrac{1}{n^2a(n)} 	\mathcal{R} \left(\nu_{n,\phi}\|\mu_n\right)  \leq  \int_{[0,1]^2_T} \left( \sqrt{(1-\phi)\beta^+} - \sqrt{\phi \beta^-}\right)^2.
	\end{equation*}
    This completes the proof that with our choice of $\nu_{n,\phi}$ property \eqref{condition:2}, stated at the beginning of this section, holds as well and concludes the proof of the lower bound.

		\section{LDP under the cut-norm topology}\label{sec:cutldp}
	In this section we will prove Theorem~\ref{thm:LDP-avg-path} and Proposition~\ref{prop:LDP-avg-path-finite}. Throughout the section Assumption \ref{assmp:2} will be in force and will not be explicitly noted in the statement of various results.

    We start with the proof of Proposition~\ref{prop:LDP-avg-path-finite}. We shall only prove the one-dimensional version, i.e., the $|\cT|=1$ case. The multi-dimensional case follows from exactly the same arguments. To ease the notational burden, we shall prove the LDP for the time-average graphon over the whole time-period $[0,T]$, as given in the following proposition. The general version can be obtained by a time rescaling.

    \begin{prop}{\label{prop:cutLDP-avg-whole}}
    	Define the time-averaged graphon 
    	$$ M^n := \dfrac{1}{T}\int_0^T H^n_s\,ds.$$
    	Then, the sequence $\{\widehat{M^n} : n \geq 1\}$ satisfies a large deviation principle in $\left(\widehat{\cls_0}, \delta_{\square}\right)$ with speed $a(n)n^2$ and rate function $\widehat I^{(\beta^+,\;\beta^-)}_M$ defined as
    	$$ \widehat I^{(\beta^+,\;\beta^-)}_M \left(\widehat f\right) := \sup_{\eta >0} \;\; \inf \left\{I_M^{(\beta^+,\,\beta^-)}(g) : \delta_{\square}\left(\widehat g, \widehat f\right) < \eta \right\},  \;\; \forall \; \widehat f\in \widehat{\cls_0}.$$
    	where
    	\begin{align*}
    		I^{(\beta^+,\;\beta^-)}_M(g) := \inf \left\{ J^{(\beta^+,\,\beta^-)} (\phi) : g=\dfrac{1}{T}\int_0^T \phi_s\,ds \right\}, \;\;  \forall \; g \in \cls_0.
    	\end{align*}
    \end{prop}

    We begin by observing that, by definition, $\widehat I^{(\beta^+,\;\beta^-)}_M$ is lower semi-continuous, as it is the lower semi-continuous envelope of the function $\widehat f \mapsto \inf_{f \in \widehat f} I^{(\beta^+,\;\beta^-)}_M(f)$ 
    on $\widehat{\cls_0}$. Since $\widehat{\cls_0}$ is compact, this shows that $\widehat I^{(\beta^+,\;\beta^-)}_M$ is a rate function.  Rest of the proof is organized as follows. In Section \ref{sec-lowbdcut} we prove the  lower bound for the LDP in Proposition~\ref{prop:cutLDP-avg-whole}.  The corresponding upper bound is proved in Section \ref{sec-ldpuppcut}. The two sections together complete the proof of Proposition~\ref{prop:cutLDP-avg-whole}, and hence Proposition~\ref{prop:LDP-avg-path-finite}.

	\subsection{Proof of the lower bound} \label{sec-lowbdcut} We now prove the lower bound in the LDP stated in Proposition~\ref{prop:cutLDP-avg-whole}. Take any bounded continuous (with respect to the topology induced by the metric $d_{\square}$) function $G : \cls_0 \to \mathbb{R}$. Define $\Psi: \cW_0 \to \mathbb{R}$ as 
	$$\Psi(\phi) := G\left(\dfrac{1}{T}\int_0^T \phi_s ds\right).$$
    If $\phi$ is continuous on $[0,1]^2 \times [0,T]$ with $\delta \le \phi \le 1-\delta$ for some $\delta \in (0,1/2)$, then, exactly as in Section \ref{sec:pflowbd},  with $\nu_{n,\phi}$ defined  in that section,
    \begin{align}
			\limsup_{n \to \infty} - \dfrac{1}{n^2a(n)} \log \E \exp\left[ -a(n)n^2 G(M^n)\right] &\leq  \limsup_{n \to \infty} \dfrac{1}{n^2a(n)}\mathcal{R} \left(\nu_{n,\phi} || \mu_n\right) + \Psi(\phi) \nonumber\\
			& \leq  J^{(\beta^+,\,\beta^-)}(\phi) + \Psi(\phi).	\label{eq:953n}		
            \end{align}
            Here the first inequality is obtained by observing that under the probability measure $\nu_{n,\phi}$, the random variable $M^n$ converges to $(\int_0^T \phi_s\,ds)/T$ with respect to the metric $d_{\square}$ and hence $G(M^n)$ converges to $\Psi(\phi)$ almost surely. This follows from applying the bound derived during the proof of Proposition~\ref{prop:lln-avg-path}, with $(\beta^{n, \pm}, \beta^{\pm})$ replaced with $(\upsilon^{n, \pm}, \upsilon^{\pm})$.
            Now exactly the same argument as in Section \ref{sec:pflowbd} shows that for any $\phi \in \cW_0$
    \begin{align*}
			\limsup_{n \to \infty} - \dfrac{1}{n^2a(n)} \log \E \exp\left[ -a(n)n^2 G(M^n)\right] 
			& \leq  J^{(\beta^+,\,\beta^-)}(\phi) + \Psi(\phi).		
            \end{align*}
            Here we use the fact that convergence of $\phi^n \to \phi$ in $L^2\left([0,1]^2_T\right)$ implies convergence of $\int_0^T \phi^n_s ds \to \int_0^T \phi_s ds$ in $d_{\square}$.
Taking infimum over all $\phi \in \cW_0$ such that $T^{-1}\int_0^T \phi_s\,ds=f$ for some fixed $f \in \cls_0$ in the above display, we have
	$$ 	\limsup_{n \to \infty} - \dfrac{1}{n^2a(n)} \log \E \exp\left[ -a(n)n^2 G(M^n)\right] \leq  I^{(\beta^+,\;\beta^-)}_M(f) + G(f).$$
    It then follows (cf. \cite[Theorem 1.8]{Budhirajaweakconv}), 
	\begin{equation}{\label{cutldp_lower2}}
		\liminf_{n \to \infty}  \dfrac{1}{n^2a(n)} \log \bP \left( M^n \in B_{\square}\left(f,r\right) \right) \geq  -I^{(\beta^+,\;\beta^-)}_M(f), \;\; \mbox{ for all }  f \in \cls_0, \, r >0,
	\end{equation}	
	where $B_{\square}\left(f,r\right)  := \left\{g \in \cls_0 \; : \; d_{\square}(f,g) <r\right\}$. We  introduce the notation
	$$ \widehat{B}_{\square}\left(\widehat g, r\right) := \left\{ \widehat h \, : \; \delta_{\square} \left( \widehat g, \widehat h\right) < r\right\}, \; \; \mbox{for all } \; \widehat{g} \in \widehat{\cls_0},\; r > 0. $$
	Now fix any $\widehat f \in \widehat{\cls_0}$ and $r \in (0,\infty)$. For any $\eta \in (0,r/2)$ and $\widehat g \in B_{\square}\left(\widehat f,\eta\right)$, we have from \eqref{cutldp_lower2}, that
	\begin{align*}
			\liminf_{n \to \infty}  \dfrac{1}{n^2a(n)} \log \bP \left( \widehat{M^n} \in \widehat B_{\square}\left(\widehat f,r\right) \right) &\geq 	\liminf_{n \to \infty}  \dfrac{1}{n^2a(n)} \log \bP \left( \widehat{M^n} \in \widehat B_{\square}\left(\widehat g,r/2\right) \right) \\
			& \geq 	\liminf_{n \to \infty}  \dfrac{1}{n^2a(n)} \log \bP \left( M^n \in B_{\square}\left(g,r/2\right) \right) \\
            & \geq -I^{(\beta^+,\;\beta^-)}_M(g),
	\end{align*}
	for all $g \in \widehat g$. Therefore,
    \begin{align*}
        \liminf_{n \to \infty}  \dfrac{1}{n^2a(n)} \log \bP \left( \widehat{M^n} \in \widehat B_{\square}\left(\widehat f,r\right) \right) &\geq - \sup_{\eta >0} \;\; \inf_{\widehat g \in B_{\square}\left(\widehat f,\eta\right)} \;\; \inf_{g \in \widehat g}\; I^{(\beta^+,\;\beta^-)}_M(g) \\
        &= - \widehat I^{(\beta^+,\;\beta^-)}_M \left(\widehat f\right).
    \end{align*}
	This concludes the proof of the LDP lower bound in Proposition~\ref{prop:cutLDP-avg-whole}.

	\subsection{Proof of the upper bound}\label{sec-ldpuppcut}
	Theorem~\ref{thm:weakLDP} and contraction principle implies that the sequence $\left\{M^n : n \geq 1\right\}$ satisfies an LDP on the space $\left(\cS_0, d_{\cS_0, \text{weak}}\right)$ with speed $a(n)n^2$ and rate function $I^{(\beta^+,\;\beta^-)}_M$. Our basic approach now is to lift this LDP upper bound for $M^n$ with respect to the weak topology to the LDP upper bound for $\widehat{M^n}$ with respect to the cut-metric topology. This idea was first introduced in~\cite{Chatterjee2011}, and later developed in other contexts in \cite{Dhara2022},\cite{Markering2023} and \cite{Braunsteins2023}. The main obstacle in such a lifting argument arises from the relabeling of the vertices of the graphs with respect to some measure-preserving bijection, since the collection $\sS$ has few nice structural properties to work with. To overcome this, we first show that we can replace the functions $\beta^{n,\pm}$ with functions with an appropriate block structure by incurring only a small cost in probability under the log-scale. Next we show that with jump-rate functions having such nice block structure we can restrict attention to a finite subset of $\sS$ as the set of possible relabeling. This allows us to  apply the weak topology LDP upper bound mentioned before.
	
	\subsubsection{Approximation by block rate functions} 
	For any $\alpha^+, \alpha^- \in \cW$, 	let $\{X^{n,\alpha}_{i,j}(t), t \in [0,T]\}_{i,j \in [n]}$ be a collection of $n^2$ mutually independent Markov chains with state space $\mathcal{I} := \{0,1\}$, $X^{n,\alpha}_{i,j}(0) = x^n_{i,j} \in \left\{0,1\right\}$ for 
	all $i,j \in [n]$, and (time-inhomogeneous) rate matrices $R^{n,\alpha}_{i,j}(\cdot,\cdot\,;t)$ defined as
	\begin{align*}
		R^{n,\alpha}_{i,j}(0,0;t) & := -a(n)\alpha^{+}\left(i/n,j/n,t\right), \;
		R^{n,\alpha}_{i,j}(0,1;t) := a(n)\alpha^{+}\left(i/n,j/n,t\right), \\ 
		R_{i,j}^{n,\alpha}(1,0;t)& := a(n)\alpha^{-}\left(i/n,j/n,t\right), \; 	R_{i,j}^{n,\alpha}(1,1;t) := -a(n)\alpha^{-}\left(i/n,j/n,t\right).
	\end{align*}
	The (directed) random graph at time $t \in [0,T]$, denoted by $G^{n,\alpha}(t)$, is defined on the vertex set $[n]$ and has the edge which connects $i$ to $j$ (for $i,j \in [n]$) if and only if $X_{i,j}^{n,\alpha}(t)=1$. We define 
	$$ M^{n,\alpha} := \dfrac{1}{T}\int_0^T H^{n,\alpha}_s\,ds,$$
	where $H^{n,\alpha}_s$ is the graphon representation of $G^{n,\alpha}(s)$.

	\begin{lem}{\label{lem:fixedgraphon}}
		Suppose $\alpha_1^{\pm},\alpha_2^{\pm} \in \cW$ are such that, for some $\varepsilon >0$,
			\begin{equation}{\label{upper:cond}}
			\dfrac{1}{n^2} \sum_{i,j=1}^n \sup_{s \in [0,T]} \Big \rvert \alpha_1^{\pm}(i/n,j/n,s) - \alpha_2^{\pm}(i/n,j/n,s) \Big \rvert \leq  \varepsilon.
		\end{equation}
		 Fix $\eta >0$ and set $\alpha_3^{\pm} := \alpha_2^{\pm} + \eta$. Then for any $n \geq 1, h \in \cls_0$ and $\delta >0$, we have 
		$$ \dfrac{1}{a(n)n^2} \log  \mathbb{P} \left( M^{n,\alpha_1} \in \bar B_{\square}\left(h,\delta \right) \right) \leq \dfrac{1}{a(n)n^2} \log \mathbb{P} \left( M^{n,\alpha_3} \in \bar B_{\square}\left(h,\delta+4\varepsilon/\eta \right) \right) + 4T\eta.$$
		Here, $\bar B_{\square}(f,r) := \left\{g \in \cls_0\; : \; d_{\square}(f,g) \leq r\right\}$, for  $f \in \cls_0$ and $r>0$.
	\end{lem}

	\begin{proof}
		Let 
		$$ A^{\pm} := \left\{(i,j) \in [n]^2 : \sup_{s \in [0,T]} \Big \rvert \alpha_1^{\pm}(i/n,j/n,s) - \alpha_2^{\pm}(i/n,j/n,s) \Big \rvert > \eta\right\},$$
		and $A:= A^+ \cup A^-$. By \eqref{upper:cond}, we have
		\begin{align*}
			\operatorname{card}(A) \leq 	\operatorname{card}\left(A^+ \right) + 	\operatorname{card}\left(A^-\right) & \leq \dfrac{1}{\eta} \sum_{i,j=1}^n \sup_{s \in [0,T]} \Big \rvert \alpha_1^{+}(i/n,j/n,s) - \alpha_2^{+}(i/n,j/n,s) \Big \rvert \\
			& \hspace{0.5 cm} +\dfrac{1}{\eta} \sum_{i,j=1}^n \sup_{s \in [0,T]} \Big \rvert \alpha_1^{-}(i/n,j/n,s) - \alpha_2^{-}(i/n,j/n,s) \\
            &\leq \dfrac{2\varepsilon n^2}{\eta}.
		\end{align*}
		For $f, g \in \cls_0$ and $ R := \bigcup_{(i,j) \in A} Q^n_{i,j} \subseteq [0,1]^2$, let
		$$ d_{\square,R}(f,g) :=  \sup_{S,T \subseteq [0,1]} \Big \rvert \int_{(S \times T)\cap R^c} (f(x,y)-g(x,y))\, dx\,dy \Big \rvert, $$ 
		 It is easy to see that for any $S, T \subseteq [0,1]$, 
		\begin{align*}
		 \Big \rvert \int_{(S \times T)\cap R^c} (f(x,y)-g(x,y))\, dx\,dy - \int_{S \times T} (f(x,y)-g(x,y))\, dx\,dy\Big \rvert 
        & \leq \int_R |f(x,y)-g(x,y)|\,dx\,dy \\
		 & \leq |R| = \dfrac{\operatorname{card}(A)}{n^2} \leq 2\varepsilon/\eta,
		\end{align*}
		and thus 
		\begin{equation}{\label{es2}}
			\big \rvert d_{\square}(f,g) - d_{\square,R}(f,g) \big \rvert \leq 2\varepsilon/\eta, \;\;\mbox{ for all } f,g \in \cls_0.
		\end{equation} 
		This yields that 
		\begin{equation}{\label{es3}}
			\mathbb{P} \left( M^{n,\alpha_1} \in \bar B_{\square}\left(h,\delta\right)\right) \leq \mathbb{P} \left( M^{n,\alpha_1} \in \bar B_{\square,R}\left(h,\delta+2\varepsilon/\eta\right)\right),
		\end{equation}
		where 
		$$  \bar B_{\square,R}(f,r) := \left\{g \in \cls_0 \; : \; d_{\square,R} \left(f,g\right) \leq r\right\}, \; \forall\; f \in \cls_0, \; r >0.$$
	Similarly, 
		\begin{equation}{\label{es4}}
		\mathbb{P} \left( M^{n,\alpha_3} \in \bar B_{\square,R}\left(h,\delta+2\varepsilon/\eta\right)\right) \leq \mathbb{P} \left( M^{n,\alpha_3} \in \bar B_{\square}\left(h,\delta+4\varepsilon/\eta\right)\right).
	\end{equation}
	
    We will now compare the probability on the right side of \eqref{es3} with that on the left side of \eqref{es4}. We will use the  the representation in \eqref{evolution} with $X^n$ replaced with $X^{n, \alpha_u}$, $u=1,3$ :
			\begin{align}{\label{pois5}}
			X_{i,j}^{n,\alpha_u}(t) &= x_{i,j}^n + \int_{[0,t]}  (1-X_{i,j}^{n,\alpha_u}(s-))\,  \cln_{i,j}^{n,u,+}(ds)  -\int_{[0,t]} X_{i,j}^{n,\alpha_u}(s-) \, \cln_{i,j}^{n,u,-}(ds), 
		\end{align}
	for all $t \geq 0$, where $\left\{\cln_{i,j}^{n, u, \pm} : (i,j) \in [n]^2\right\}$ is a collection of independent Poisson processes on $[0,T]$ with intensity at time $s$  given by $\lambda_{i,j}^{n,u,\pm}(s)$, defined as follows :  
	$$  \lambda_{i,j}^{n,u,\pm}(s) = a(n)\alpha_u^{\pm}\left(i/n,j/n,s\right).$$
    Note that $\cln_{i,j}^{n,u,\pm}$ are random variables with values 
    in the space $\cld_0$ of c\`adl\`ag non-decreasing piecewise constant functions, starting from $0$, jump size $1$, and equipped with the usual Skorokhod topology. Denote the corresponding probability laws as $\Gamma_{i,j}^{n,u,\pm}$. Then by Girsanov's theorem for point processes (see e.g. \cite[Theorem 8.15]{Budhirajaweakconv}), for
    $\Gamma_{i,j}^{n, 3, \pm}$ a.s. $\vn$,
	\begin{equation}{\label{beforees5}}
		\frac{d \Gamma_{i,j}^{n, 1, \pm} }{d \Gamma_{i,j}^{n, 3, \pm} }( \vn ) = \exp \left[ \int_0^T \log \left( \frac{\lambda_{i,j}^{n,1,\pm}(s)}{\lambda_{i,j}^{n,3,\pm}(s)}\right)\, 
		\vn(ds) - \int_0^T \left( \lambda_{i,j}^{n,1,\pm}(s) - \lambda_{i,j}^{n,3,\pm}(s)\right)\,ds\right].
	\end{equation}
    Note that, for any $(i,j) \in A^c$, 
	\begin{equation}
    \lambda_{i,j}^{n,3,\pm}(s) -2 a(n)\eta \leq \lambda_{i,j}^{n,1,\pm}(s)  \leq \lambda_{i,j}^{n,3,\pm}(s),
    \end{equation}
	and therefore for such $(i,j)$ the first term in the exponent of the RHS in~\eqref{beforees5} is non-positive. Consequently, for $(i,j) \in A^c$, 
		\begin{equation}
		\frac{d \Gamma_{i,j}^{n, 1, \pm} }{d \Gamma_{i,j}^{n, 3, \pm} }( \vn ) \leq \exp\left( 2Ta(n)\eta\right), \; \Gamma_{i,j}^{n, 3, \pm} \mbox{ a.s. }
	\end{equation}
	Note that $d_{\square,R}(f,g)$ only depends on $f$ and $g$ through their restrictions on $R^c$. Therefore, the event  $\left( M^{n,\alpha_u} \in \bar B_{\square,R}\left(h,\delta+2\varepsilon/\eta\right)\right)$ is a measurable function of $\left\{\cln_{i,j}^{n, u, \pm} : (i,j) \in A^c\right\}$. Since, 
	\begin{align*}
        \prod_{(i,j) \in A^c} \frac{d \Gamma_{i,j}^{n, 1, +} }
        {d \Gamma_{i,j}^{n, 3, +}} \prod_{(i,j) \in A^c} \frac{d \Gamma_{i,j}^{n, 1, -}}{d \Gamma_{i,j}^{n, 3, -} } \leq \exp \left( 4T\operatorname{card}(A^c)a(n)\eta\right)
		 \leq \exp \left( 4Tn^2a(n)\eta\right),
	\end{align*} 
	we  conclude that 
	\begin{equation}{\label{es5}}
		 \mathbb{P} \left( M^{n,\alpha_1} \in \bar B_{\square,R}\left(h,\delta+2\varepsilon/\eta \right) \right) \leq  \exp \left( 4Tn^2a(n)\eta\right) \mathbb{P} \left( M^{n,\alpha_3} \in \bar B_{\square,R}\left(h,\delta+2\varepsilon/\eta \right) \right).
	\end{equation}
	Combining \eqref{es3}, \eqref{es4} and \eqref{es5} completes the proof.
\end{proof}
	
	\begin{lem}{\label{lem:approxblock}}
        For any $\varepsilon >0$ we can find $N_1=N_1(\varepsilon) \in \mathbb{N}$ such that for any $k \geq N_1$, there exists $N_2=N_2(k,\varepsilon) \in \mathbb{N}$ satisfying the following :
		 $$ \dfrac{1}{a(n)n^2} \log  \mathbb{P} \left( M^n \in \bar B_{\square}\left(h,\delta \right) \right) \leq \dfrac{1}{a(n)n^2} \log \mathbb{P} \left( M^{n,\beta^k + \sqrt{\varepsilon}} \in \bar B_{\square}\left(h,\delta+4\sqrt{\varepsilon} \right) \right) + 4T\sqrt{\varepsilon},$$
		 for any $n \geq N_2, h \in \cls_0, \delta >0$. 
	\end{lem}

	\begin{proof}
		By Assumption \ref{assmp:1}\eqref{assump:1D} and Assumption \ref{assmp:2}\eqref{assump:1AB},
        we can get $N_1 = N_1(\varepsilon) \in \mathbb{N}$ such that for all $k \geq N_1$, we have 
		\begin{equation}{\label{es6}}
			\int_{[0,1]^2} \sup_{s \in [0,T]} \big \rvert \beta_s^{k,\pm}(x,y) - \beta_s^{\pm}(x,y) \big \rvert\,dx\,dy \leq \varepsilon/4, \;\; \sup_{i,j \in [k]} \sup_{s \in [0,T]} \beta^{k,\pm}_{i,j}(s) < \infty, 
		\end{equation}
		and thus for any $n\geq k \geq N_1$, 
			\begin{equation}{\label{es7}}
			\int_{[0,1]^2} \sup_{s \in [0,T]} \big \rvert \beta_s^{n,\pm}(x,y) - \beta_s^{k,\pm}(x,y) \big \rvert\,dx\,dy \leq \varepsilon/2.
		\end{equation}
		Let 
		$$A := \left\{ i \in [n] \; : \; Q_i^n \subseteq Q_u^k \text{ for some } u \in [k]  \right\}.$$
		In other words, 
		$$i \notin A \Rightarrow \exists \; u \in [k] \text{ such that } \dfrac{i-1}{n} < \dfrac{u}{k} < \dfrac{i}{n}.$$
		  It is easy to see that for any $u \in [k]$, there can be at most one $i \in [n]$ satisfying the afore-mentioned inequalities and thus $\operatorname{card}(A^c)\leq k$. Take $i,j \in A$ and get $u,v \in [k]$ such that $Q_i^n \subseteq Q_u^k, Q_j^n \subseteq Q_v^k$. For any $(x,y) \in Q_{i,j}^n \subseteq Q_{u,v}^k$, we have 
		  $$ \beta_s^{n,\pm}(x,y) = \beta_{i,j}^{n,\pm}(s) = \beta^{n,\pm}\left(i/n,j/n,s\right), \; \; \beta_s^{k,\pm}(x,y) = \beta_{u,v}^{k,\pm}(s)= \beta^{k,\pm}\left(i/n,j/n,s\right),$$
		  where the last equality follows from the observation that $(i/n,j/n) \in Q_{u,v}^k$. On the other hand, for any $i,j \in [n]$ and $(x,y) \in Q_{i,j}^n$, we have 
		  \begin{align*}
		  &\Bigg \rvert 	\big \rvert \beta_s^{n,\pm}(x,y) - \beta_s^{k,\pm}(x,y) \big \rvert - \big \rvert \beta^{n,\pm}\left(i/n,j/n,s\right) - \beta^{k,\pm}\left(i/n,j/n,s\right) \big \rvert \Bigg \rvert \\
		  & \hspace{2 in}  \leq \big \rvert  \beta_s^{k,\pm}(x,y) - \beta^{k,\pm}\left(i/n,j/n,s\right) \big \rvert \leq \sup_{u,v \in [k]} \beta^{k,\pm}_{u,v}(s).
		  \end{align*}
		  Therefore,
		  \begin{align*}
		  	\varepsilon/2 &\geq 	\int_{[0,1]^2} \sup_{s \in [0,T]} \big \rvert \beta_s^{n,\pm}(x,y) - \beta_s^{k,\pm}(x,y) \big \rvert\,dx\,dy \\
		  	&= \sum_{i,j \in A} 	\int_{Q_{i,j}^n} \sup_{s \in [0,T]} \big \rvert \beta_s^{n,\pm}(x,y) - \beta_s^{k,\pm}(x,y) \big \rvert\,dx\,dy \\
		  	& \hspace{0.1 in}+ \sum_{\substack{i \in A^c \\ \text{ or } j \in A^c}} 	\int_{Q_{i,j}^n} \sup_{s \in [0,T]} \big \rvert \beta_s^{n,\pm}(x,y) - \beta_s^{k,\pm}(x,y) \big \rvert\,dx\,dy \\
		  	& \geq \dfrac{1}{n^2}\sum_{i,j \in A}  \sup_{s \in [0,T]} \big \rvert \beta^{n,\pm}(i/n,j/n,s) - \beta^{k,\pm}(i/n,j/n,s) \big \rvert \\
		  	& \hspace{0.1 in} + \dfrac{1}{n^2}\sum_{\substack{i \in A^c \\ \text{ or } j \in A^c}}  \sup_{s \in [0,T]} \big \rvert \beta^{n,\pm}(i/n,j/n,s) - \beta^{k,\pm}(i/n,j/n,s) \big \rvert -  \dfrac{1}{n^2}\sum_{\substack{i \in A^c \\ \text{ or } j \in A^c}}  \sup_{\substack{u,v \in [k]\\s \in [0,T]}} \beta^{k,\pm}_{u,v}(s) \\
		  	& \geq \dfrac{1}{n^2}\sum_{i,j \in [n]}  \sup_{s \in [0,T]} \big \rvert \beta^{n,\pm}(i/n,j/n,s) - \beta^{k,\pm}(i/n,j/n,s) \big \rvert - \dfrac{2n\operatorname{card}(A^c)}{n^2} \sup_{\substack{u,v \in [k]\\s \in [0,T]}} \beta^{k,\pm}_{u,v}(s) \\
		  	& \geq \dfrac{1}{n^2}\sum_{i,j \in [n]}  \sup_{s \in [0,T]} \big \rvert \beta^{n,\pm}(i/n,j/n,s) - \beta^{k,\pm}(i/n,j/n,s) \big \rvert - \dfrac{2k}{n} \sup_{\substack{u,v \in [k]\\s \in [0,T]}} \beta^{k,\pm}_{u,v}(s).
		  \end{align*}
		   Using \eqref{es6} we can find $N_2 = N_2(k,\varepsilon) \in \mathbb{N}$, such that 
		  $$ \dfrac{2k}{N_2} \sup_{\substack{u,v \in [k]\\s \in [0,T]}} \beta^{k,\pm}_{u,v}(s) \leq \varepsilon/2.$$
		  Then for any $k \geq N_1(\varepsilon)$ and for any $n \geq N_2(k,\varepsilon)$, we have 
		  $$ \dfrac{1}{n^2}\sum_{i,j \in [n]}  \sup_{s \in [0,T]} \big \rvert \beta^{n,\pm}(i/n,j/n,s) - \beta^{k,\pm}(i/n,j/n,s) \big \rvert \leq \varepsilon.$$
		  Applying Lemma~\ref{lem:fixedgraphon} with $\eta= \sqrt{\varepsilon}$, $\alpha_1 = \beta^n$, and $\alpha_2 = \beta^k$,  and observing that $M^{n} \stackrel{d}{=} M^{n,\beta^n}$, we conclude: For any $h \in \cls_0$, $\delta>0$, and $n \ge N_2(k, \varepsilon)$,
		  $$ \dfrac{1}{a(n)n^2} \log  \mathbb{P} \left( M^n \in \bar B_{\square}\left(h,\delta \right) \right) \leq \dfrac{1}{a(n)n^2} \log \mathbb{P} \left( M^{n,\beta^k + \sqrt{\varepsilon}} \in \bar B_{\square}\left(h,\delta+4\sqrt{\varepsilon} \right) \right) + 4T\sqrt{\varepsilon}.$$
          The result follows.
	\end{proof}

\subsubsection{Approximation by block functions}
	
	We start with a formal definition of block rate functions and block graphons. Fix $k\geq 1$. Let $\cW^{(k)}$ denote the set of all $\phi \in \cW$ which can be expressed in the following form :
	$$ \phi(x,y,s) = \sum_{u,v\in[k]} \phi_{u,v}(s)\, \mathbbm{1}_{Q_{u,v}^k}(x,y), \; \forall \, (x,y,s) \in [0,1]^2 \times [0,T],$$
	where $\phi_{u,v} : [0,T]\to [0,\infty)$ are measurable functions for all $u,v \in [k]$. Similarly,  $\cls^{(k)}$ will denote the set of all $f \in \cls$ which can be expressed in the following form :
	$$ f(x,y) = \sum_{u,v\in[k]} f_{u,v}\, \mathbbm{1}_{Q_{u,v}^k}(x,y), \; \forall \, (x,y) \in [0,1]^2,$$
	where $f_{u,v} \in [0,\infty)$ for all $u,v \in [k]$. Set $\cW_0^{(k)} := \cW_0 \cap \cW^{(k)}$ and $\cls_0^{(k)} := \cls_0 \cap \cls^{(k)}$. Let $\sS_k$ be the set of all permutations on $[k]$. For $\sigma \in \sS_k$, let $\tilde \sigma \in \sS$ be its canonical representation as a measure preserving bijection. In other words, 
	$$ \tilde \sigma (0) =0, \; \; \tilde\sigma \left( \frac{u}{k}\right) = \frac{\sigma(u)}{k}, \;  u \in [k],$$
	with $\tilde \sigma$ being linear on each interval $((u-1)/k,u/k]$ with slope $1$. The following lemma shows that there exists a finite subset $\sT \subseteq \sS$ such that for all $n$ large enough the distribution of $\left(M^n\right)^{\tilde \sigma}$, for any $\sigma \in \sS_n$ can be approximated by $\left(M^n\right)^{\tau}$ for some $\tau \in \sT$.
	
	\begin{lem}{\label{lem:approxmeasure}}
		Fix $\alpha^+, \alpha^- \in \cW^{(k)}$ and $f \in \cls_0^{(k)}$ for some $k \geq 1$. Then for any $\varepsilon >0$, there exists $N_3 = N_3(k,\varepsilon) \in \NN$ and a finite subset $\sT=\sT(k,\varepsilon) \subseteq \sS$ such that for any $n \geq N_3$ and $\sigma_n \in \sS_n$, there exists $\tau \in \sT$ satisfying
		$$ \bP \left[ \left(M^{n,\alpha}\right)^{\tilde \sigma_n} \in \bar B_{\square}\left(f, \varepsilon \right) \right] \leq \bP \left[ \left(M^{n,\alpha}\right)^{\tau} \in \bar B_{\square}\left(f, 2\varepsilon \right) \right]. $$
	\end{lem}
	
	The proof of Lemma~\ref{lem:approxmeasure} is a straightforward adaption of the proof of~\cite[Lemma 3.3]{Dhara2022} and is therefore omitted. 
     	Fix $K \in \mathbb{N}$. Let $\Lambda_{K}$ be the subset $\cls_0$ satisfying the following conditions : We have $f \in \Lambda_K$ if and only if there exists $\varepsilon \in [0,1/2]$ and $\left\{p_{i,j} : 0 \leq i,j \leq K\right\} \subseteq [0,1]$ such that
	\begin{equation}{\label{totalform}}
		f = \sum_{i,j=0}^K p_{i,j} \mathbbm{1}_{L_i \times L_j},
	\end{equation} 
	where $L_0 := [0,\varepsilon]$ and $L_i = (\varepsilon+(i-1)(1-\varepsilon)/K, \varepsilon+i(1-\varepsilon)/K]$, for all $1 \leq i \leq K$.   
	For any $1 \leq m \leq M < \infty$, we set 
	$$ \Lambda_{m,M} := \bigcup_{m \leq K \leq M} \Lambda_{K}.$$
	The proof of the following standard result is omitted.
	\begin{lem}{\label{totalbound}}
		For any $K \in \mathbb{N}$, the set $\Lambda_{K}$ is totally bounded in $L^2\left([0,1]^2\right)$ (with respect to the $L^2$-norm).
	\end{lem}
    
		The following lemma is an analogue of~\cite[Theorem 3.1]{Chatterjee2011}, which considers $\{0,1\}$ valued graphons, for the setting of graphons with values in $[0,1]$.
		\begin{lem}{\label{sze2}}
			Fix $\varepsilon \in (0,1)$ and $m \geq 1$. Then there exists $M=M(\varepsilon,m) \in \mathbb{N}$ such that for any $f \in \cls_0^{(n)}$ with $n \geq M$, there exists a permutation $\pi \in \sS_n$ such that 
			$$ \inf_{g \in \Lambda_{m,M}} d_{\square} \left( f^{\tilde \pi} , g\right) \leq \varepsilon.$$ 
		\end{lem}
		
		\begin{proof}
			The statement of the above lemma is true for $\left\{0,1\right\}$-valued graphons by~\cite[Theorem 3.1]{Chatterjee2011}. In particular, there exists $M_1=M_1(\varepsilon,m) \in \mathbb{N}$ such that for any $h \in \cls_0^{(n)}$ with $n \geq M_1$ and $h$ taking values in $\left\{0,1\right\}$ only, there exists a permutation $\pi \in \sS_n$ such that 
			$$ \inf_{g \in \Lambda_{m,M_1}} d_{\square} \left( h^{\tilde \pi} , g\right) \leq \varepsilon/2.$$
			Indeed, \cite[Theorem 3.1]{Chatterjee2011} only proves this result for the case when $h$ is symmetric, but the general case is a straightforward adaptation. Let $M=M(\varepsilon,m) \in \NN$ be such that $M \geq M_1$ and $2^{2n}\exp(-n^2\varepsilon^2/4) <1$ for all $n \geq M$.
			
			Take any $f \in \cls_0^{(n)}$ where $n \geq M$. It can be expressed as 
			$$ f = \sum_{i,j \in [n]} f_{i,j} \mathbbm{1}_{Q_{i,j}^n},$$
			for some $f_{i,j} \in [0,1]$. Generate $X_{i,j} \stackrel{\text{ind}}{\sim} Ber(f_{i,j})$ and set 
			$$ H := \sum_{i,j \in [n]} X_{i,j} \mathbbm{1}_{Q_{i,j}^n}.$$ By Lemma~\ref{es11}, we can conclude that 
			$$  \mathbb{P} \left( d_{\square}\left(H,f\right) > \varepsilon/2 \right) \leq \mathbb{P} \left( d_{\infty \to 1}\left(H,f\right) > \varepsilon/2 \right) \leq 2^{2n}\exp\left(-n^2\varepsilon^2/4\right) < 1.$$
			Hence, there exists $h \in \cls_0^{(n)}$ taking values in $\left\{0,1\right\}$ such that $d_{\square}\left(f,h\right) \leq \varepsilon/2$. Since, $n \geq M \geq M_1$, we can get a permutation $\pi \in \sS_n$ such that 
			$$ \inf_{g \in \Lambda_{m,M_1}} d_{\square} \left( h^{\tilde \pi} , g\right) \leq \varepsilon/2.$$ Therefore,
			\begin{align*}
				\inf_{g \in \Lambda_{m,M}} d_{\square} \left( f^{\tilde \pi} , g\right) \leq 	\inf_{g \in \Lambda_{m,M_1}} d_{\square} \left( f^{\tilde \pi} , g\right) & \leq 	d_{\square}  \left( f^{\tilde \pi},h^{\tilde \pi}  \right) +  \inf_{g \in \Lambda_{m,M_1}} d_{\square} \left( h^{\tilde \pi} , g\right) \\
                & \leq d_{\square}\left(f,h\right) + \varepsilon/2 \leq \varepsilon.
			\end{align*}
			 	\end{proof}
			
			\begin{lem}{\label{lem:approx}}
				For any $f \in \cls_0$ and $\varepsilon >0$, there exists $k=k(f,\varepsilon)\in \NN$ such that $d_{\square}\left(\cls_0^{(n)},f \right) \leq \varepsilon$, for all $n \geq k$. Here, $d_{\square}\left(\cT,f\right) := \inf_{g \in \cT} d_{\square}(f,g)$
				for any $\cT \subseteq \cls_0$.
			\end{lem}
			\begin{proof}
			For  $n \geq 1$, set 
			$$ f_n := \sum_{i,j \in [n]} f_{i,j}^n \mathbbm{1}_{Q_{i,j}^n} \in \cls_0^{(n)}, \; \; f_{i,j}^n := n^2 \int_{Q_{i,j}^n} f,  \; \mbox{ for } i,j \in [n].$$
			Then $f_n$ converges to $f$ almost everywhere on $[0,1]^2$ and hence in $L^2$. This implies that $d_{\square}(f_n,f) \longrightarrow 0$ as $n \to \infty$. Since, $d_{\square}\left(\cls_0^{(n)},f \right) \leq d_{\square}(f_n,f) $, this completes the proof.
			\end{proof}

	\subsubsection{Completing the proof of the upper bound in Proposition~\ref{prop:cutLDP-avg-whole}}
    Since $\widehat{\cls_0}$ is compact, by an elementary covering argument, it suffices to show that
	\begin{equation}{\label{ldp-upcut:ts}}
    \lim_{\eta \longrightarrow 0} \limsup_{n \to \infty} \dfrac{1}{n^2a(n)} \log \mathbb{P} \left(\widehat{M^n} \in \widehat{B}_{\square}\left(\widehat{h}, \eta \right) \right) \leq - \widehat I^{(\beta^+,\;\beta^-)}_M\left(\widehat{h} \right),
    \end{equation}
	for any $h \in \cls_0$. Let, for $g \in \cls_0$, $B^*(g, \eta) := \left\{f \in \cls_0 \; : \; \widehat{f} \in  \widehat{B}_{\square}\left(\widehat{g},\varepsilon\right)\right\}.$ It is enough to show that 
	$$ \lim_{\eta \longrightarrow 0} \limsup_{n \to \infty} \dfrac{1}{n^2a(n)} \log \mathbb{P} \left(M^n \in B^*(h,\eta) \right) \leq - \widehat I^{(\beta^+,\;\beta^-)}_M\left(\widehat{h} \right).$$
	Fix $\varepsilon>0$. From Lemma~\ref{sze2} we can find $R=R(\varepsilon) \in \mathbb{N}$ such that for $n \geq R$,  
	$$ \inf_{\pi \in \sS_n} \inf_{g \in \Lambda_{1,R}} d_{\square}\left(\left(M^n\right)^{\tilde \pi},g\right) < \varepsilon.$$ 
	Since the set $B^*(h,\eta)$ is invariant under the action of measure-preserving transformations, we have
	\begin{align*}
		\mathbb{P} \left( M^n \in B^*(h,\eta)\right) &\leq \sum_{\pi \in \mathscr{S}_n} \mathbb{P} \left( M^n \in B^*(h,\eta), \left(M^n\right)^{\tilde \pi}  \in B_{\square}(\Lambda_{1,R},\varepsilon)\right) \\
		& = \sum_{\pi \in \mathscr{S}_n} \mathbb{P} \left( \left(M^n\right)^{\tilde \pi} \in B^*(h,\eta) \cap B_{\square}(\Lambda_{1,R},\varepsilon)\right) \\
		& \leq n! \sum_{K=1}^R \sup_{\pi \in \mathscr{S}_n} \mathbb{P} \left(\left(M^n\right)^{\tilde \pi}\in B^*(h,\eta) \cap B_{\square}(\Lambda_{K},\varepsilon)\right),
	\end{align*}  
	where $B_{\square}\left(\cT,r\right) := \left\{f \in \cls_0 : d_{\square}(\cT,f) < r\right\}$ for any $\cT \subseteq \cls_0$.
	Since, $\log n! /n^2a(n) \to 0$ as $n \to \infty$, it is enough to show that for any $K \geq 1$, 
	$$ \lim_{\eta \longrightarrow 0} \limsup_{n \to \infty} \dfrac{1}{n^2a(n)} \log \sup_{\pi \in \mathscr{S}_n}\mathbb{P} \left(\left(M^n\right)^{\tilde \pi} \in B^*(h,\eta) \cap B_{\square}(\Lambda_{K},\varepsilon)\right) \leq - \widehat I^{(\beta^+,\;\beta^-)}_M\left(\widehat{h} \right).$$
	From Lemma~\ref{totalbound}, $\Lambda_K$ is totally bounded in $L^2$-metric, and consequently the same holds for the cut metric. Thus it suffices to show that for any $g \in \cls_0$, 
	\begin{equation}{\label{eq2}}
		\lim_{\eta \longrightarrow 0} \limsup_{n \to \infty} \dfrac{1}{n^2a(n)} \log \sup_{\pi \in \mathscr{S}_n}\mathbb{P} \left(\left(M^n\right)^{\tilde \pi} \in B^*(h,\eta) \cap \bar{B}_{\square}(g,2\varepsilon)\right) \leq - \widehat I^{(\beta^+,\;\beta^-)}_M\left(\widehat{h} \right).
	\end{equation}
	With $N_1(\cdot)$ as in Lemma~\ref{lem:approxblock}, fix any $k \geq N_1(\varepsilon^2)$. Note that, from this lemma, we can find a $N_2 = N_2(k, \varepsilon^2)$ such that the inequality in the statement of the lemma holds, with $\varepsilon$ replaced by $\varepsilon^2$, for all $n \ge N_2$, $h \in \cls_0$, and $\delta>0$. Next, to prove  \eqref{eq2},  it suffices, in view of Lemma~\ref{lem:approx},  to show that 
	for any $k \geq 1$ and $g \in \cls_0^{(k)}$, 
	\begin{equation}{\label{eq2new}}
		\lim_{\eta \longrightarrow 0} \limsup_{n \to \infty} \dfrac{1}{n^2a(n)} \log \sup_{\pi \in \mathscr{S}_n}\mathbb{P} \left(\left(M^n\right)^{\tilde \pi} \in B^*(h,\eta) \cap \bar{B}_{\square}(g,3\varepsilon)\right) \leq - \widehat I^{(\beta^+,\;\beta^-)}_M\left(\widehat{h} \right).
	\end{equation} 
	Without loss of generality we can assume that $B^*(h,\eta) \cap \bar{B}(g,3\varepsilon) \neq \emptyset$. In that case, there exists $f \in \cls_0$ with $d_{\square}(f,g) \leq 3 \varepsilon$ and $\delta_{\square}\left(\widehat{f},\widehat{h}\right) < \eta$ and hence $\delta_{\square} \left(\widehat{g},\widehat{h}\right)< 3\varepsilon + \eta$. Assume, without loss of generality, that $\eta < \varepsilon$.  Then $g \in B^*(h,4\varepsilon)$ and consequently we can find $h_1 \in \widehat{h}$ (depending on $h,g,\varepsilon$) such that $d_{\square}(g,h_1) < 4 \varepsilon$ and hence \begin{equation}\label{eq:906}\bar{B}_{\square}(g,14\varepsilon) \subseteq B_{\square}(h_1,18 \varepsilon) \subseteq B^*(h,18\varepsilon).\end{equation}

	Now, apply Lemma~\ref{lem:approxmeasure} with $\alpha = \beta^k +\varepsilon$, $f$ replaced by $g$, and $\varepsilon$ replaced by $7\varepsilon$. Then, we can obtain a finite $\sT=\sT(k,7\varepsilon) \subseteq \sS$ such that the statement in the lemma holds with the above replacements. Thus, for any $\eta \in (0,\varepsilon)$, we have 
	\begin{align*}
	&	\limsup_{n \to \infty} \dfrac{1}{n^2a(n)} \log \sup_{\pi \in \mathscr{S}_n}\mathbb{P} \left(\left(M^n\right)^{\tilde \pi}  \in B^*(h,\eta) \cap \bar{B}_{\square}(g,3\varepsilon)\right) \\
		& \hspace{1 in} \leq \limsup_{n \to \infty} \dfrac{1}{n^2a(n)} \log \sup_{\pi \in \mathscr{S}_n}\mathbb{P} \left(\left(M^n\right)^{\tilde \pi}  \in  \bar{B}_{\square}(g,3\varepsilon)\right) \\
		& \hspace{1 in} \leq \limsup_{n \to \infty} \dfrac{1}{n^2a(n)} \log \sup_{\pi \in \mathscr{S}_n}\mathbb{P} \left(\left(M^{n,\beta^k+\varepsilon}\right)^{\tilde \pi}  \in  \bar{B}_{\square}(g,7\varepsilon)\right) + 4T\varepsilon \\
		& \hspace{1 in} \leq \limsup_{n \to \infty} \dfrac{1}{n^2a(n)} \log \sup_{\tau \in \sT}\mathbb{P} \left(\left(M^{n,\beta^k+\varepsilon}\right)^{\tau}  \in  \bar{B}_{\square}(g,14\varepsilon)\right) + 4T\varepsilon,
	\end{align*}
	where the second and third inequality are obtained by applying Lemma~\ref{lem:approxblock} and Lemma~\ref{lem:approxmeasure} respectively, as described above.
    
    It is easy to see that the set $\bar{B}_{\square}(g,14\varepsilon)$ is also closed with respect to the weak topology on $\cls_0$. Moreover, for every $\sigma \in \sS$, the map $\cls_0 \ni f \mapsto f^{\sigma} \in \cls_0$ is continuous (with respect to the weak topology). Therefore, applying the LDP result for the sequence $M^n$ with $\beta$ replaced by $\beta^k +\varepsilon$, we have for any $\tau \in \sT$, 
	$$ \limsup_{n \to \infty} \dfrac{1}{n^2a(n)} \log \mathbb{P} \left(\left(M^{n,\beta^k+\varepsilon}\right)^{\tau}  \in  \bar{B}_{\square}(g,14\varepsilon)\right) \leq -\inf_{f : f^{\tau} \in \bar{B}_{\square}(g,14\varepsilon)} I_M^{(\beta^{k,+}+\varepsilon,\;\beta^{k,-}+\varepsilon)}(f).$$ 
	Here we need to establish that $\beta^k+\varepsilon$ satisfies
    Assumption \ref{assmp:1} and Assumption~\ref{assmp:2}\eqref{assump:1A}. However, that is a routine argument along the lines of the proof of Lemma~\ref{lem:approxblock} and is therefore omitted. Since, $\sT$ is finite, apply Lemma~\ref{lem:uniform}   to conclude that for any $\eta \in (0,\varepsilon)$ and $k \geq N_1(\varepsilon^2)$, 
	\begin{align*}
		&	\limsup_{n \to \infty} \dfrac{1}{n^2a(n)} \log \sup_{\pi \in \mathscr{S}_n}\mathbb{P} \left(\left(M^n\right)^{\tilde \pi}  \in B^*(h,\eta) \cap \bar{B}_{\square}(g,3\varepsilon)\right) \\
		& \hspace{1 in} \leq -\inf_{\tau \in \sT} \;\;  \inf_{f : f^{\tau} \in \bar{B}_{\square}(g,14\varepsilon)} I_M^{(\beta^{k,+}+\varepsilon,\;\beta^{k,-}+\varepsilon)}(f) + 4T \varepsilon \\
		& \hspace{1 in} \leq -\inf_{\sigma \in \sS} \;\;  \inf_{f : f^{\sigma} \in \bar{B}_{\square}(g,14\varepsilon)} I_M^{(\beta^{k,+}+\varepsilon,\;\beta^{k,-}+\varepsilon)}(f) + 4T \varepsilon \\
		& \hspace{1 in} \leq -\inf_{\sigma \in \sS} \;\; \inf_{f : f^{\sigma} \in \bar{B}_{\square}(g,14\varepsilon)} I_M^{(\beta^+,\;\beta^-)}(f) + \delta(k,\varepsilon)+4T \varepsilon \\
		& \hspace{1 in} \leq -\inf_{\sigma \in \sS}\;\;  \inf_{f : f^{\sigma} \in B^*(h,18\varepsilon)} I_M^{(\beta^+,\;\beta^-)}(f) + \delta(k,\varepsilon)+4T \varepsilon \\
			& \hspace{1 in} = - \inf_{\widehat f \;: \; \widehat f \in \widehat B_{\square} \left(\widehat h,18\varepsilon\right)} \;\; \inf_{f \in \widehat f}\;I_M^{(\beta^+,\;\beta^-)}(f) + \delta(k,\varepsilon)+4T \varepsilon,
	\end{align*}
	where 
	\begin{align*}
		\delta(k,\varepsilon) &:= \| \beta^{k,+}+\varepsilon - \beta^+\|_1 + \| \beta^{k,-}+\varepsilon - \beta^-\|_1 \\
        & \hspace{1 in} + 2 \|\sqrt{(\beta^{k,+}+\varepsilon)(\beta^{k,-}+\varepsilon)} - \sqrt{\beta^+\beta^-}\|_1.
	\end{align*}
	By Assumption \ref{assmp:1}\eqref{assump:1Aweak}, we have $\beta^{k,\pm} \stackrel{L^1}{\longrightarrow} \beta^{\pm}$ as $k \to \infty$ and hence 
	\begin{align*}
		\lim_{k \to \infty} \delta(k,\varepsilon)& =   2T\varepsilon +  2\|\sqrt{(\beta^++\varepsilon)(\beta^{-}+\varepsilon)} - \sqrt{\beta^+\beta^-}\|_1  \leq 2T\varepsilon + \dfrac{2\varepsilon}{c_{\beta}} \| \beta^++\beta^-\|_1   + \dfrac{2T\varepsilon^2}{c_{\beta}}.
	\end{align*}
    The last expression converges to $0$ as $\varepsilon \downarrow 0$.
	Therefore, 
	\begin{align*}
		 &\lim_{\eta \to 0} \limsup_{n \to \infty} \dfrac{1}{n^2a(n)} \log \sup_{\pi \in \mathscr{S}_n}\mathbb{P} \left(\left(M^n\right)^{\tilde \pi}  \in B^*(h,\eta) \cap \bar{B}_{\square}(g,3\varepsilon)\right) \\
         & \hspace{2 in} \leq - \lim_{\varepsilon \downarrow 0} \inf_{\widehat f \in \widehat B_{\square} \left(\widehat h,18\varepsilon\right)} \;\; \inf_{f \in \widehat f}\;I_M^{(\beta^+,\;\beta^-)}(f) = - \widehat I_M^{(\beta^+,\;\beta^-)}\left( \widehat h \right).
		 \end{align*}
		 This completes the proof of \eqref{eq2new} and therefore of the upper bound. This completes the proof of Proposition~\ref{prop:cutLDP-avg-whole}.

		 \subsection{Proof of Theorem~\ref{thm:LDP-avg-path}}{\label{sec:pf-LDP-path-cut}} In this section, we shall use the technique of projective limit LDP and complete the proof of Theorem~\ref{thm:LDP-avg-path} building upon Proposition~\ref{prop:LDP-avg-path-finite}. Our main tool will be the the results given in~\cite[Theorem 4.28, 4.30]{feng2006large}. The first order of business is to prove exponential tightness of the sequence $\left\{\widehat{A^{n,\epsilon}_{\bcdot}} : n \geq 1\right\}$ in the space $\mathbb{C} \left([0,T] : \left( \widehat{\cS_0}, \delta_{\square}\right)\right)$. By Theorem~\ref{thm:ldp-to-exp-tight}, Theorem~\ref{thm:exp-tight-D-space} and Proposition~\ref{prop:LDP-avg-path-finite}, the exponential tightness of 
		 $\left\{\widehat{A^{n,\epsilon}_{\bcdot}} : n \geq 1\right\}$ in $\mathbb{C} \left([0,T] : \left( \widehat{\cS_0}, \delta_{\square}\right)\right)$ with speed $a(n)n^2$ is equivalent to the following : For any $\eta >0$, 
		 \begin{equation}{\label{eq:exp-to-show}}
		 	 \lim_{\delta \downarrow 0} \limsup_{n \to \infty} \dfrac{1}{a(n)n^2}\log \mathbb{P} \left[ \text{Osc}\left(\widehat{A^{n,\epsilon}_{\bcdot}}, \delta;\left( \widehat{\cS_0}, \delta_{\square}\right)\right) > \eta\right] = - \infty.
		 \end{equation}
		 From~(\ref{eq:av-lip}), we see that 
		 $$ \text{Osc}\left(\widehat{A^{n,\epsilon}_{\bcdot}}, \delta;\left( \widehat{\cS_0}, \delta_{\square}\right)\right) \leq 2\delta \left(\dfrac{1}{\epsilon} + \dfrac{T}{\epsilon^2}\right)=\delta L(T,\epsilon), \; \text{a.s.}, \; \forall \; \delta >0,$$
		 and, therefore, (\ref{eq:exp-to-show}) follows immediately. We can now apply \cite[Theorem 4.30]{feng2006large} to conclude that the sequence $\left\{\widehat{A^{n,\epsilon}_{\bcdot}} : n \geq 1\right\}$ satisfies an LDP in the space $\mathbb{C} \left([0,T] : \left( \widehat{\cS_0}, \delta_{\square}\right)\right)$ with speed $a(n)n^2$ and rate function given by 
		 $$ \widehat{\cI}^{(\beta^+,\,\beta^-)}_{\epsilon} \left(\widehat{\phi_{\bcdot}}\right) =  \sup_{\cT}  	\widehat I^{(\beta^+,\,\beta^-)}_{\epsilon,\cT} \left(\left(\widehat{\phi_t}\right)_{t \in \cT}\right), $$
		 for any $\widehat{\phi_{\bcdot}} \in \mathbb{C} \left([0,T] : \left( \widehat{\cS_0}, \delta_{\square}\right)\right)$, where the supremum is taken over all finite subsets of $[0,T]$. All that remains to show is $ \widehat{I}^{(\beta^+,\,\beta^-)}_{\epsilon} = \widehat{\cI}^{(\beta^+,\,\beta^-)}_{\epsilon}$.

		The implication $\widehat{I}^{(\beta^+,\,\beta^-)}_{\epsilon} \geq  \widehat{\cI}^{(\beta^+,\,\beta^-)}_{\epsilon}$ is immediate from the definitions. Thus, it is enough to show that $\widehat{I}^{(\beta^+,\,\beta^-)}_{\epsilon} (\widehat{\phi_{\bcdot}}) \leq  \widehat{\cI}^{(\beta^+,\,\beta^-)}_{\epsilon}(\widehat{\phi_{\bcdot}}),$ for any $\widehat{\phi_{\bcdot}} \in \mathbb{C}\left([0,T]:\left(\widehat{\mathcal{S}_0}, \delta_{\square}\right)\right)$. 
        It is enough to show that for any $\gamma >0$, 
		 \begin{align*}
		 &	\inf \left\{ 	I^{(\beta^+,\,\beta^-)}_{\epsilon}(\phi^*_{\bcdot}) : \sup_{t \in [0,T]} \delta_{\square} \left( \widehat{\phi^*_t}, \widehat{\phi_t} \right) < \gamma \right\}  \leq \sup_{\cT} \sup_{\eta >0} 	\inf \left\{ 	I^{(\beta^+,\,\beta^-)}_{\epsilon}(\phi^*_{\bcdot}) : \max_{t \in \cT} \delta_{\square} \left( \widehat{\phi^*_t}, \widehat{\phi_t} \right) < \eta\right\}, \label{eq:eq-rate-ts}
		 \end{align*}
		 assuming the right hand side is finite. 
		 
		 Argument as in Remark \ref{rem:3.17}  shows that  $t \mapsto \widehat{\phi_t}$ is $L(T,\epsilon)$-Lipschitz. Next we fix small enough $\delta, \kappa >0$ such that $2\delta L(T,\epsilon) +  \kappa < \gamma$ and a finite set $\cT^{\prime}=\left\{0 = t_0 < \cdots < t_k = T\right\}$ such that $\max_{i=1}^k (t_i-t_{i-1}) < \delta$. We can get $\phi^{\prime}_{\bcdot}$ such that  $\max_{t \in \cT^{\prime}} \delta_{\square} \left( \widehat{\phi^{\prime}_t}, \widehat{\phi_t} \right) < \kappa$ and 
		 $$ 	I^{(\beta^+,\,\beta^-)}_{\epsilon}(\phi^\prime_{\bcdot}) \leq \inf \left\{ 	I^{(\beta^+,\,\beta^-)}_{\epsilon}(\phi^*_{\bcdot}) : \max_{t \in \cT^\prime} \delta_{\square} \left( \widehat{\phi^*_t}, \widehat{\phi_t} \right) < \kappa \right\} + \delta < \infty.$$
		 Arguing again as in Remark~\ref{rem:3.17} shows  that $t \mapsto \widehat{\phi^{\prime}_t}$ is $L(T,\epsilon)$-Lipschitz. For any $t \in [0,T]$, we can find $r \in \cT^\prime$ such that $|t-r|\leq \delta$ and hence
		 $$ \delta_{\square}\left(\widehat{\phi^\prime_t}, \widehat{\phi_t}\right)  \leq \delta_{\square}\left(\widehat{\phi^\prime_t}, \widehat{\phi^\prime_r}\right)  + \delta_{\square}\left(\widehat{\phi^\prime_r}, \widehat{\phi_r}\right)  + \delta_{\square}\left(\widehat{\phi_r}, \widehat{\phi_t}\right)  \leq 2\delta L(T,\epsilon) + \kappa.$$
		 This implies that $\sup_{t \in [0,T]} \delta_{\square} \left( \widehat{\phi^\prime_t}, \widehat{\phi_t} \right) < \gamma$ and, therefore, 
		  \begin{multline*}
		 		\inf \left\{ 	I^{(\beta^+,\,\beta^-)}_{\epsilon}(\phi^*_{\bcdot}) : \sup_{t \in [0,T]} \delta_{\square} \left( \widehat{\phi^*_t}, \widehat{\phi_t} \right) < \gamma \right\} 
		 	 \leq 	I^{(\beta^+,\,\beta^-)}_{\epsilon}(\phi^\prime_{\bcdot})\\
		 	 \leq \inf \left\{ 	I^{(\beta^+,\,\beta^-)}_{\epsilon}(\phi^*_{\bcdot}) : \max_{t \in \cT^\prime} \delta_{\square} \left( \widehat{\phi^*_t}, \widehat{\phi_t} \right) < \kappa \right\} + \delta  \\
		 	 \leq \sup_{\cT} \sup_{\eta >0} 	\inf \left\{ 	I^{(\beta^+,\,\beta^-)}_{\epsilon}(\phi^*_{\bcdot}) : \max_{t \in \cT} \delta_{\square} \left( \widehat{\phi^*_t}, \widehat{\phi_t} \right) < \eta\right\} + \delta.
		 \end{multline*}
		 	\qed

	\subsection{Special case : Time-homogeneous jump rates}{\label{validrate:cut-time-hom}}
	
	In this section, we shall consider the special case of time-homogeneous jump rates, i.e., $\beta^{\pm}_s = \gamma^{\pm} \in \cls$ for all $s \in [0,T]$. To simplify the presentation, we  only consider the LDP for the time-averaged graphon $\widehat{M^n}$ on the whole time period $[0,T]$ given by Proposition~\ref{prop:cutLDP-avg-whole}. The treatment of the path space LDP for the window-averaged graphon process is  similar and is therefore omitted. We shall prove the following representation of the rate function $\widehat{I}_M$.

	\begin{prop}{\label{prop:LDP-whole-avg-time-hom}}
		Consider the special case when $\beta^{\pm}_s = \gamma^{\pm} \in \cls$ for all $s \in [0,T]$. Suppose Assumption~\ref{assmp:2} is satisfied. Then 
the rate function $\widehat I^{(\beta^+,\;\beta^-)}_M$ in Proposition \ref{prop:cutLDP-avg-whole} has the following form :
		$$ \widehat I^{(\beta^+,\;\beta^-)}_M \left(\widehat f\right) =T \inf_{f \in \widehat{f}} \int_{[0,1]^2} \cQ_1\left(f(x,y),\gamma^+(x,y),\gamma^-(x,y)\right)dx\,dy, \;\; \forall \; \widehat f\in \widehat{\cls_0}.$$
	\end{prop}
	
	The following lemmas will be useful while proving Proposition~\ref{prop:LDP-whole-avg-time-hom}. The first lemma shows that for time-homogeneous jump rates, the search for optimal trajectories can be restricted to constant paths, a consequence of the convexity of the function $\cQ_1$.
	
	\begin{lem}{\label{lem:special-rate-hom}}
			Consider the special case when $\beta^{\pm}_s = \gamma^{\pm} \in \cls$ for all $s \in [0,T]$. In this case, 
			$$ I_M^{(\beta^+, \, \beta^-)}(f) = T \int_{[0,1]^2} \cQ_1\left(f(x,y),\gamma^+(x,y),\gamma^-(x,y)\right)dx\,dy, \;\; \forall \; f\in \cls_0.$$
	\end{lem}
	
	\begin{proof}
	Note that the function $u \mapsto \cQ_1(u,v^+,v^-)$ is convex on $[0,1]$ for any fixed $v^+,v^- >0$. 	Therefore, for any $\phi \in \cW_0$ with $(\int_0^T \phi_s\,ds)/T =f$,
	\begin{align*}
		&	\int_{[0,1]^2_T} \cQ_1(\phi_s(x,y),\gamma^+(x,y),\gamma^-(x,y)) dx\,dy\,ds \\
		& \geq T \int_{[0,1]^2} \cQ_1 \left( \dfrac{1}{T} \int_0^T \phi_s(x,y)\,ds,\gamma^+(x,y),\gamma^-(x,y)\right) dx\,dy\,ds \\
		&=  T \int_{[0,1]^2} \cQ_1 \left( f(x,y),\gamma^+(x,y),\gamma^-(x,y)\right) dx\,dy\,ds,
	\end{align*}
	with the equality holding true if $\phi_s \equiv f$ for all $s \in [0,T]$. This completes the proof.
		\end{proof}
	
    The next elementary lemma establishes the uniform fluctuation of the rate function with respect to the change in jump parameters.

	\begin{lem}{\label{lem:uniform}}
		Let $\beta^{\pm}, \alpha^{\pm} \in \cW$. Then for any $\phi \in \cW_0$, we have 
		$$ 	\bigg \rvert J^{(\beta^+,\,\beta^-)}(\phi) - J^{(\alpha^+,\,\alpha^-)}(\phi) \bigg \rvert  \leq \|\beta^+-\alpha^+\|_1 + \|\beta^--\alpha^-\|_1 + 2\|  \sqrt{\beta^+\beta^-} - \sqrt{\alpha^+\alpha^-} \|_1.$$
		The above inequality is also true when $J^{(\beta^+,\,\beta^-)}(\phi)$ and $ J^{(\alpha^+,\,\alpha^-)}(\phi)$ are replaced by $I_M^{(\beta^+,\,\beta^-)}(f)$ and $I_M^{(\alpha^+,\,\alpha^-)}(f)$, respectively, for any $f \in \cls_0$. 
	\end{lem}
	
	\begin{proof}
		For any $\phi \in \cW_0$, we have 
		\begin{align*}
			\bigg \rvert J^{(\beta^+,\,\beta^-)}(\phi) - J^{(\alpha^+,\,\alpha^-)}(\phi) \bigg \rvert 
            &\leq \int_{[0,1]^2_T} \bigg \rvert \left( \sqrt{\beta^+(1-\phi)} - \sqrt{\beta^-\phi}\right)^2 - \left( \sqrt{\alpha^+(1-\phi)} - \sqrt{\alpha^-\phi}\right)^2 \bigg \rvert  \\
			& \leq \int_{[0,1]^2_T} \big \rvert \beta^+ - \alpha^+\big \rvert (1-\phi) +   \int_{[0,1]^2_T} \big \rvert \beta^- - \alpha^-\big \rvert \phi \\
			& \hspace{1 in} +2\int_{[0,1]^2_T} \big \rvert \sqrt{\beta^+\beta^-} - \sqrt{\alpha^+\alpha^-} \big \rvert \sqrt{\phi(1-\phi)} \\
			& \leq \|\beta^+-\alpha^+\|_1 + \|\beta^--\alpha^-\|_1 + 2\|  \sqrt{\beta^+\beta^-} - \sqrt{\alpha^+\alpha^-} \|_1.
			\end{align*}
            This proves the first statement in the lemma.
			The second assertion in the lemma is immediate on recalling the relationship between $I_M^{(\beta^+,\,\beta^-)}$ (resp. $I_M^{(\alpha^+,\,\alpha^-)}$) and
            $J^{(\beta^+,\,\beta^-)}$(resp. $ J^{(\alpha^+,\,\alpha^-)}$).
	\end{proof}

    The main ingredient in the proof of Proposition~\ref{prop:LDP-whole-avg-time-hom} is the following lemma, proof is given in Section \ref{add:proofvalid}.
    
	\begin{lem}{\label{lem:add}}
	Fix $M \in (0,\infty)$ and suppose that $\gamma^{\pm} \in \cls$ are continuous on $[0,1]^2$ taking values in $[0,M]$. Let $\Delta : [0,1]\times [0,M] \times [0,M] \to \mathbb{R}$ be a function satisfying the following conditions :
	\begin{enumerate}
		\item $\Delta$ is continuous.
		\item The map $u \mapsto \Delta (u,v^+,v^-)$ is convex on $[0,1]$ for any $v^+,v^- \in [0,M]$.
		\item For any $\delta \in (0,1/2)$, there exists $K_{\delta} \in (0,\infty)$ such that $\delta K_{\delta} \to 0$ as $\delta \to 0$ and 
		$$ \rvert \Delta\left(u_1,v^+,v^-\right) - \Delta\left(u_2,v^+,v^-\right) \rvert \leq K_{\delta}|u_1-u_2|, \; \mbox{ for all } \; v^+,v^- \in [0,M], \, u_1 \in [0,1],$$
		whenever  $u_2 \in [\delta,1-\delta]$.
	\end{enumerate}
	Let for $f \in \cls_0$ and $\widehat f \in \widehat{\cls_0}$
	$$ \mathscr{I}(f) := \int_{[0,1]^2} \Delta\left(f(x,y),\gamma^+(x,y),\gamma^-(x,y)\right)\,dx\,dy, \;\;  \widehat{\mathscr{I}} \left(\widehat f\right) := \inf_{f \in \widehat f} \mathscr{I}(f).$$
	Then $\widehat{\mathscr{I}}$ is lower semi-continuous on $\widehat{\cls_0}$ with respect to the metric $\delta_{\square}$.
\end{lem}

	\begin{proof}[Proof of Proposition~\ref{prop:LDP-whole-avg-time-hom}]
	In light of Proposition~\ref{prop:cutLDP-avg-whole} and Lemma~\ref{lem:special-rate-hom}, it suffices to show that the map 
	$$ \widehat{\cls_0} \ni \widehat f \mapsto \inf_{f \in \widehat f} I_M^{(\text{hom},\gamma^+,\,\gamma^-)}(f)$$
	is lower semi-continuous, where 
	$$ I_M^{(\text{hom},\gamma^+,\,\gamma^-)}(f):=T \int_{[0,1]^2} \cQ_1\left(f(x,y),\gamma^+(x,y),\gamma^-(x,y)\right)dx\,dy, \;\; \forall \; f\in \cls_0.$$ By Assumption \ref{assmp:2}, we have $\gamma^{\pm} \in L^1\left([0,1]^2\right)$ and hence we can get a sequence $\left\{\gamma^{n,\pm} : n \geq 1\right\} \subseteq \cls$ such that $\gamma^{n,\pm}$ is continuous and $\|\gamma^{\pm} - \gamma^{n,\pm}\|_1 \longrightarrow 0.$ Note that
	\begin{align*}
		\varepsilon_k &:= \sup_{\widehat f \in \widehat{\cls_0}} \bigg \rvert \inf_{f \in \widehat f} I_M^{(\text{hom},\gamma^{k,+},\,\gamma^{k,-})}(f) - \inf_{f \in \widehat f} I_M^{(\text{hom},\gamma^+,\,\gamma^-)}(f)\bigg \rvert \\
		&\leq  \sup_{f \in  \cls_0} \bigg \rvert I_M^{(\text{hom},\gamma^{k,+},\,\gamma^{k,-})}(f) -  I_M^{(\text{hom},\gamma^+,\,\gamma^-)}(f)\bigg \rvert \\
		& \leq \|\gamma^{k,+}-\gamma^+\|_1 + \|\gamma^{k,-}-\gamma^-\|_1 + 2 \| \sqrt{\gamma^{k,+}\gamma^{k,-} }- \sqrt{\gamma^{+}\gamma^-}\|_1 \longrightarrow 0,
	\end{align*}
	as $k \to \infty$, where the second inequality is a consequence of Lemma~\ref{lem:uniform}. In view of the above uniform convergence it now suffices to show that for any continuous $\gamma^{\pm}$, the map $\widehat f \mapsto \inf_{f \in \widehat f} I_M^{(\text{hom},\gamma^{+},\,\gamma^{-})}(f)$ is lower semi-continuous.

    Since $\gamma^{\pm}$ are continuous, $\sup_{(x,y) \in [0,1]^2} \gamma^{\pm}(x,y) := M <\infty$. Note that,
	 $\cQ_1$ is  continuous. Moreover, $u \mapsto \cQ_1(u,v^+,v^-)$ is convex on $[0,1]$ for any $v^{\pm} \in [0,M]$. Finally, for any $\delta \in (0,1/2)$ and $u_1, u_2 \in [0,1]$ with $u_2 \in [\delta,1-\delta]$, we have 
	\begin{align*}
		\big \rvert \cQ_1(u_1,v^+,v^-) - \cQ_1(u_2,v^+,v^-)\big \rvert 
        &\leq |u_1-u_2||v^+-v^-| + 2\sqrt{v^+v^-} \,\Big \rvert \sqrt{u_1(1-u_1)}-\sqrt{u_2(1-u_2)} \Big \rvert \\
		& \leq M|u_1-u_2| +   2M \dfrac{\rvert u_1(1-u_1)-u_2(1-u_2)\rvert}{\sqrt{u_1(1-u_1)}+\sqrt{u_2(1-u_2)}} \\
		& \leq M|u_1-u_2| + 2M \dfrac{|u_1-u_2|}{\sqrt{\delta(1-\delta)}} \leq \left( M+ \dfrac{4M}{\sqrt{\delta}} \right) |u_1-u_2|,
	\end{align*}
	for any $v^+,v^- \in [0,M]$. Note that $\delta \left( M+ \dfrac{4M}{\sqrt{\delta}} \right) \to 0$ as $\delta \to 0$.  Now the proposition follows from an application of Lemma~\ref{lem:add}.
	\end{proof}

		\section{Proofs of rare event asymptotics results}{\label{sec:applications}}
        In this section we collect the proofs of Theorem \ref{thm:likely-path-cond-edge}, Proposition \ref{prop:edge-max-eq}, and Proposition \ref{prop:edge-max-eq2}.

	\begin{proof}[Proof of Theorem~\ref{thm:likely-path-cond-edge}]
		Recall that $F$ is a single directed edge. The constant $\phi \equiv p^*+\delta$ satisfies $\sup_{t \in \cT} \ft(\widehat{\text{Av}^{\epsilon}\left(\phi\right)_t},F) = p^*+\delta$. This shows, by Remark~\ref{rem:domain-J}, that $\widehat{\cH}^{\epsilon}_{\delta, \cT}$ is a subset of domain of the rate function $\widehat I^{(\beta^+,\, \beta^-)}_{\epsilon}$, which, as argued in Remark~\ref{rem:3.17}, is compact. This implies $\widehat{\cH}^{\epsilon}_{\delta, \cT}$ is non-empty and compact.
		
		Next, using Theorem~\ref{thm:LDP-avg-path},
		\begin{align*}
	&	-\limsup_{n \to \infty} \dfrac{1}{a(n)n^2}	\log \mathbb{P} \left( \inf_{\widehat{\phi_{\bcdot}} \in \widehat \cH^{\epsilon}_{\delta, \cT}} \sup_{t \in [0,T]} \delta_{\square} \left( \widehat{A^{n,\epsilon}_t},\, \widehat{\phi_t}\right) \geq \varepsilon ,\; \sup_{t \in \cT} E_t^{n,\epsilon} \geq p^*+\delta\right) \\
	& \geq  \inf \left\{ \widehat I^{(\beta^+,\, \beta^-)}_{\epsilon} \left( \widehat{\phi^*_{\bcdot}}\right) : \inf_{\widehat{\phi_{\bcdot}} \in \widehat \cH^{\epsilon}_{\delta, \cT}} \sup_{t \in [0,T]} \delta_{\square} \left( \widehat{\phi^*_t},\, \widehat{\phi_t}\right) \geq \varepsilon ,\; \sup_{t \in \cT} \ft(\widehat{\phi^*_t},F) \geq p^*+\delta \right\} =:  I_1(\delta),
		\end{align*}
		whereas 
		\begin{align*}
				-\liminf_{n \to \infty} \dfrac{1}{a(n)n^2}	\log \mathbb{P} \left( \sup_{t \in \cT} E_t^{n,\epsilon} \geq p^*+\delta\right)  &\leq  \inf \left\{ \widehat I^{(\beta^+,\, \beta^-)}_{\epsilon} \left( \widehat{\phi^*_{\bcdot}}\right) : \sup_{t \in \cT} \ft(\widehat{\phi^*_t},F) > p^*+\delta \right\} \\
				&=:  I_2(\delta).
		\end{align*}
		It is enough to show that $I_1 > I_2$. In the proof of Proposition~\ref{prop:edge-max-eq}, we shall show that the optimal value of the rate function $\widehat I^{(\beta^+,\, \beta^-)}_{\epsilon}$ over $\widehat{\cH}^{\epsilon}_{\delta, \cT}$ is continuous in $\delta$; therefore 
		$$I_2=  \inf \left\{ \widehat I^{(\beta^+,\, \beta^-)}_{\epsilon} \left( \widehat{\phi^*_{\bcdot}}\right) : \sup_{t \in \cT} \ft(\widehat{\phi^*_t},F) \geq p^*+\delta \right\} =  \inf \left\{ \widehat I^{(\beta^+,\, \beta^-)}_{\epsilon} \left( \widehat{\phi^*_{\bcdot}}\right) :  \widehat{\phi^*_{\bcdot}} \in \widehat{\cH}^{\epsilon}_{\delta, \cT}\right\}.$$
		By arguments presented earlier both the optimums in the definition of $I_1$ and (later) characterization of $I_2$ are attained. But all the optimal points for $I_2$ lie in $\widehat{\cH}^{\epsilon}_{\delta, \cT}$, whereas any optimal point for $I_1$ is at least $\varepsilon$-distance away from this set. This shows that $I_1 > I_2$.
	\end{proof}

	\begin{proof}[Proof of Proposition~\ref{prop:edge-max-eq}] 
		Consider any $\phi \in \cW_0$ which satisfies the constraint in~(\ref{eq:edge-var-prob}). We can strictly improve the objective function by changing $\phi_s$ to $w^*$ on $s \notin U_{\epsilon}(t)$ since $\cQ_1(w^*,\gamma^+,\gamma^-)=0$ on $[0,1]^2$. This change reduces the objective function to 
		$$ \int_{[0,1]^2 \times [0,T]}  \cQ_1(\phi, \gamma^+,\gamma^-) = \int_{[0,1]^2 \times U_{\epsilon}(t)} \cQ_1(\phi, \gamma^+,\gamma^-).$$
		Moreover, $x \mapsto \cQ_1(x,v^+,v^-)$ is strictly convex on $[0,1]$ for any $v^+,v^- >0$ and hence 
		$$ \int_{[0,1]^2} \int_{U_{\epsilon}(t)} \cQ_1(\phi_s, \gamma^+,\gamma^-) \geq |U_{\epsilon}(t)| \int_{[0,1]^2} \cQ_1 \left( \dfrac{1}{U_{\epsilon}(t)} \int_{U_{\epsilon}(t)} \phi_s\,ds, \gamma^+, \gamma^-\right), $$
		and equality holds true if and only if $\phi_s(x,y)$ is constant almost everywhere on $U_{\epsilon}(t)$ for a.e.$(x,y)$. Note that, replacing $\phi_s$ on $U_{\epsilon}(t)$ by its average over this interval does not affect the constraint. 
		We can conclude that any optimizer $\phi$ to~(\ref{eq:edge-var-prob}) takes the following form : 
		$$ \phi_s = \begin{cases}
			h_{\delta}, & \text{for } s \in U_{\epsilon}(t),\\
			w^*, & \text{for } s \notin U_{\epsilon}(t),
		\end{cases}$$
where $h_{\delta} \in \cS_0$ is a solution to the following optimization problem :
		\begin{equation*}{\label{eq:edge-var-prob2}}
			\text{Minimize } \int_{[0,1]^2} \cQ_1(h, \gamma^+,\gamma^-), \text{ subject to } \int_{[0,1]^2} h \geq p^* + \delta.
		\end{equation*}
		Since the unconstrained optimization problem is uniquely solved at $w^*$, strict convexity of $\cQ_1$ guarantees that $h_{\delta} \in \cS_0$ is a solution to 
		\begin{equation}{\label{eq:edge-var-prob3}}
			\text{Minimize } \int_{[0,1]^2} \cQ_1(h, \gamma^+,\gamma^-), \text{ subject to } \int_{[0,1]^2} h = p^* + \delta.
		\end{equation}
		
		To solve~(\ref{eq:edge-var-prob3}), we first assume that $\gamma^{\pm}$ has the block form, i.e., $\gamma^{\pm} = \sum_{i,j\in [n]} \gamma_{i,j}^{\pm}\mathbbm{1}_{Q_{i,j}^n}$ for some $n$. Using the strict convexity of $\cQ_1$ and linearity of the constraint , we can see that the optimum $h_{\delta}$ must be constant on each block $Q_{i,j}^n$. In other words, $h_{\delta} =   \sum_{i,j\in [n]} h_{i,j,\delta}\mathbbm{1}_{Q_{i,j}^n}$, where $\left(h_{i,j,\delta}\right)_{i,j \in [n]}$ is a solution to 
		\begin{equation}{\label{eq:edge-var-prob4}}
			\text{Minimize } \dfrac{1}{n^2} \sum_{i,j \in [n]} \cQ_1(h_{i,j}, \gamma_{i,j}^+,\gamma_{i,j}^-), \text{ subject to } \dfrac{1}{n^2} \sum_{i,j \in [n]} h_{i,j}= p^* + \delta.
		\end{equation}
		Linear constraint and strictly convex objective function in~(\ref{eq:edge-var-prob4}) guarantees the existence of an unique solution which can be obtained through \textit{Lagrange multiplier theorem} : The optimum $(h_{i,j})_{i,j \in [n]}$ solves 
		$$ \cQ_{1,1}(h_{i,j}, \gamma_{i,j}^+,\gamma_{i,j}^-) - \lambda =0, \; \forall \; i,j \in [n], \; \dfrac{1}{n^2} \sum_{i,j \in [n]} h_{i,j}= p^* + \delta,$$
		for some $\lambda \in \mathbb{R}$, where $\cQ_{1,1}$ is the derivative of $\cQ_1$ with respect to its first coordinate. We can simplify the above set of equations to 
	$$ \gamma_{i,j}^+-\gamma_{i,j}^- +\gamma_{i,j}^+\gamma_{i,j}^-\dfrac{1-2h_{i,j}}{\sqrt{h_{i,j}(1-h_{i,j})}} + \lambda =0, \; \forall \; i,j \in [n], \; \dfrac{1}{n^2} \sum_{i,j \in [n]} h_{i,j}= p^* + \delta.$$
	A few algebraic manipulations result in 
	$$ h_{i,j} = \dfrac{1}{2} + \dfrac{\gamma_{i,j}^+-\gamma_{i,j}^-+\lambda}{2\sqrt{\left(\gamma_{i,j}^+-\gamma_{i,j}^-+\lambda\right)^2+4\gamma_{i,j}^+\gamma_{i,j}^-}}, \forall \; i,j \in [n], \; \dfrac{1}{n^2} \sum_{i,j \in [n]} h_{i,j}= p^* + \delta.$$
Note that, the expressions for $h_{i,j}$'s above are strictly increasing in $\lambda$ with the values ranging from $0$ to $1$ as $\lambda$ varies from $-\infty$ to $\infty$. Hence, there is an unique $\lambda$ satisfying the above equation. This completes the proof for this special case of $\gamma^{\pm}$ with block structure.

To complete the proof for general $\gamma^{\pm}$, we define $\gamma^{n,\pm}$ to be the block approximation of $\gamma^{\pm}$ as defined in Example~\ref{ex:block}. Recall that $\gamma^{n,\pm}$ converges in $L^1$ and almost everywhere to $\gamma^{\pm}$ on $[0,1]^2$, therefore
$$ \kappa_n := \| \gamma^{n,+}-\gamma^+\|_1 +  \| \gamma^{n,-}-\gamma^-\|_1 +  2\| \sqrt{\gamma^{n,+}\gamma^{n,-}}-\sqrt{\gamma^+\gamma^-}\|_1  \longrightarrow 0.$$
 Let $h_n \in \cS_0$ be the unique minimizer  for~(\ref{eq:edge-var-prob3}) when $\gamma^{\pm}$ is replaced by $\gamma^{n,\pm}$. By out previous argument for block kernels, we know that $h_n$ takes the form 
$$ h_n= \dfrac{1}{2} + \dfrac{\gamma^{n,+}-\gamma^{n,-}+\lambda_n}{2\sqrt{\left(\gamma^{n,+}-\gamma^{n,-}+\lambda_n\right)^2 + 4 \gamma^{n,+}\gamma^{n,-}}}, \; \; \int_{[0,1]^2} h_n = p^*+\delta,$$
for some $\lambda_n \in \mathbb{R}$. We can get a subsequence along which $\lambda_n$ converges to $\lambda \in [-\infty,\infty]$ and without loss of generality we assume this happens along the entire subsequence. Then $h_n$ converges almost everywhere (and hence in $L^1$ as $0 \leq h_n \leq 1$) to $h^*$ which takes the form 
$$ h^*= \dfrac{1}{2} + \dfrac{\gamma^{+}-\gamma^{-}+\lambda}{2\sqrt{\left(\gamma^{+}-\gamma^{-}+\lambda\right)^2 + 4 \gamma^{+}\gamma^{-}}}, \; \; \int_{[0,1]^2} h^* = p^*+\delta.$$
Note that $\lambda=\infty$ or $-\infty$ would imply that $h^* \equiv 1$ or $0$, respectively, which does not satisfy the integral constraint, hence $\lambda \in \mathbb{R}$. For any $h \in \cS_0$, satisfying the integral constraint in	~(\ref{eq:edge-var-prob3}), we have 
\begin{align*}
	\int_{[0,1]^2} \cQ_1(h,\gamma^+,\gamma^-) &\geq 	\liminf_{n \to \infty} \left(\int_{[0,1]^2} \cQ_1(h,\gamma^{n,+},\gamma^{n,-}) - \kappa_n  \right) \\
	& \geq 	\liminf_{n \to \infty} \int_{[0,1]^2} \cQ_1(h_n,\gamma^{n,+},\gamma^{n,-}) \\
	& \geq 	\liminf_{n \to \infty} \left(\int_{[0,1]^2} \cQ_1(h_n,\gamma^{+},\gamma^{-}) - \kappa_n  \right) \geq \int_{[0,1]^2} \cQ_1(h^*,\gamma^{+},\gamma^{-}),
\end{align*}
where the first and third inequality follow from Lemma~\ref{lem:uniform} and the last from \textit{Fatou's Lemma}. This shows that $h^*$ is a solution to~(\ref{eq:edge-var-prob3}), whereas the uniqueness of the solution is immediate from linear constraint and strict convexity of the objective function. 

The monotonicity and continuity of  $\delta \mapsto \lambda_{\delta}$ is trivial from~(\ref{eq:def-lambda}). Finally, one can plug-in the optimal trajectory in the objective function to get the formula in~(\ref{eq:formula-obj}). Since, $\lambda_{\delta}$ is continuous in $\delta$, same is true for $f^{*}_{\delta}$ (with respect to almost everywhere convergence).  The continuity of the optimal value now follows from dominated convergence and the observation that 
$$ 0 \leq \cQ_1 (f^*_{\delta}, \gamma^+,\gamma^-) \leq \gamma^+ + \gamma^- \in L^1 \left([0,1]^2\right), \; \forall \; \delta >0.$$
\end{proof}

\begin{proof}[Proof of Proposition~\ref{prop:edge-max-eq2}]
Recall that $F$ is a single directed edge. Fix any arbitrary $\widehat{\phi_{\bcdot}}$ which satisfies $\ft\left(\widehat{\phi_t},F\right) \geq p^*+\delta$ for some $t \in \cT$ and has finite rate function value. 
	Using the formulas in Theorem~\ref{thm:LDP-avg-path} and Corollary~\ref{cor:weakLDP-path} we can get $\phi^{n,*} \in \cW_0$ such that 
	\begin{equation}{\label{eq:lim-opt1}}
		J^{(\beta^+,\, \beta^-)} (\phi^{n,*}) \leq \widehat I^{(\beta^+,\, \beta^-)}_{\epsilon}\left(\widehat{\phi_{\bcdot}}\right)  + 2/n, \; \; \sup_{s \in [0,T]} \delta_{\square} \left(\widehat{\phi_s}, \widehat{\text{Av}^{\epsilon}\left(\phi^{n,*}\right)_s}\right) \leq 1/n.
	\end{equation}
	Note that, 
		\begin{align*}
	 \dfrac{1}{|U_{\epsilon}(t)|} \int_{[0,1]^2 \times U_{\epsilon}(t)} \phi^{n,*}_s =  \ft(\widehat{\text{Av}^{\epsilon}(\phi^{n,*})_t},F)\geq \ft(\widehat{\phi_t},F) - 1/n \geq p^*+\delta-1/n. 
	\end{align*}
	Therefore, writing $\delta_n=\delta-1/n$, 
	\begin{align*}
		\widehat I^{(\beta^+,\, \beta^-)}_{\epsilon}\left(\widehat{\phi_{\bcdot}}\right) &\geq \inf \left\{ \int_{[0,1]^2_T} \cQ_1 \left( \phi^*, \gamma^+,\gamma^-\right)   : \dfrac{1}{|U_{\epsilon}(t)|} \int_{[0,1]^2 \times U_{\epsilon}(t)}  \phi^*_s \geq p^*+\delta_n, \phi^* \in \cW_0\right\} \\
		& \geq \inf_{s \in \cT} |U_{\epsilon}(s)| \int_{[0,1]^2} \cQ_1 \left( f^*_{\delta_n}, \gamma^+,\gamma^-\right). 
	\end{align*}
	By Proposition~\ref{prop:edge-max-eq}, $f^*_{\delta_n}$ converges almost everywhere to $f_{\delta}^*$ and hence we conclude that 
	\begin{equation}{\label{eq:charac-lb}}
	\inf \left\{	\widehat I^{(\beta^+,\, \beta^-)}_{\epsilon}\left(\widehat{\phi_{\bcdot}}\right) : \sup_{t \in \cT} \ft(\widehat{\phi_t},F) \geq p^*+\delta \right\} \geq \inf_{s \in \cT} |U_{\epsilon}(s)| \int_{[0,1]^2} \cQ_1 \left( f^*_{\delta}, \gamma^+,\gamma^-\right)=:I. 
	\end{equation}
It is easy to see that $\widehat{\phi_{\bcdot}} = \widehat{\text{Av}^{\epsilon}(\phi^{*,s,\delta})}$ attains the infimum $s \in \cT^*$. This is because 
$$ 	\widehat I^{(\beta^+,\, \beta^-)}_{\epsilon}\left(\widehat{\text{Av}^{\epsilon}(\phi^{*,s,\delta})}\right) \leq I^{(\beta^+,\, \beta^-)}_{\epsilon}\left(\text{Av}^{\epsilon}(\phi^{*,s,\delta})\right) \leq J^{(\beta^+,\, \beta^-)} (\phi^{*,s,\delta}) = I.$$
On the otherhand, if $\widehat{\phi_{\bcdot}}$ attains the infimum with $\ft\left(\widehat{\phi_t},F\right) \geq p^*+\delta$, then we have just established the existence of $\phi^{n,*}\in \cW_0$ which satisfies the conditions in the statement of Lemma~\ref{lem:lem-var-unique} and $\delta_{\square}\left(\widehat{\phi_s}, \widehat{\text{Av}^{\epsilon}\left(\phi^{n,*}\right)_s}\right) \to 0$,  for all $s \in [0,T]$. Lemma~\ref{lem:lem-var-unique} guarantees that $\| \phi^{n,*} - \phi^{*,t,\delta}\|_2 \to 0$. In particular, $\| \text{Av}^{\epsilon}(\phi^{n,*})_s - \text{Av}^{\epsilon}(\phi^{*,t,\delta})_s\|_2 \to 0$. Hence $ \widehat{\phi_s}= \widehat{\text{Av}^{\epsilon}\left(\phi^{n,*}\right)_s}$ for all $s \in [0,T]$. This completes the proof. 		
	\end{proof}

\section{LDP for the Dynamical System}{\label{appn}}

In Section \ref{prop:continter} we prove the well-posedness stated in Propositions \ref{prop:unique} for the stochastic dynamical system $u^n$ introduced in \eqref{model:interact4}, and  we  provide the proof of the LDP for $u^n$ stated in Theorem \ref{thm:interactldp}, in Section 
\ref{ref:ldpdynpr}.

\subsection{Proof of Proposition~\ref{prop:unique}}
Let $z, K$ and $(F,D)$ be as in the statement of the proposition. For any solution $v$ of \eqref{model:interact2}-\eqref{model:interact3}, we have, for $t \in [0,1]$, 
$$ \| v_t(\cdot)\|_2 = \left \lVert z(\cdot) + \int_0^t F(v_s,s,\cdot)\,ds   + \int_0^t \int_0^1 W(\cdot,y,s)D(v_s,s,\cdot,y)\,dy\,ds  \right \rVert_2 \leq K+2LT,$$
whereas, for $0\le t_1 \le t_2 \le T$ and $x \in [0,1]$,
\begin{align*}
	\Big \rvert v(t_2,x)-v(t_1,x) \Big \rvert&\leq  \int_{t_1}^{t_2} \Big \rvert F(v_s,s,x) \Big \rvert \, ds  + \int_{t_1}^{t_2} \int_0^1  W(x,y,s) \Big \rvert D(v_s,s,x,y) \Big \rvert \,dy\, ds \\
    & \leq 2L(t_2-t_1). 
\end{align*}
This shows that 
\begin{equation}{\label{Lipschitz}}
	\| v_{t_2}-v_{t_1}\|_{w,K+2LT} \leq \| v_{t_2}-v_{t_1}\|_2 \leq 2L(t_2-t_1), \; \; \forall \; 0 \leq t_1 \leq t_2 \leq T.
	\end{equation}
	Thus, $v \in C_{2L}\left([0,T]; L_w^2\left([0,1],K+2LT\right) \right)$.

We now argue uniqueness. Let $u,v \in C\left([0,T]; L_w^2\left([0,1],K+2LT\right) \right)$ be two solutions of \eqref{model:interact2}-\eqref{model:interact3}. Define, $\Delta(t):= \| u_t-v_t \|_{2}$ for all $t \in [0,T]$. Clearly, $t \mapsto \Delta(t)$ is Lipschitz continuous. Then, for any $(t,x) \in [0,T]\times [0,1]$, 
\begin{align*}
	\Big \rvert u(t,x) - v(t,x) \Big \rvert &\leq \int_0^t \Big \rvert F(u_s,s,x) - F(v_s,s,x)\Big \rvert\,ds \\
	& \hspace{0.5 in}+ \int_0^t \int_0^1  W(x,y,s)\,\Big \rvert D(u_s,s,x,y) - D(v_s,s,x,y)\Big \rvert\,dy\,ds \\
	& \leq L\int_0^t \|u_s - v_s\|_2\,ds + L\int_0^t \int_0^1  W(x,y,s) \|u_s - v_s \|_2\,dy\,ds \\
    & \leq 2L \int_0^t \Delta(s)\,ds.
\end{align*} 
By an application of Gr\"onwall's inequality, this shows that $u=v$ and uniqueness follows.

Existence of a solution follows by the standard \textit{Picard's iteration method}. It is enough to show that there exists $\varepsilon=\varepsilon(F,D) \in (0,1)$ such that, for any $z \in L^2\left([0,1]\right)$ and $W \in \cW_0$, the following integral equation 
\begin{equation}{\label{eq}}
	v(t,x) = z(x) + \int_0^t F(v_s,s,x)\,ds + \int_0^t \int_0^1 W(x,y,s)D(v_s,s,x,y)\,dy\,ds 
\end{equation}
on $(t,x) \in [0,\varepsilon] \times [0,1]$ has a solution  $v$ such that $v_t \in L^2\left([0,1]\right)$ for any $t \in [0,\varepsilon]$ and $t \mapsto v_t$ is continuous as a function from $[0,\varepsilon]$ to $L^2\left([0,1]\right)$ (equipped with the weak topology). It is immediate that if the solution $v$ indeed exists, it must satisfy $\| v_t \|_2 \leq \|z\|_2 + 2Lt$ and $|v(t,x)-v(s,x)| \leq 2L|t-s|$; therefore $\| v_t - v_s \|_2 \leq 2L|t-s|$ for all $t,s \in [0,\varepsilon].$ Thus $v$ clearly satisfies the conditions stated immediately after \eqref{eq}.

We shall establish the existence of the solution for $\varepsilon = \min(1,T,(1/4L))$. For $M \in (0,\infty)$, define $C_M\left([0,\varepsilon]; L^2\left([0,1]\right)\right)$ to be the set of functions $v$ with $[0,\varepsilon]\times [0,1] \ni (t,x) \mapsto v(t,x) \in \mathbb{R}$ such 
that $v_t \in L^2\left([0,1]\right)$ and $\|v_t-v_s\|_2 \leq M|t-s|$, for all $s,t \in [0,\varepsilon]$, and equipped with the norm 
$$ \big \lVert  v \big \rVert_{C\left([0,\varepsilon];L^2\left([0,1]\right)\right)} := \sup_{t \in [0,\varepsilon]} \| v_t \|_2.$$
Note that $C_M\left([0,\varepsilon]; L^2\left([0,1]\right)\right)$ is a complete metric space.
For $(t,x) \in [0,\varepsilon] \times [0,1]$, set $v^{(0)}(t,x) := z(x)$ and for $n\ge 1$,
$$  v^{(n)}(t,x) := z(x) + \int_0^t F\left(v^{(n-1)}_s,s,x\right)\,ds + \int_0^t \int_0^1 W(x,y,s)D\left(v^{(n-1)}_s,s,x,y\right)\,dy\,ds.$$
Clearly, $\| v^{(n)}_t - v^{(n)}_s \|_2 \leq 2L|t-s|$ for all $t,s \in [0,\varepsilon]$. Thus $v^{(n)} \in C_{2L}\left([0,\varepsilon]; L^2\left([0,1]\right)\right)$. Now set $\delta_n := \sup_{t \in [0,\varepsilon]} \| v^{(n)}_t - v^{(n-1)}_t \|_2$. For any $(t,x) \in [0,\varepsilon] \times [0,1]$ and $n \geq 1$, we have 
\begin{align*}
		 \Big \rvert v^{(n+1)}(t,x) - v^{(n)}(t,x) \Big \rvert 
        &\leq \int_0^t \Big \rvert F\left(v^{(n)}_s,s,x\right) - F\left(v^{(n-1)}_s,s,x\right)\Big \rvert\,ds \\
	& \hspace{0.5in}+ \int_0^t \int_0^1  W(x,y,s)\,\Big \rvert D\left(v^{(n)}_s,s,x,y\right) - D\left(v^{(n-1)}_s,s,x,y\right)\Big \rvert\,dy\,ds \\
	& \leq L\int_0^t \left \lVert v^{(n)}_s - v^{(n-1)}_s \right \rVert_2\,ds + L\int_0^t \int_0^1  W(x,y,s) \left \lVert v^{(n)}_s - v^{(n-1)}_s \right \rVert_2dy\,ds \\
	&\leq 2L\int_0^t \delta_n\,ds \leq 2L\varepsilon \delta_n \leq \delta_n/2,
\end{align*}
and thus $\delta_{n+1} \leq \delta_n/2$ for all $n \geq 1$. This shows that $\left\{v^{(n)} : n \geq 1\right\}$ is a Cauchy sequence in $C_{2L}\left([0,\varepsilon]; L^2\left([0,1]\right)\right)$ and hence converges to some $v \in C_{2L}\left([0,\varepsilon]; L^2\left([0,1]\right)\right)$.  For any $(t,x) \in [0,\varepsilon]\times [0,1]$, 
\begin{align*}
	\bigg \rvert \int_0^t F\left(v_s^{(n)},s,x\right)\,ds -  \int_0^t F\left(v_s,s,x\right)\,ds \bigg \rvert &\leq \int_0^t \Big \rvert F\left(v_s^{(n)},s,x\right)- F\left(v_s,s,x\right)\Big \rvert\,ds \\
	& \leq \int_0^t L \left \lVert v_s^{(n)}-v_s \right \rVert_2ds \leq  \sup_{s \in [0,\varepsilon]} \left \lVert v_s^{(n)}-v_s \right \rVert_2.
\end{align*}  
Similar arguments also show that 
\begin{align*}
&\bigg \rvert \int_0^t \int_0^1 W(x,y,s)D\left(v_s^{(n)},s,x,y\right)\,dy\,ds -  \int_0^t \int_0^1 W(x,y,s) D\left(v_s,s,x,y\right)\,dy\,ds \bigg \rvert \leq  \sup_{s \in [0,\varepsilon]}  \left \lVert v_s^{(n)}-v_s \right \rVert_2.
\end{align*}
Therefore, 
$$ v^{(n)}(t,x) \stackrel{n \to \infty}{\longrightarrow} z(x) +  \int_0^t F\left(v_s,s,x\right)\,ds + \int_0^t \int_0^1 W(x,y,s)D\left(v_s,s,x,y\right)\,dy\,ds=:v^*(t,x)$$
for all $(t,x) \in [0,\varepsilon] \times [0,1]$. It also establishes that 
$$ \big \rvert v^{(n)}(t,x) - v^*(t,x) \big \rvert \leq 2 \sup_{s \in [0,\varepsilon]} \left \lVert v_s^{(n)}-v_s \right \rVert_2,$$
and hence 
$$ \sup_{t \in [0,\varepsilon]} \left \lVert v_t^{(n)}-v^*_t \right \rVert_2 \leq 2 \sup_{s \in [0,\varepsilon]} \left \lVert v_s^{(n)}-v_s \right \rVert_2 \longrightarrow 0.$$
This shows that $v_t=v_t^*$ almost everywhere on $[0,1]$ for all $t \in [0,\varepsilon]$ and hence for any $(t,x) \in [0,\varepsilon]\times [0,1]$, 
\begin{align*}
	v^*(t,x) &= z(x) +  \int_0^t F\left(v_s,s,x\right)\,ds + \int_0^t \int_0^1 W(x,y,s)D\left(v_s,s,x,y\right)\,dy\,ds \\
	& = z(x) +  \int_0^t F\left(v^*_s,s,x\right)\,ds + \int_0^t \int_0^1 W(x,y,s)D\left(v^*_s,s,x,y\right)\,dy\,ds.
\end{align*}
This concludes the proof. \qed

\subsection{Proof of Theorem~\ref{thm:interactldp}}
\label{ref:ldpdynpr}
We begin with the following proposition which says that the solution of the equation in \eqref{model:interact2} is continuous with respect to the initial data and the model parameters.

\begin{prop}{\label{prop:continter}}
	Suppose that $(F^n,D^n)$, $(F,D)$ satisfy  Assumption~\ref{assump:FDconv}.
    Suppose $W^n, W \in \cW_0$ and $z^n, z \in L^2_w\left([0,1],K\right)$ are such that $W^n$ and $z^n$ converge to $W$ and $z$ with respect to the weak topology on $\cW_0$ and in $L^2_w\left([0,1],K\right)$ respectively. Let $v^n$ be the unique solution in $C_{2L}\left([0,1]; L_w^2\left([0,1],K+2LT\right)\right)$ of the following equation :
	\begin{align*}
		\dfrac{\partial}{\partial t} v^n(t,x) &= F^n\left(v^n_t,t,x \right) + \int_0^1 W^n\left(x,y,t\right) D^n(v^n_t,t,x,y)\;dy, \; \; (t,x) \in [0,T] \times [0,1],\\
		v^n(0,x) &= z^n(x), \; \;  x \in [0,1].
	\end{align*}
	Then $v^n$ converges in  $C_{2L}\left([0,1]; L_w^2\left([0,1],K+2LT\right)\right)$ to $v$ which solves the equation
		\begin{align}
		\dfrac{\partial}{\partial t} v(t,x) &= F\left(v_t,t,x \right) + \int_0^1 W\left(x,y,t\right) D(v_t,t,x,y)\;dy, \; \;  (t,x) \in [0,T] \times [0,1], \label{model:interact8}\\
		v(0,x) &= z(x), \; \;  x \in [0,1]. \label{model:interact9}
	\end{align}
\end{prop}
\begin{proof}
Let $(F,D)$, $(F^n, D^n)$, $W^n$, $W$, $z^n$, $z$, $v^n$, $v$, be as in the statement of the proposition. Since, $C_{2L} \left( [0,1]; L^2_w\left([0,1],K+2LT\right)\right)$ is compact, any subsequence of $v^n$ has a further convergent subsequence. 
Relabel the sequence as $\{n\}$ and denote the limit as $u$. It suffices to show that $u=v$.

For $\varphi \in L_{\infty}\left([0,1]\right)$, we have,
$$ \bigg \rvert \int_0^1 \int_0^t \left(F^n\left(v_s^{n},s,x\right) - F\left(v_s^n,s,x\right)\right) \varphi(x)\, ds\,dx \bigg \rvert  \leq  \|\varphi\|_{\infty} d_1(F^n,F;K+2LT) \longrightarrow 0,$$
whereas, using Assumption \ref{assump:FDconv}\eqref{assump:FDconv2}, and since $F$ is bounded,
$$ \int_0^1 \int_0^t F\left(v_s^n,s,x\right)\varphi(x)\,ds\,dx \to \int_0^1 \int_0^t F\left(u_s,s,x\right)\varphi(x)\,ds\,dx.$$ 
This shows that 
$$ \int_0^1 \int_0^t F^n\left(v_s^{n},s,x\right)\varphi(x)\,ds\,dx \to   \int_0^1 \int_0^t F\left(u_s,s,x\right)\varphi(x)\,ds\,dx.$$
By similar arguments we have 
$$ \int_{[0,1]^2} \int_0^t  W^n(x,y,s)\left[D^n\left(v_s^{n},s,x,y\right)-D\left(u_s,s,x,y\right)\right]\varphi(x)\,ds\,dx\,dy \to 0.$$
Since, $W^n$ converges to $W$ weakly on $\cW_0$ and both $D$ and $\varphi$ are uniformly bounded, we have 
\begin{align*}
    & \lim_{n \to \infty} \int_{[0,1]^2} \int_0^t  W^n(x,y,s)D\left(u_s,s,x,y\right)\varphi(x)\,ds\,dx\,dy \\
    & \hspace{2 in} =\int_{[0,1]^2} \int_0^t  W(x,y,s)D\left(u_s,s,x,y\right)\varphi(x)\,ds\,dx\,dy.
\end{align*}
Moreover, $z^n$ converging to $z$ weakly on $L^2\left([0,1]\right)$ implies that 
$$ \int_0^1 z^n(x)\varphi(x)\,dx \stackrel{n \to \infty}{\to}  \int_0^1 z(x)\varphi(x)\,dx.$$
Finally, using the fact  that, for any $t \in [0,T]$ we have $\|v^n_t - u_t\|_{w,K+2LT} \to 0$, we have
\begin{align*}
	&\int_0^1 u(t,x)\varphi(x)\,dx \\
	&= \lim_{n \to \infty} 	\int_0^1 v^n(t,x)\varphi(x)\,dx \\
	&= \lim_{n \to \infty}  \int_0^1 \left[ z^n(x) + \int_0^t F^n\left(v_s^{n},s,x\right)\,ds +\int_0^t \int_0^1 W^n(x,y,s)D^n\left(v^n_s,s,x,y\right)\,dy\,ds \right] \varphi(x)\,dx \\
	&=  \int_0^1 \left[ z(x) + \int_0^t F\left(u_s,s,x\right)\,ds +\int_0^t \int_0^1 W(x,y,s)D\left(u_s,s,x,y\right)\,dy\,ds \right] \varphi(x)\,dx.
\end{align*}
This shows that both $u$ and $v$ solve \eqref{model:interact8}-\eqref{model:interact9}. From uniqueness shown in Proposition \ref{prop:unique} we conclude that $u=v$ and the result follows.
\end{proof}

We now complete the proof of Theorem~\ref{thm:interactldp}.
The main idea is to apply the contraction principle. Since we have a sequence of maps $(F^n,D^n)$ rather than a single continuous map, we proceed as follows.
Define $\mathcal{F}$ to be the space of all (measurable) functions from $L^2\left([0,1]\right) \times [0,T] \times [0,1]$ to $ [-L,L]$, equipped with the metric
$$ d_{\mathcal{F}} \left(F_1,F_2\right) := d_1(F_1, F_2, K+2LT), \; \;  F_1, F_2 \in \mathcal{F}.$$
Similarly,  define $\mathcal{D}$ to be the space of all (measurable) functions from $ L^2\left([0,1]\right) \times [0,T] \times [0,1]^2$ to $[-L,L]$, equipped with the metric
$$ d_{\mathcal{D}} \left(D_1,D_2\right) := d_2(D_1, D_2, K+2LT), \;   \; D_1, D_2 \in \mathcal{D}.$$
Then $d_{\mathcal{F}}(F^n,F) \to 0$ and thus the sequence $\left\{F^n : n \geq 1\right\}$ satisfies a (trivial) LDP in the space $\mathcal{F}$ with speed $a(n)n^2$ and rate function $I_{\mathcal{F}} = \infty \mathbbm{1}_{\{F\}^c}$ (with the convention that $\infty \times 0 :=0$). Similarly, the sequence $\left\{D^n : n \geq 1\right\}$ satisfies an LDP in the space $\mathcal{D}$ with speed $a(n)n^2$ and rate function $I_{\mathcal{D}} = \infty \mathbbm{1}_{\{D\}^c}$. The random variables $H^n$, $z^n$, are independent of each other. Moreover, by Proposition~\ref{prop:continter}, the map which takes $(H^n,z^n,F^n,D^n)$ to $v^n$ is continuous. Therefore, by the contraction principle, $\left\{v^n : n \geq 1\right\}$ satisfies an LDP on the space $C_{2L}\left([0,1];L_w^2 \left([0,1],K+2LT\right)\right)$, with speed $a(n)n^2$ and rate function $I_{\text{solution}}$ given by 
\begin{equation}{\label{rate}}
	I_{\text{solution}}(v) := \inf_{\left(\phi,z,F^*,D^*\right) \in \mathcal{C}^*(v)} \left[ J^{(\beta^+, \, \beta^-)}\left(\phi \right) + I_{\text{initial}}(z)+I_{\mathcal{F}}(F^*)+I_{\mathcal{D}}(D^*)\right], 
\end{equation}
where $\mathcal{C}^*(v)$ is the set of all $(\phi,z,F^*,D^*) \in \cW_0 \times L_w^2\left([0,1],K\right) \times \mathcal{F} \times \mathcal{D}$ such that $v$ solves the equations
\begin{align*}
	\dfrac{\partial}{\partial t} v(t,x) &= F^*\left(v_t,t,x \right) + \int_0^1 \phi\left(x,y,t\right) D^*(v_t,t,x,y)\;dy, \; \; \forall \; (t,x) \in [0,T] \times [0,1],\\
	v(0,x) &= z(x), \; \; \forall \; x \in [0,1].
\end{align*}
Since $(F^*,D^*)$ must equal $(F,D)$ at the infimum in \eqref{rate}, we conclude the proof.

	\section{Some Auxiliary Results}
\label{sec:aux}
In this section we provide proofs of some auxiliary statements that were deferred for later.

\subsection{Proof of Lemma~\ref{lem:unique-speed}} Let $I(x_0)=0$ and take any $x \neq x_0$. Consider an open ball $U$ containing $x$ such that $x_0 \notin \bar U$, where $\bar U$ is the closure of $U$. 

{We claim that $\inf_{y \in \bar U} I(y) >0$.
Indeed, suppose that $\inf_{y \in \bar U} I(y) =0$. Then there exists $\{y_n\}\subset \bar U$ such that $I(y_n) \le 1/n$. Since the set $\{y: I(y) \le 1\}$ is compact, we conclude that the sequence $\{y_n\}$ must converge along some subsequence, to some $y \in \bar U$. By lower semicontinuity of $I$ we mist have $I(y) =0$.
However, since $\bar U$ does not contain $x_0$, this contradicts the fact that $I^{-1}\left(\left\{0\right\}\right)$ is a singleton. This proves the claim.}

By (\ref{ldpdefined}), we now have 
$$ \liminf_{n \to \infty} \dfrac{\log \mathbb{P}(Y_n \in U)}{b(n)} \leq \limsup_{n \to \infty} \dfrac{\log \mathbb{P}(Y_n \in \bar U)}{b(n)} \leq - \inf_{y \in \bar U} I(y) <0.$$
This implies, 
$$-\inf_{y \in U} I^*(y) \leq  \liminf_{n \to \infty} \dfrac{\log \mathbb{P}(Y_n \in U)}{b^*(n)} = -\infty,$$
and hence   $I^*(x)=\infty$. This shows that $I^*$ is infinite on $S \setminus \left\{x_0\right\}$. This completes the proof.

\subsection{Proof of (\ref{tightbound})} 
\label{sec:sub7.4}
Recall we have that $\Lambda^n \to \Lambda$ a.s.
Fix $M, \varepsilon \in (0,\infty)$. Then $\beta^+ \wedge M \in L^1\left([0,1]^2_T\right)$ and thus we can find continuous $\beta^* :[0,1]^2 \times [0,T] \to [0,M]$ such that $\|\beta^+ \wedge M - \beta^*\|_1 \leq \varepsilon/M$. This yields,
\begin{align}
	\E	\bigg \rvert \int_{\clv} \left( \ell(v^+) \wedge M \right)\left(\beta_s^+(x,y) \wedge M - \beta^*_{s}(x,y)\right)\, \Lambda^n(d\theta)  \bigg \rvert \leq M \int_{[0,1]^2_T} \big \rvert \beta^+ \wedge M -\beta^* \big \rvert \leq \varepsilon. \label{appproof:1}
\end{align}
The above inequality also holds  when $\Lambda^n$ is replaced by $\Lambda$. From Assumption \ref{assmp:1}\eqref{assump:1Aweak}, we also have 
\begin{align*}
		&\E	\bigg \rvert \int_{\clv} \left( \ell(v^+) \wedge M \right)\left(\beta_s^{n,+}(x,y)  - \beta^+_{s}(x,y)\right)\, \Lambda^n(d\theta)  \bigg \rvert \\
        & \leq M \int_{[0,1]^2_T} \big \rvert \beta^{n,+}_s(x,y) - \beta^+_s(x,y)\big \rvert \, dx\,dy\,ds  \to 0, \;\; \text{as } n \to \infty.
\end{align*}
Again, the above convergence holds true even when $\Lambda^n$ is replaced by $\Lambda$. Combining the above estimates, we get
\begin{align*}
	\liminf_{n \to \infty} \E \int_{\clv}  \ell(v^+)\beta_s^{n,+}(x,y)\, \Lambda^n(d\theta) & \geq \liminf_{n \to \infty} \E \int_{\clv} \left( \ell(v^+) \wedge M \right)\beta_s^{n,+}(x,y) \, \Lambda^n(d\theta) \\
	& = \liminf_{n \to \infty} \E \int_{\clv} \left( \ell(v^+) \wedge M \right)\beta_s^{+}(x,y) \, \Lambda^n(d\theta) \\
	& \geq \liminf_{n \to \infty} \E \int_{\clv} \left( \ell(v^+) \wedge M \right)\left(\beta_s^+(x,y) \wedge M \right)\, \Lambda^n(d\theta) \\
	& \geq \liminf_{n \to \infty}\E \int_{\clv} \left( \ell(v^+) \wedge M \right)\beta^*_s(x,y)\, \Lambda^n(d\theta) - \varepsilon.
\end{align*}
Since the function $(v^+,v^-,u,x,y,s) \mapsto ( \ell(v^+) \wedge M)\beta_s^*(x,y)$ is bounded continuous on $\clv$ and $\Lambda^n \stackrel{a.s.}{\to} \Lambda$, we have 
\begin{align*}
		\liminf_{n \to \infty} \E \int_{\clv}  \ell(v^+)\beta_s^{n,+}(x,y)\, \Lambda^n(d\theta)& \geq \liminf_{n \to \infty}\E \int_{\clv} \left( \ell(v^+) \wedge M \right)\beta^*_s(x,y)\, \Lambda^n(d\theta) - \varepsilon \\
		 &= \E \int_{\clv} \left( \ell(v^+) \wedge M \right)\beta^*_s(x,y)\, \Lambda(d\theta) - \varepsilon \\
	& \geq \E \int_{\clv} \left( \ell(v^+) \wedge M \right)\left(\beta_s^+(x,y) \wedge M \right)\, \Lambda(d\theta) - 2\varepsilon.
\end{align*}
Taking $\varepsilon \downarrow 0$ and $M \uparrow \infty$, we conclude the proof. \qed

\subsection{Proof of (\ref{eq:conv2})}
\label{sec:sub7.5} 
We will argue  that for any continuous (and hence bounded) $f:[0,1]^2 \times [0,T] \to \mathbb{R}$, we have 
\begin{equation}{\label{toshow}}
	\int_{\clv} f(x,y,s)\beta_s^{n,+}(x,y)(1-u)v^+ \, \Lambda^n(d\theta) \to   \int_{\clv} f(x,y,s)\beta_s^{+}(x,y)(1-u)v^+ \, \Lambda(d\theta),
\end{equation} 
in probability,
under Assumption \ref{assmp:1}, Assumption \ref{assmp:2}\eqref{assump:1A}, and conditions 
 $\Lambda^n \stackrel{a.s.}{\to} \Lambda$ and 
\begin{equation}{\label{cond1}}
	\sup_{n \geq 1} \E \int_{\clv} \ell(v^+)\beta_s^{n,+}(x,y)\, \Lambda^n(d\theta) =: C < \infty.
\end{equation} 
Analogous statement involving $\beta^{n,-}$ can be shown similarly. Combining these two statements with \eqref{eq:conv} yields \eqref{eq:conv2}.

Without loss of generality, by using a subsequential argument, we also assume that the convergence in Assumption \ref{assmp:2}\eqref{assump:1A} occurs almost everywhere on $[0,1]^2_T$.  Note that condition~\eqref{cond1}, by virtue of \eqref{tightbound}, also implies that $\E \int_{\clv} \ell(v^+)\beta_s^{+}(x,y)\, \Lambda(d\theta) \leq C$. Thus 
$$ \E \int_{\clv} v^+\beta_s^{+}(x,y)\, \Lambda(d\theta) \leq \E \int_{\clv} (\ell(v^+)+2)\beta_s^{+}(x,y)\, \Lambda(d\theta) \leq C + 2\int_{[0,1]^2_T} \beta^+ < \infty,$$
	and hence the right hand side of \eqref{toshow} is a well-defined random variable. Here we have used that $x \leq \ell(x)+2$ for all $x \geq 0$. 

 We will make repeated use of the following basic inequality :
$$ ab \leq \frac{1}{r}\left(\ell(a) + \exp(r b) -1\right), \; \; \mbox{ for all } \; a,b \geq 0, r >0.$$
Fix $M, M^*, r, r^* \in (0,\infty)$ and $\varepsilon \in (0,1)$. We have
\begin{align*}
&	\bigg \rvert \int_{\clv} f(x,y,s)\left[\beta_s^{n,+}(x,y) - \left(\beta_s^{n,+}(x,y) \wedge M \right)\right](1-u)v^+ \, \Lambda^n(d\theta)\bigg \rvert \\
&  \leq \|f\|_{\infty}\int_{\clv} \beta_s^{n,+}(x,y)\mathbbm{1}_{(\beta_s^{n,+}(x,y)) \geq M)} v^+ \Lambda^n(d\theta) \\
&  \leq \|f\|_{\infty}\int_{\clv}\dfrac{\beta_s^{n,+}(x,y)}{r} \left[ \ell(v^+) + \exp \left(r \mathbbm{1}_{(\beta_s^{n,+}(x,y)) \geq M)}\right) -1 \right]  \Lambda^n(d\theta)\\
& \leq \dfrac{\|f\|_{\infty}}{r} \int_{\clv} \ell(v^+) \beta_s^{n,+}(x,y) \Lambda^n(d\theta) \\
& \hspace{1 in} + \dfrac{\|f\|_{\infty}\left(\exp(r)-1\right)}{r} \int_{[0,1]^2_T} \beta_s^{n,+}(x,y)\mathbbm{1}_{(\beta_s^{n,+}(x,y)) \geq M)}\, dx\,dy\,ds. 
\end{align*}
Taking expectations on both sides and using~\eqref{cond1}, we obtain
\begin{align*}
	&	\E \, \bigg \rvert \int_{\clv} f(x,y,s)\left[\beta_s^{n,+}(x,y) - \left(\beta_s^{n,+}(x,y) \wedge M \right)\right](1-u)v^+ \, \Lambda^n(d\theta)\bigg \rvert \\
	&\leq \dfrac{C\|f\|_{\infty}}{r}+ \dfrac{\|f\|_{\infty}\left(\exp(r)-1\right)}{r} \int_{[0,1]^2_T} \beta_s^{n,+}(x,y)\mathbbm{1}_{(\beta_s^{n,+}(x,y)) \geq M)}\, dx\,dy\,ds. 
	\end{align*}
Since $\beta^{n,+} {\to} \beta^+$ almost everywhere and in $L^1$, we use Fatou's Lemma to conclude 
\begin{align*}
	\limsup_{n \to \infty} \int_{[0,1]^2_T} \beta^{n,+}\mathbbm{1}_{(\beta^{n,+} \geq M)} &= \int_{[0,1]^2_T} \beta^+ - \liminf_{n \to \infty} \int_{[0,1]^2_T} \beta^{n,+}\mathbbm{1}_{(\beta^{n,+} < M)}\\
	& \leq \int_{[0,1]^2_T} \beta^+ - \int_{[0,1]^2_T} \beta^+\mathbbm{1}_{\left(\beta^+<M\right)} = \int_{[0,1]^2_T} \beta^+\mathbbm{1}_{\left(\beta^+\geq M\right)}.
\end{align*}
Therefore,
\begin{align}
	&	\limsup_{n \to \infty} \E \, \bigg \rvert \int_{\clv} f(x,y,s)\left[\beta_s^{n,+}(x,y) - \left(\beta_s^{n,+}(x,y) \wedge M \right)\right](1-u)v^+ \, \Lambda^n(d\theta)\bigg \rvert  \nonumber \\
	&\hspace{1.5 in}  \leq \dfrac{C\|f\|_{\infty}}{r}+ \dfrac{\|f\|_{\infty}\left(\exp(r)-1\right)}{r} \int_{[0,1]^2_T} \beta^{+}\mathbbm{1}_{(\beta^{+} \geq M)}.  \label{step1}
\end{align}
Similar estimates show 
\begin{align}
	&	 \E\,\bigg \rvert \int_{\clv} f(x,y,s)\left(\beta_s^{+}(x,y)-\beta_s^+(x,y)\wedge M\right)(1-u)v^+ \, \Lambda(d\theta) \bigg \rvert  \nonumber \\
	&\hspace{2 in}  \leq \dfrac{C\|f\|_{\infty}}{r}+ \dfrac{\|f\|_{\infty}\left(\exp(r)-1\right)}{r} \int_{[0,1]^2_T} \beta^{+}\mathbbm{1}_{(\beta^{+} \geq M)}. \label{step11} 
\end{align}
Also,
\begin{align*}
	&	\bigg \rvert \int_{\clv} f(x,y,s)\left(\beta_s^{+}(x,y) \wedge M - \beta_s^{n,+}(x,y) \wedge M \right)(1-u)v^+ \, \Lambda^n(d\theta)\bigg \rvert \\
	& \leq \dfrac{\|f\|_{\infty}}{r} \left[\int_{\clv} \ell(v^+) \Lambda^n(d\theta) +  \int_{[0,1]^2_T} \exp\left(r\,|\beta^+ \wedge M - \beta^{n,+} \wedge M |\right)-1\right]. 
\end{align*}
Taking expectations on both sides and recalling that $\beta^{n,+} \stackrel{a.e.}{\longrightarrow}\beta^{+}$ and 
$$ \E\, \int_{\clv} \ell(v^+) \Lambda^n(d\theta) \leq c_{\beta}^{-1} \E \int_{\clv} \ell(v^+) \beta_s^{n+}(x,y) \Lambda^n(d\theta) \leq c_{\beta}^{-1}C,$$
we conclude that
\begin{equation}{\label{step2}}
 \limsup_{n \to \infty} \E \,	\bigg \rvert \int_{\clv} f(x,y,s)\left(\beta_s^{+}(x,y) \wedge M - \beta_s^{n,+}(x,y) \wedge M\right)(1-u)v^+ \, \Lambda^n(d\theta) \bigg \rvert \leq \dfrac{C\|f\|_{\infty}}{r c_{\beta}}.
\end{equation}
Next, we can find continuous $\beta^* :[0,1]^2_T \to [0,M]$ such that $\|\beta^+ \wedge M - \beta^*\|_1 \leq \varepsilon \exp(-r M)$. This yields,
\begin{align}
	&\E	\, \bigg \rvert \int_{\clv} f(x,y,s)\left(\beta_s^{+}(x,y) \wedge M - \beta_s^*(x,y)\right)(1-u)v^+ \, \Lambda^n(d\theta) \bigg \rvert \nonumber \\
	&\hspace{ 0.5 in}  \leq \|f\|_{\infty}\E \, \bigg \rvert \int_{\clv} \left(\beta_s^{+}(x,y) \wedge M - \beta_s^*(x,y) \right)v^+ \, \Lambda^n(d\theta) \bigg \rvert \nonumber  \\
	& \hspace{0.5 in} \leq \dfrac{\|f\|_{\infty}}{r} \E \int_{\clv} \ell(v^+)\Lambda^n(d\theta) + \dfrac{\|f\|_{\infty}}{r} \int_{[0,1]^2_T} \left[ \exp\left( r\, |\beta^+ \wedge M - \beta^*|\right) -1 \right] \nonumber \\
	&  \hspace{0.5 in} \leq \dfrac{C\|f\|_{\infty}}{r c_{\beta}}  + \|f\|_{\infty}\exp(r M) \int_{[0,1]^2_T}  |\beta^+ \wedge M - \beta^*|  \leq \|f\|_{\infty}\left(\dfrac{C}{r c_{\beta}} + \varepsilon \right). \label{step3} 
\end{align}
The same estimate is true if $\Lambda^n$ is replaced by $\Lambda$. 
Similarly, for $r^*, M^*>0$, with $\bar \Lambda^n = \E\Lambda^n$,
and $\bar \Lambda = \E\Lambda$,
\begin{align}
&	\limsup_{n \to \infty} \E\,	\bigg \rvert \int_{\clv} f(x,y,s)\beta_s^*(x,y)(1-u)|v^+ - v^+ \wedge M^*| \, \Lambda^n(d\theta) \bigg \rvert  \nonumber\\
&\hspace{0.5 in}  \leq M\|f\|_{\infty} \limsup_{n \to \infty} \E \int_{\clv} v^+\mathbbm{1}_{(v^+ \geq M^*)}\, \Lambda^n(d\theta) \nonumber \\
	& \hspace{0.5 in} \leq \dfrac{M\|f\|_{\infty}}{r^*} \left[ \limsup_{n \to \infty} \E \int_{\clv}\left(\ell(v^+)+\exp(r^*\mathbbm{1}_{(v^+\geq M^*)})-1\right)\Lambda^n(d\theta) \right] \nonumber \\
	& \hspace{0.5 in} \leq \dfrac{M\|f\|_{\infty}}{r^*} \left[ \dfrac{C}{c_{\beta}}+ \left(\exp(r^*)-1\right) \limsup_{n \to \infty} \bar \Lambda^n\left(v^+ \geq M^* \right) \right] \nonumber \\
	& \hspace{0.5 in} \leq \dfrac{M\|f\|_{\infty}}{r^*} \left[ \dfrac{C}{c_{\beta}}+ \left(\exp(r^*)-1\right)  \bar \Lambda\left(v^+ \geq M^* \right) \right],  \label{step4}
\end{align}
where the last inequality uses the fact that $\bar \Lambda^n$ converges weakly to $\bar \Lambda$. Similar computations show that
\begin{align}
	& \E	\bigg \rvert \int_{\clv} f(x,y,s)\beta_s^*(x,y)(1-u)|v^+ - v^+ \wedge M^*| \, \Lambda(d\theta) \bigg \rvert \nonumber \\
    & \leq \dfrac{M\|f\|_{\infty}}{r^*} \left[ \dfrac{C}{c_{\beta}}+ \left(\exp(r^*)-1\right)  \bar \Lambda\left(v^+ \geq M^* \right) \right]. \label{step41}
\end{align} 
Finally, since $\Lambda^n \stackrel{a.s.}{\longrightarrow} \Lambda$ and $(v^+,v^-,u,x,y,s) \mapsto f(x,y,s)\beta_s^*(x,y)(1-u)(v^+ \wedge M^*)$ is bounded and continuous on $\clv$, we have 
\begin{equation}{\label{step5}}
	\lim_{n \to \infty}	\E	\, \bigg \rvert \int_{\clv} f(x,y,s)\beta_s^*(x,y)(1-u)(v^+ \wedge M^*)\, (\Lambda^n-\Lambda)(d\theta) \bigg \rvert =0.
\end{equation}
Combining \eqref{step1},\eqref{step11},\eqref{step2},\eqref{step3},\eqref{step4},\eqref{step41} and \eqref{step5}, we can write the following for any choice of $M,M^*,r,r^* >0$ and $\varepsilon \in (0,1)$ :
\begin{align}
	&\limsup_{n \to \infty} \E \, \bigg \rvert \int_{\clv} f(x,y,s)\beta_s^{n,+}(x,y)(1-u)v^+\, \Lambda^n(d\theta) -  \int_{\clv} f(x,y,s)\beta_s^{+}(x,y)(1-u)v^+\, \Lambda(d\theta) \bigg \rvert  \nonumber \\
& \leq \dfrac{2C\|f\|_{\infty}}{r}+ \dfrac{2\|f\|_{\infty}\left(\exp(r)-1\right)}{r} \int_{[0,1]^2_T} \beta^{+}\mathbbm{1}_{(\beta^{+} \geq M)} + \dfrac{C\|f\|_{\infty}}{r c_{\beta}} + 2\|f\|_{\infty}\left(\dfrac{C}{r c_{\beta}} + \varepsilon \right) \nonumber \\
& \hspace{2 in} + \dfrac{2M\|f\|_{\infty}}{r^*} \left[ \dfrac{C}{c_{\beta}}+ \left(\exp(r^*)-1\right) \bar \Lambda\left(v^+ \geq M^* \right) \right]. \label{step6}
\end{align}
Now take $M^* \uparrow \infty, r^* \uparrow \infty, M \uparrow \infty, r \uparrow \infty$ and $\varepsilon \downarrow 0$, in that specified order, in the right hand side of \eqref{step6} to conclude that 
$$ \int_{\clv} f(x,y,s)\beta_s^{n,+}(x,y)(1-u)v^+\, \Lambda^n(d\theta) \stackrel{L^1}{\longrightarrow}  \int_{\clv} f(x,y,s)\beta_s^{+}(x,y)(1-u)v^+\, \Lambda(d\theta),$$
which completes the proof. \qed

\subsection{Proof of Lemma~\ref{lem:add}} {\label{add:proofvalid}}
 The following lemma will be needed. The proof of this lemma can be found in~\cite[pp. 572--574]{Dupuis2022}.

\begin{lem}{\label{lem:add2}}
For any $\sigma \in \mathscr{S}$, let $\nu_{\sigma}$ be the probability measure on $[0,1]^2$ induced by $(X,\sigma(X))$, where $X$ is distributed uniformly on $[0,1]$. Let $\nu^*$ be a probability measure on $[0,1]^2$ such that both marginals of $\nu$ are uniform probability measures on $[0,1]$. Then, there exists a sequence $\left\{\sigma_n : n \geq 1\right\} \subseteq \mathscr{S}$ such that $\nu_{\sigma_n}$ converges weakly to $\nu^*$.
\end{lem}

We now proceed to the proof of Lemma~\ref{lem:add}. We start by observing that $f \mapsto \mathscr{I}(f)$ is continuous if $\cls_0$ is equipped with the $L^2$ norm induced from $L^2[0,1]^2$. This can easily be verified by employing a subsequence argument and using the fact that $\Delta$ is bounded and continuous. Moreover, for any $\delta \in (0,1/2)$ and $g \in \cls_0$ satisfying $\delta \leq g \leq 1-\delta$, we have
\begin{equation}
	\rvert \mathscr{I}(f) - \mathscr{I}(g) \rvert \leq \int_{[0,1]^2} \Big \rvert \Delta(f,\gamma^+,\gamma^-) - \Delta(g,\gamma^+,\gamma^-)\Big \rvert \leq K_{\delta} \int_{[0,1]^2} |f-g| \leq K_{\delta} \| f-g \|_2. 
\end{equation}

The continuity of $\mathscr{I}$ with respect to the $L^2$-norm gives us  the following equivalent expression for the function $\widehat{\mathscr{I}}$ :  
\begin{equation}{\label{add:step1}}
	\widehat{\mathscr{I}} \left(\widehat f\right) = \inf_{\sigma \in \mathscr{S}} \mathscr{I}(f^{\sigma}), \; \forall \; f \in \widehat{f}, \;\widehat f \in \widehat{\cls_0}.
\end{equation}	
Indeed, to establish \eqref{add:step1}, we take any $g \in \widehat f$. By~\cite[Corollary 8.14]{Lovasz}, there exists a sequence $\left\{\sigma_n : n \geq 1\right\} \subseteq \mathscr{S}$ such that $f^{\sigma_n} \stackrel{L^2}{\longrightarrow} g$. Therefore,
\begin{align*}
	\mathscr{I}(g) = \lim_{n \to \infty} \mathscr{I}\left(f^{\sigma_n}\right) \geq \inf_{\sigma \in \mathscr{S}} \mathscr{I}(f^{\sigma}).
	\end{align*}
Taking infimum over $g \in \widehat f$ on the first term in the above display, we can conclude that $ \widehat{\mathscr{I}} \left(\widehat f\right) \geq 	\inf_{\sigma \in \mathscr{S}} \mathscr{I}(f^{\sigma}).$ The reverse inequality is obvious since $f^{\sigma} \in \widehat f$ for any $\sigma \in \mathscr{S}$. The representation in \eqref{add:step1} also guarantees that for any $f \in \widehat f, g \in \widehat g$ with $\delta \leq g \leq 1-\delta$, we have 
\begin{align}
	\rvert \widehat{\mathscr{I}} \left( \widehat f \right) - 
    \widehat{\mathscr{I}}\left( \widehat g \right) \rvert = 	\bigg \rvert \inf_{\sigma \in \mathscr{S}} \mathscr{I} \left( f^{\sigma} \right) -  \inf_{\sigma \in \mathscr{S}}  \mathscr{I} \left( g^{\sigma} \right) \bigg \rvert &\leq \sup_{\sigma \in \mathscr{S}} \rvert \mathscr{I} \left(f^{\sigma}\right) - \mathscr{I} \left(g^{\sigma} \right) \rvert \label{addstep:lip} \\
    & \leq K_{\delta}  \sup_{\sigma \in \mathscr{S}} \| f^{\sigma} - g^{\sigma} \|_2 = K_{\delta} \| f - g \|_2. \nonumber
\end{align}

We now proceed to prove that $\widehat{\mathscr{I}}$ is lower semi-continuous. The following proof is an expanded form of the arguments in~\cite[Appendix]{Dupuis2022} with  details provided for reader's convenience. Take any $\widehat f_n \longrightarrow \widehat f$ with respect to the metric $\delta_{\square}$. We need to show
\begin{equation}{\label{add:toshow}}
	 \liminf_{n \to \infty} \widehat{\mathscr{I}} \left(\widehat f_n\right) \geq \widehat{\mathscr{I}} \left( \widehat f \right).
	\end{equation}
Note that, without loss of generality we can assume $d_{\square}(f_n,f) \longrightarrow 0$. Moreover, we claim that it suffices to show \eqref{add:toshow} for $f:[0,1]^2 \to [0,1]$ continuous. To see this claim, suppose that \eqref{add:toshow} holds for all continuous $f$ and now consider an arbitrary $f$. Fix  $\delta \in (0,1/2)$ and let $h \in \cls_0$ be such that $h$ is continuous, $\delta \leq h \leq 1-\delta$ and $\|f-h\|_2 \leq 2\delta$. 
Set, 
$$ h_n := \max \left( \delta, \min\left( 1- \delta, f_n - f +h \right) \right) \in \cls_0, \; \forall \; n \geq 1.$$ The following elementary inequality 
$$ \big \rvert \max (\delta, \min(1-\delta,x)) - \max (\delta, \min(1-\delta,y)) \big \rvert \leq |x-y|$$
yields that $|h_n-h| \leq |f_n-f|$ pointwise and hence $d_{\square}(h_n,h) \to 0$. Since by construction $h$ is continuous, we have 
\begin{equation}{\label{add:toshow2}}
	\liminf_{n \to \infty} \widehat{\mathscr{I}} \left(\widehat h_n\right) \geq \widehat{\mathscr{I}} \left( \widehat h \right).
\end{equation}
On the other hand,
$$ |f_n-h_n| \leq |f_n - \max(\delta,\min(1-\delta,f_n))| + |h_n - \max(\delta,\min(1-\delta,f_n))| \leq \delta + |f-h|,$$
and hence $\|f_n-h_n\|_2 \leq 3\delta$. Applying \eqref{addstep:lip} we conclude that 
\begin{align*}
		\liminf_{n \to \infty} \widehat{\mathscr{I}} \left(\widehat f_n\right) \geq  	\liminf_{n \to \infty} \widehat{\mathscr{I}} \left(\widehat h_n\right) - 3\delta K_{\delta}  \geq 
        \widehat{\mathscr{I}} \left( \widehat h \right) - 3\delta K_{\delta}  \geq \widehat{\mathscr{I}} \left( \widehat f \right) - 5\delta K_{\delta}.
		\end{align*}
Since, $\delta$ is arbitrary and $\delta K_{\delta} \to 0$ as $\delta \to 0$, we have that \eqref{add:toshow} holds, 
proving the claim.

It now remains to show that \eqref{add:toshow} holds for all continuous $f$. By \eqref{add:step1}, there exists $\sigma_n \in \mathscr{S}$ such that $\widehat{\mathscr{I}} \left( \widehat f_n \right) \geq \mathscr{I} \left(f_n^{\sigma_n}\right) - 1/n$, for all $n \geq 1$. Note that,
$$ \mathscr{I}\left(f_n^{\sigma_n}\right) = \int_{[0,1]^2} \Delta \left(f_n^{\sigma_n},\gamma^+,\gamma^-\right) = \int_{[0,1]^2} \Delta \left(f_n,\gamma^+ \circ \sigma_n^{-1},\gamma^- \circ \sigma_n^{-1}\right),$$
where $\gamma^{\pm} \circ \sigma^{-1}(x,y) := \gamma^{\pm}\left(\sigma^{-1}(x),\sigma^{-1}(y)\right)$, for all $(x,y) \in [0,1]^2$ and $\sigma \in \mathscr{S}$. 
It is enough to show the following : For any given subsequence of $\mathbb{N}$, there exists a further subsequence (say $\bar n$) such that  
\begin{equation}{\label{add:toshowmain}}
	\liminf_{\bar n \to \infty}  \int_{[0,1]^2} \Delta \left(f_{\bar n},\gamma^+ \circ \sigma_{\bar n}^{-1},\gamma^- \circ \sigma_{\bar n}^{-1}\right) \geq  \inf_{\sigma \in \sS} \int_{[0,1]^2} \Delta \left(f,\gamma^+ \circ \sigma^{-1},\gamma^- \circ \sigma^{-1}\right).
\end{equation} 
Now fix a subsequence and relabel it as $\{n\}$.
Let $X^*,Y^*$ be independent random variables uniformly distributed on $[0,1]$. Let $\mu_n$ be the law induced on $[0,1]^5$ by 
$$\left(f_n(X^*,Y^*),X^*,\sigma_n^{-1}(X^*),Y^*,\sigma_n^{-1}(Y^*)\right).$$ 
Since, $[0,1]^5$ is compact, the sequence of probability measures $\left\{\mu_n : n \geq 1\right\}$ is tight and hence converges weakly to some probability measure $\mu$ along some subsequence. This subsequence will be our required subsequence and from now on, without loss of generality, we shall assume that the convergence holds along the original subsequence. By \textit{Skorokhod's representation theorem}, we can find $(X_n,Y_n)$, independent random variables uniformly distributed on $[0,1]$, such that
$$ \left(f_n(X_n,Y_n),X_n,\sigma_n^{-1}(X_n),Y_n,\sigma_n^{-1}(Y_n)\right) \stackrel{a.s.}{\longrightarrow} \left(W,X,U,Y,V\right) \sim \mu.$$
Since, $\left\{(X_n,\sigma_n^{-1}(X_n))\right\}_{n \geq 1}$ and $\left\{(Y_n,\sigma_n^{-1}(Y_n))\right\}_{n \geq 1}$ are sequence of independent and identically distributed random variables on $[0,1]^2$ with uniform marginals on $[0,1]$, the same is true for $(X,U)$ and $(Y,V)$.

We now argue that $\mathbb{E} \left(W \mid X,U,Y,V\right) = f(X,Y)$ almost surely. Take any continuous $\varphi_1,\varphi_2 : [0,1]^2 \to [-1,1]$. It is enough to show that 
\begin{equation}{\label{add:step2}}
	\mathbb{E} \left( W \varphi_1(X,U) \varphi_2(Y,V) \right) =   \mathbb{E} \left( f(X,Y) \varphi_1(X,U) \varphi_2(Y,V) \right).
\end{equation}
From dominated convergence,
\begin{equation}{\label{add:step3}}
	\lim_{n \to \infty} \mathbb{E} \left( f_n(X_n,Y_n) \varphi_1 \left(X_n,\sigma_n^{-1}(X_n)\right) \varphi_2 \left(Y_n,\sigma_n^{-1}(Y_n)\right) \right) = \mathbb{E} \left( W \varphi_1(X,U) \varphi_2(Y,V) \right).
\end{equation}
On the other hand, if we define $\psi_{i,n} : [0,1] \to [-1,1]$ as $\psi_{i,n}(x) := \varphi_i\left(x,\sigma_n^{-1}(x)\right)$ for $i=1,2$, then we have 
\begin{equation}{\label{add:step4}}
	 \int_{[0,1]^2} f_n(x,y) \psi_{1,n}(x)\psi_{2,n}(y)\; dx\,dy - \int_{[0,1]^2} f(x,y) \psi_{1,n}(x)\psi_{2,n}(y)\; dx\,dy \longrightarrow 0,
\end{equation}
where we have used the fact that  $d_{\square}(f_n,f) \longrightarrow 0$. It is easy to see that 
\begin{align*}
&\int_{[0,1]^2} f_n(x,y) \psi_{1,n}(x)\psi_{2,n}(y)\; dx\,dy \\
& =  \mathbb{E} \left( f_n(X_n,Y_n) \varphi_1 \left(X_n,\sigma_n^{-1}(X_n)\right) \varphi_2 \left(Y_n,\sigma_n^{-1}(Y_n)\right) \right),
\end{align*}
whereas
\begin{align*}
	\int_{[0,1]^2} f(x,y) \psi_{1,n}(x)\psi_{2,n}(y)\; dx\,dy &=  \mathbb{E} \left( f(X_n,Y_n) \varphi_1 \left(X_n,\sigma_n^{-1}(X_n)\right) \varphi_2 \left(Y_n,\sigma_n^{-1}(Y_n)\right) \right) \\
	&\longrightarrow  \mathbb{E} \left( f(X,Y) \varphi_1 \left(X,U\right) \varphi_2 \left(Y,V\right) \right),
\end{align*}  
where the last convergence used the fact that $f$ is continuous. Combining these observations with \eqref{add:step3} and \eqref{add:step4} yields \eqref{add:step2}. Applying \textit{Jensen's inequality} in the fifth line below, we obtain
\begin{align*}
	&\liminf_{n \to \infty} \mathscr{I} \left(f_n^{\sigma_n}\right) \\
    &= 	\liminf_{ n \to \infty}  \int_{[0,1]^2} \Delta \left(f_{n},\gamma^+ \circ \sigma_{n}^{-1},\gamma^- \circ \sigma_{n}^{-1}\right) \\
	 &= \liminf_{n \to \infty} \E \left[ \Delta \left(f_n(X_n,Y_n),\gamma^+\left(\sigma_n^{-1}(X_n),\sigma_n^{-1}(Y_n)\right),  \gamma^-\left(\sigma_n^{-1}(X_n),\sigma_n^{-1}(Y_n)\right)\right)\right]\\
	 &= \E \left[ \Delta \left(W,\gamma^+\left(U,V\right),  \gamma^-\left(U,V\right)\right)\right] \\
	 & = \E \left[ \E \left(\Delta \left(W,\gamma^+\left(U,V\right),  \gamma^-\left(U,V\right)\right) \big \rvert X,U,Y,V\right)\right] \\
	 & \geq  \E \left[\Delta \left(\E \left(W  \mid  X,U,Y,V \right),\gamma^+\left(U,V\right),  \gamma^-\left(U,V\right)\right)\right] \\
	 &=  \E \left[ \Delta \left(f(X,Y),\gamma^+\left(U,V\right),  \gamma^-\left(U,V\right)\right)\right].
\end{align*}
Now recall that $(X,U),(Y,V) \stackrel{i.i.d.}{\sim} \nu^*$, where $\nu^*$ is a probability measure on $[0,1]^2$ with uniform marginals. Apply Lemma~\ref{lem:add2} to get $\left\{\tilde \sigma_k : k \geq 1\right\} \subseteq \mathscr{S}$ such that $\nu_{\tilde \sigma_k^{-1}}$ converges weakly to $\nu^*$. Therefore,
\begin{align*}
	& \liminf_{n \to \infty} \mathscr{I} \left(f_n^{\sigma_n}\right) \\
    &\geq   \E \left[ \Delta \left(f(X,Y),\gamma^+\left(U,V\right),  \gamma^-\left(U,V\right)\right)\right] \\
	&= \int_{[0,1]^2\times [0,1]^2} \Delta \left(f(x,y),\gamma^+\left(x^{\prime},y^{\prime}\right),\gamma^-\left(x^{\prime},y^{\prime}\right) \right)\, d\nu^*(x,x^{\prime})\, d\nu^*(y,y^{\prime})\\
	&= \lim_{k \to \infty} \int_{[0,1]^2 \times [0,1]^2} \Delta \left(f(x,y),\gamma^+\left(x^{\prime},y^{\prime}\right),\gamma^-\left(x^{\prime},y^{\prime}\right) \right)\, d \nu_{\tilde \sigma_k^{-1}}(x,x^{\prime})\, d \nu_{\tilde \sigma_k^{-1}}(y,y^{\prime})\\
	&= \lim_{k \to \infty} \int_{[0,1]^2} \Delta \left(f(x,y),\gamma^+\left(\tilde \sigma_k^{-1}(x), \tilde \sigma_k^{-1}(y)\right),\gamma^-\left(\tilde \sigma_k^{-1}(x), \tilde \sigma_k^{-1}(y)\right) \right)\, dx \,dy \\
	&= \lim_{k \to \infty} \int_{[0,1]^2} \Delta \left(f,\gamma^+\circ \tilde \sigma_k^{-1},\gamma^- \circ \tilde \sigma_k^{-1} \right)
    \ge \inf_{\sigma \in \sS}\int_{[0,1]^2} \Delta \left(f,\gamma^+\circ \sigma^{-1},\gamma^- \circ  \sigma^{-1} \right),
\end{align*}
where the second equality uses the fact that $\Delta, f$ and $ \gamma^{\pm}$ are all bounded and continuous. This proves \eqref{add:toshowmain} and the result follows. \qed

\subsection{Auxiliary result for variational problem}

\begin{lem}{\label{lem:lem-var-unique}}
	Consider the setting  of Section \ref{sec:lded}. Let $\phi^n \in \cW_0$ satisfies 
	\begin{equation}{\label{eq:cons1}}
		\dfrac{1}{|U_{\epsilon}(t)|} \int_{[0,1]^2 \times U_{\epsilon}(t)} \phi^n \geq p^*+\delta_n, \; \forall \; n \geq 1,
	\end{equation}
	for some $\delta_n >0$ satisfying $\delta_n \to \delta$ 
	and 
	$$ \lim_{n \to \infty} \int_{[0,1]^2_T} \cQ_1 \left( \phi^n, \gamma^+,\gamma^-\right) = \int_{[0,1]^2_T} \cQ_1 \left( \phi^{*,t,\delta}, \gamma^+,\gamma^-\right).$$
	Then $\phi^n$ converges to $\phi^{*,t,\delta}$ in $L^2\left([0,1]^2_T\right)$. 
\end{lem}

\begin{proof}
	Recall that $\cQ_1(u,v^+,v^-) = (\sqrt{v^+(1-u))}-\sqrt{v^-u})^2$ for all $u \in [0,1]$, $v^+,v^- >0$. Hence, 
	$$ \cQ_{1,1}(u,v^+,v^-) = -v^++v^- -\sqrt{v^+v^-}\dfrac{1-2u}{\sqrt{u(1-u)}},\;  \cQ_{1,11}(u,v^+,v^-) = \dfrac{\sqrt{v^+v^-}}{2(u(1-u))^{3/2}},$$
	where $\cQ_{1,1}$ and $\cQ_{1,11}$ are the first and second order derivatives of $\cQ_1$ with respect to its first argument. Note that $\cQ_{1,11} \geq 4c_{\beta}$, whenever $v^+,v^- \geq c_{\beta}$ and thus for any $u,u^* \in (0,1), v^+,v^-\geq c_{\beta}$, 
	\begin{equation}{\label{eq:cons2}}
		 \cQ_1(u,v^+,v^-) \geq \cQ_1(u^*,v^+,v^-) + (u-u^*)\cQ_{1,1}(u^*,v^+,v^-) + 4c_{\beta}(u-u^*)^2.
	\end{equation}
	Now recall $\gamma^{k,\pm}$, the block approximation for $\gamma^{\pm}$. Let $\phi^{*,t,\delta_n,k}$ be the optimal solution to~(\ref{eq:edge-var-prob}) with $\gamma^{\pm}$ replaced by $\gamma^{k,\pm}$ and $\delta$ replaced by $\delta_n$. Since, $\gamma^{k,\pm}$ converges to $\gamma^{\pm}$ in $L^1$, we can conclude from Proposition~\ref{prop:edge-max-eq} that $\phi^{*,t,\delta_n,k}$ converges in $L^2$ to $\phi^{*,t,\delta_n}$, as $k \to \infty$. Note that the constraint in~(\ref{eq:cons1}) is linear and hence, for any $\epsilon \in (0,1)$ and $n, k \geq 1$, 
	$$  \dfrac{1}{\epsilon}\int_{[0,1]^2_T} \left(\cQ_1(\epsilon\phi^n+(1-\epsilon)\phi^{*,t,\delta_n,k},\gamma^{k,+}, \gamma^{k,-}) -  \cQ_1(\phi^{*,t,\delta_n,k},\gamma^{k,+}, \gamma^{k,-}) \right) \geq 0. $$
	Since $\gamma^{k,\pm}$ only takes finitely many values, same is true for $\phi^{*,t,\delta_n,k}$ and it is bounded away from $0$ and $1$. This shows that $|\cQ_{1,1}(\epsilon\phi^n+(1-\epsilon)\phi^{*,t,\delta_n,k},\gamma^{k,+}, \gamma^{k,-})|$ is bounded above for small enough $\epsilon>0$ and thus taking $\epsilon \downarrow 0$ and using dominated convergence we see that the first-order optimality condition holds: 
	$$ \int_{[0,1]^2_T} \cQ_{1,1}(\phi^{*,t,\delta_n,k}, \gamma^{k,+}, \gamma^{k,-}) \left( \phi^n- \phi^{*,t,\delta_n,k}\right) \geq 0.$$
	Using~(\ref{eq:cons2}), we now get 
	\begin{equation*}
		\| \phi^n - \phi^{*,t,\delta_n,k}\|^2_2 \leq \dfrac{1}{4c_{\beta}}\left( \int_{[0,1]^2_T} \cQ_1\left( \phi^n, \gamma^{k,+},\gamma^{k,-}\right) - \int_{[0,1]^2_T} \cQ_1 \left( \phi^{*,t,\delta_n,k}, \gamma^{k,+},\gamma^{k,-}\right)   \right).
	\end{equation*}
	By Lemma~\ref{lem:uniform}, dominated convergence, and the fact that $\phi^{*,t, \delta_n,k}$ converges to $\phi^{*,t,\delta_n}$ as $k \to \infty$, we conclude that 
	\begin{equation*}
	\| \phi^n - \phi^{*,t,\delta_n}\|^2_2 \leq \dfrac{1}{4c_{\beta}}\left( \int_{[0,1]^2_T} \cQ_1\left( \phi^n, \gamma^{+},\gamma^{-}\right) - \int_{[0,1]^2_T} \cQ_1 \left( \phi^{*,t,\delta_n}, \gamma^{+},\gamma^{-}\right)   \right).
	\end{equation*}
	Proposition~\ref{prop:edge-max-eq} guarantees that $\delta \mapsto \phi^{*,t,\delta}$ is continuous with respect to the $L^2$-norm. This completes the proof.
\end{proof}

\section*{Acknowledgments}

Bhamidi was partially supported by NSF DMS-2113662, DMS-2413928, and DMS-2434559. Budhiraja
was partially supported by NSF DMS-2152577.  Bhamidi and Budhiraja were partially funded by NSF RTG grant DMS-2134107. Part of
this material is based upon work supported by the National Science Foundation under Grant
No. DMS-1928930, while Bhamidi was in residence at the Simons Laufer
Mathematical Sciences Institute in Berkeley, California, during the Spring 2025 semester.

\noindent{\bf Conflict of interest statement:} The authors declare that they have no conflict of interest. 

\bibliographystyle{plain}
	\bibliography{dynamical_ldp}

\end{document}